\documentclass[a4paper,10pt,reqno]{amsart}
\usepackage{amssymb}
\usepackage{enumerate}
\usepackage{amscd,eucal}
\usepackage{graphicx}
\newtheorem{theorem}{Theorem}[section]
\newtheorem{lemma}[theorem]{Lemma}
\newtheorem{proposition}[theorem]{Proposition}
\newtheorem{corollary}[theorem]{Corollary}
\newtheorem{theorem-definition}[theorem]{Theorem-Definition}

\theoremstyle{definition}
\newtheorem{definition}[theorem]{Definition}

\newtheorem{notation}[theorem]{Notation}
\newtheorem{problem}[theorem]{Problem}

\newtheorem{example}[theorem]{Example}

\theoremstyle{remark}
\newtheorem{remark}[theorem]{Remark}

\numberwithin{equation}{section}

\allowdisplaybreaks

\def\aa{{\mathfrak{B}}}
\def\bb{{\mathcal{B}}}

\def\cc{{{C}}}

\def\gh{{\mathcal{H}}}

\def\hc{{\mathfrak{H}}}

\def\mk{{\mathfrak{m}}}
\def\mm{{{\widetilde{\mathfrak{m}}}_c}}
\def\m2{{{\widetilde{\mathfrak{m}}}_2}}

\def\pp{{\mathcal{P}}}

\def\rs{{\delta}}

\DeclareFontEncoding{OT2}{}{}
\DeclareFontSubstitution{OT2}{cmr}{m}{n}
\DeclareSymbolFont{cyss}{OT2}{wncyss}{m}{n}
\DeclareMathSymbol{\sh}{\mathbin}{cyss}{`x}

\allowdisplaybreaks
\begin{document}

\baselineskip 16pt 

\title[$p$-adic multiple $L$-functions]{Fundamentals of $p$-adic multiple $L$-functions and 
evaluation of their special values}


\author[H. Furusho, Y. Komori, K. Matsumoto, and H. Tsumura]{Hidekazu Furusho, Yasushi Komori, Kohji Matsumoto, and Hirofumi Tsumura}



\thanks{
Research of the authors
supported by Grants-in-Aid for Science Research (no. 24684001 for HF,
no. 25400026 for YK, no. 25287002 for KM, no. 15K04788 for HT, respectively), JSPS}

\maketitle

\begin{abstract}
We construct $p$-adic multiple $L$-functions in several variables,
which are generalizations of the classical Kubota-Leopoldt $p$-adic
$L$-functions, by using a specific $p$-adic measure.    Our 
construction is from the $p$-adic analytic side of view, and we 
establish various fundamental properties of these functions:\\
\ \ (a) Evaluation at non-positive integers: 
We establish their intimate connection with
the 
complex multiple zeta-functions 
by showing that the special values of the $p$-adic multiple $L$-functions at non-positive integers
are expressed by 
the twisted multiple Bernoulli numbers, 
which are 
the special values of the complex multiple  zeta-functions
at non-positive integers.  \\
\ \ (b) Multiple Kummer congruences: We extend Kummer congruences for  Bernoulli numbers 
to congruences for the twisted multiple Bernoulli numbers.\\
\ \ (c) Functional relations with a parity condition: We extend the vanishing
property of the Kubota-Leopoldt $p$-adic $L$-functions with odd
characters to our $p$-adic multiple $L$-functions. \\
\ \ (d) Evaluation at positive integers: 
We establish their close relationship with  the $p$-adic twisted multiple star polylogarithms
by showing that the special values of the $p$-adic multiple $L$-functions at positive integers
are described by those of the $p$-adic twisted multiple star polylogarithms at roots of
unity,
which generalizes the 
result of Coleman  in the single variable case.
\end{abstract} 

\setcounter{section}{-1}
\section{Introduction} \label{sec-1}

Let $\mathbb{N}$, $\mathbb{N}_0$, $\mathbb{Z}$, $\mathbb{Q}$, $\mathbb{R}$ and $\mathbb{C}$ be the set of natural numbers, non-negative integers, rational integers, rational numbers, real numbers and complex numbers, respectively. 
For $s\in \mathbb{C}$, denote by $\Re s$ and $\Im s$ the real and the imaginary parts of $s$, respectively.

The aim of the present paper is 
to introduce $p$-adic multiple $L$-functions,
multiple analogues of Kubota-Leopoldt $p$-adic $L$-functions, 
and show 
basic properties, particularly 
their 
special values at both positive and non-positive integer points.
The construction of these 
functions 
is motivated 
by the technique of desingularization of 
complex multiple zeta functions and 
a certain evaluation  method of their special values at non-positive integer points, which were 
introduced and developed in
our previous paper \cite{FKMT-Desing}.

In Definition \ref{Def-pMLF}, we construct
the multivariable {\bf $p$-adic multiple $L$-function}, 
\begin{equation}\label{pMLF}
L_{p,r}(s_1,\ldots,s_r;\omega^{k_1},\ldots,\omega^{k_r};\gamma_1,\ldots,\gamma_{r};c).
\end{equation}
It is a $p$-adic valued function in 
$p$-adic variables $s_1,\dots, s_r$,
which is associated with 
each 
$\omega^{k_1},\ldots,\omega^{k_r}$
($\omega$: the Teichm\"{u}ller character, $k_1,\ldots,k_r\in \mathbb{Z}$),
$\gamma_1,\ldots,\gamma_r\in \mathbb{Z}_p$
and $c\in \mathbb{N}_{>1}$ with $(c,p)=1$.
It can be seen 
as a multiple analogue of
the Kubota-Leopoldt $p$-adic $L$-functions (cf. \cite{Iwasawa69, K-L})
and the $p$-adic double $L$-functions (introduced in \cite{KMT-IJNT}), 
as well as 
a $p$-adic analogue of the complex multivariable multiple zeta-function 
\begin{equation}                                                                        
\label{Barnes-Lerch}                                                                    
\zeta_r((s_j);(\xi_j);(\gamma_j)):=                                                     
    \sum_{\substack{m_1=1}}^\infty\cdots \sum_{\substack{m_r=1}}^\infty                 
    \prod_{j=1}^r \xi_j^{m_j}
    (m_1\gamma_1+\cdots+m_j\gamma_j )^{-s_j}.                               
\end{equation}
It is known that $\zeta_r((k_j);(1);(1))$ for $k_j\in \mathbb{N}$ $(1\leqslant j \leqslant r)$ and $k_r\geqslant 2$ is called the multiple zeta value (henceforth MZV). 
It can be constructed by a multiple analogue of 
the $p$-adic $\Gamma$-transform of a $p$-adic measure (see Koblitz \cite{Kob79}).
Another type of $p$-adic multiple zeta-functions, which are $p$-adic
versions of Barnes multiple zeta-functions, was studied by Kashio \cite{Kashio}.

Evaluation of \eqref{pMLF} at integral points is one of the 
main themes of this paper. 
We obtain 
special values of \eqref{pMLF} at non-positive integers as follows:

{\bf Theorem \ref{T-main-1} (Evaluation at non-positive integers).}
{\it
For $n_1,\ldots,n_r\in \mathbb{N}_0$, }
\begin{align*}
& L_{p,r}(-n_1,\ldots,-n_r;\omega^{n_1},\ldots,\omega^{n_r};\gamma_1,\ldots,\gamma_{r};c)
 =\sum_{\xi_1^c=1 \atop \xi_1\not=1}\cdots\sum_{\xi_r^c=1 \atop \xi_r\not=1}\aa((n_j);(\xi_j);(\gamma_j)) \notag\\
&\quad +\sum_{d=1}^{r}\left(-\frac{1}{p}\right)^d \sum_{1\leqslant i_1<\cdots<i_d\leqslant r}\sum_{\rho_{i_1}^p=1}\cdots\sum_{\rho_{i_d}^p=1}\sum_{\xi_1^c=1 \atop \xi_1\not=1}\cdots\sum_{\xi_r^c=1 \atop \xi_r\not=1}\aa((n_j);((\prod_{j\leqslant i_l}\rho_{i_l})^{\gamma_j}\xi_j);(\gamma_j)). \label{Th-main}
\end{align*}

Here, $\aa((n_j);(\xi_j);(\gamma_j))$'s  are the twisted multiple Bernoulli numbers
(Definition \ref{Def-M-Bern}), which were shown to be the special values
of complex multivariable multiple zeta-functions \eqref{Barnes-Lerch}
at non-positive integers in \cite{FKMT-Desing}.
The above equality might be regarded as a $p$-adic and multiple analogue
of the well-known equality $\zeta(-n)=-\frac{B_{n+1}}{n+1}$
of the Riemann zeta function $\zeta(s)$.

On the other hand, 
we get 
the following special values at positive integers:

{\bf Theorem \ref{L-Li theorem-2} (Evaluation at positive integers).}
{\it
For $n_1,\dots,n_r\in \mathbb{N}$,
\begin{align*}
L_{p,r} & (n_1,\dots,n_r;\omega^{-n_1},\dots,\omega^{-n_r};1,\dots,1;c) 
=
\underset{\xi_1\neq 1}{\sum_{\xi_1^c=1}} \cdots
\underset{\xi_r\neq 1}{\sum_{\xi_r^c=1}} 
Li^{(p),\star}_{n_1,\dots,n_r}\left(\frac{\xi_1}{\xi_2},\frac{\xi_2}{\xi_3},\dots,\frac{\xi_{r}}{\xi_{r+1}}\right)   \notag \\
&+\sum_{d=1}^r\left(-\frac{1}{p}\right)^d
\sum_{1\leqslant i_1<\cdots<i_d\leqslant r}\sum_{\rho_{i_1}^p=1}\cdots\sum_{\rho_{i_d}^p=1}
\underset{\xi_1\neq 1}{\sum_{\xi_1^c=1}} \cdots
\underset{\xi_r\neq 1}{\sum_{\xi_r^c=1}} 
Li^{(p),\star}_{n_1,\dots,n_r}\Bigl(\bigl(\frac{\prod_{l=1}^d\rho_{i_l}^{\delta_{i_lj}}\xi_j}{\xi_{j+1}}\bigr)\Bigr).
\end{align*}
}

Here, each term on 
the right-hand side is a special value of
a $p$-adic TMSPL
\footnote{
In this paper,  TMSPL stands for the twisted multiple star polylogarithm.
}
$Li^{(p),\star}_{n_1,\dots,n_r}(\xi_1,\dots,\xi_{r};z)$ 
at $z=1$ with some $\xi_i\in\mu_c$,
which will be 
constructed in Subsection \ref{sec-5-5}, 
by Coleman's $p$-adic integration theory \cite{C}. 
It is 
a $p$-adic analogue of the complex twisted multiple star polylogarithm
$$
Li^\star_{n_1,\dots,n_r}(\xi_1,\dots,\xi_r;z):=
{\underset{0<k_1\leqslant\cdots\leqslant k_{r}}{\sum}}
\frac{\xi_1^{k_1}\cdots \xi_r^{k_r}}{k_1^{n_1}\cdots k_r^{n_r}}z^{k_r}.
$$
Here 'star' means that we add equalities in the running indices of the summation 
of the complex twisted multiple polylogarithm
$$
Li_{n_1,\dots,n_r}(\xi_1,\dots,\xi_r;z):=
{\underset{0<k_1<\cdots<k_{r}}{\sum}}
\frac{\xi_1^{k_1}\cdots \xi_r^{k_r}}{k_1^{n_1}\cdots k_r^{n_r}}z^{k_r},
$$
which is a twisted version of multiple polylogarithm
$Li_{n_1,\dots,n_r}(1,\cdots, 1;z)$. The star version of complex multiple
polylogarithm was already considered by Ohno-Zudilin \cite{OZ} and Ulanskii \cite{Ulan},
but the multiple star polylogarithm in the $p$-adic framework
seems to be first introduced in the present paper.
The equality in Theorem \ref{L-Li theorem-2} 
can be seen as a $p$-adic and multiple analogue
of the equality $\zeta(n)=Li_n(1)$
with the polylogarithm $Li_n(z)=Li^\star_n(1;z)$. 
In the special case $r=1$, Coleman's equation
$L_p(n;\omega^{1-n})=(1-\frac{1}{p^n})Li_n^{(p),\star}(1)$
can be recovered (Example \ref{Coleman's equation}).
In order to prove Theorem \ref{L-Li theorem-2}, we 
introduce the notion of a \textit{$p$-adic rigid TMSPLs}
$\ell^{(p),\star}_{n_1,\dots,n_r}(\xi_1,\dots,\xi_{r};z)$
and 
of a 
\textit{$p$-adic partial TMSPLs}
$\ell^{\equiv (\alpha_1,\dots,\alpha_r),(p),\star}_{n_1,\dots,n_r}(\xi_1,\dots,\xi_{r};z)$
(Definition \ref{def of pMMPL} and  \ref{def of pPMPL} respectively)
as 
intermediary functions.

Up until now, 
we were not aware of any relationship 
between 
$p$-adic multiple polylogarithms given by the first-named author
(see \cite{Fu1})
and 
multiple analogues of $p$-adic $L$-functions given by
the second-, the third- and the fourth-named authors (see
\cite[Remark 4.4]{KMT-IJNT}).
Despite having been investigated 
for totally different reasons, 
the above theorems 
indicate 
a close connection at their special values.

The following multiple Kummer congruences and functional relations 
are byproducts of our evaluations:

{\bf Theorem \ref{Th-Kummer} (Multiple Kummer congruences).}
{\it
For $m_1,\ldots,m_r,n_1,\ldots,n_r\in \mathbb{N}_0$ 
such that 
$m_j\equiv n_j$\ $\text{\rm mod}\ (p-1)p^{l_j-1}$, 
and $l_j\in \mathbb{N}$ $(1\leqslant j\leqslant r)$,}
\begin{align*}
& \sum_{\xi_1^c=1 \atop \xi_1\not=1}\cdots\sum_{\xi_r^c=1 \atop \xi_r\not=1}\aa((m_j);(\xi_j);(\gamma_j)) \notag\\
& +\sum_{d=1}^{r}\left(-\frac{1}{p}\right)^d \sum_{1\leqslant i_1<\cdots<i_d\leqslant r}\sum_{\rho_{i_1}^p=1}\cdots\sum_{\rho_{i_d}^p=1}\sum_{\xi_1^c=1 \atop \xi_1\not=1}\cdots\sum_{\xi_r^c=1 \atop \xi_r\not=1}\aa((m_j);((\prod_{j\leqslant i_l}\rho_{i_l})^{\gamma_j}\xi_j);(\gamma_j))\notag\\
& \equiv \sum_{\xi_1^c=1 \atop \xi_1\not=1}\cdots\sum_{\xi_r^c=1 \atop \xi_r\not=1}\aa((n_j);(\xi_j);(\gamma_j)) \notag\\
& \qquad +\sum_{d=1}^{r}\left(-\frac{1}{p}\right)^d \sum_{1\leqslant i_1<\cdots<i_d\leqslant r}\sum_{\rho_{i_1}^p=1}\cdots\sum_{\rho_{i_d}^p=1}\sum_{\xi_1^c=1 \atop \xi_1\not=1}\cdots\sum_{\xi_r^c=1 \atop \xi_r\not=1}\aa((n_j);((\prod_{j\leqslant i_l}\rho_{i_l})^{\gamma_j}\xi_j);(\gamma_j))\notag\\
& \qquad \qquad \left({\rm mod}\  p^{\min \{l_j\,|\,{1\leqslant j\leqslant r}\}}\right). 
\end{align*}

The above congruence
includes  an ordinary Kummer congruence for Bernoulli numbers as a very special case, 
as well as a possibly new 
type of congruence for Bernoulli number such as
double and triple Kummer congruences (Example \ref{DKC} and \ref{triple-KC}).

{\bf Theorem \ref{T-6-1} (Functional relations).}
{\it
For $r\in \mathbb{N}_{>1}$, let $k_1,\ldots,k_r\in \mathbb{Z}$ 
such that 
$k_1+\cdots+k_r\not\equiv r\pmod 2$, $\gamma_{1}\in \mathbb{Z}_p$, $\gamma_2,\ldots,\gamma_{r}\in p\mathbb{Z}_p$ and $c\in \mathbb{Z}_p^\times$. Then for any 
$s_1,\ldots,s_r\in \mathbb{Z}_p$, }
\begin{align*}
& L_{p,r} (s_1,\ldots,s_r;\omega^{k_1},\ldots,\omega^{k_r};\gamma_1,\gamma_2,\ldots,\gamma_{r};c)\notag\\
& \ =-\sum_{J\subsetneq\{1,\ldots,r\} \atop 1\in J}
    \Bigr(\frac{1-c}{2}\Bigl)^{r-|J|}\notag\\
&   \qquad \times L_{p,|J|}\Bigl(\Bigl(\sum_{j_\mu(J)\leqslant l<j_{\mu+1}(J)}s_l\Bigr)_{\mu=1}^{|J|};\bigl(\omega^{\sum_{j_\mu(J)\leqslant l<j_{\mu+1}(J)}k_l}\bigr)_{\mu=1}^{|J|};(\gamma_{j})_{j\in J};c\Bigr). 
\end{align*}

The functional relations are linear relations among $p$-adic multiple $L$-functions, 
which reduce \eqref{pMLF} as a linear sum of \eqref{pMLF} with lower $r$.
An important feature is that the relations are valid under a 
parity condition of $k_1+\cdots +k_r$. 
It extends  the known fact in the single variable case that 
any Kubota-Leopoldt $p$-adic $L$-function with odd character
is identically zero. 
In the double variable case, we recover the result of \cite{KMT-IJNT}
that a certain $p$-adic double $L$-function  is  equal to 
a Kubota-Leopoldt $p$-adic $L$-function 
up to a minor factor. 

\
\section{$p$-adic multiple $L$-functions}\label{sec-3}

In this section, 
we will construct certain $p$-adic multiple $L$-functions 
(Definition \ref{Def-pMLF}) 
which can be regarded as multiple analogues of Kubota-Leopoldt $p$-adic $L$-functions, 
as well as $p$-adic analogues of complex multiple  $L$- (zeta-)functions.
We then 
investigate 
some of their $p$-adic properties; 
in particular, 
we shall prove $p$-adic analyticity 
with respect to certain variables, 
in Theorem \ref{Th-pMLF}, and 
$p$-adic continuity 
with respect to other variables, 
in Theorem \ref{continuity theorem}.

\subsection{Kubota-Leopoldt $p$-adic $L$-functions}\label{sec-3-1}
Here, we 
first review the fundamentals of Kubota-Leopoldt $p$-adic $L$-functions.

For a prime number $p$, let $\mathbb{Z}_p$, $\mathbb{Q}_p$, $\overline{\mathbb{Q}}_p$, $\mathbb{C}_p$, ${\mathcal O}_{\mathbb{C}_p}$ and ${\frak M}_{\mathbb{C}_p}$
be respectively the set of $p$-adic integers, $p$-adic numbers, the algebraic closure of $\mathbb{Q}_p$, the $p$-adic completion of $\overline{\mathbb{Q}}_p$, the ring of integers of $\mathbb{C}_p$, 
and 
the maximal ideal of ${\mathcal O}_{\mathbb{C}_p}$.
Fixing embeddings $\overline{\mathbb{Q}}\hookrightarrow \mathbb{C}$,\ $\overline{\mathbb{Q}}\hookrightarrow \mathbb{C}_p$, we identify $\overline{\mathbb{Q}}$ simultaneously 
with a subset of $\mathbb{C}$ and $\mathbb{C}_p$. 
Denote by $|\cdot |_p$ the $p$-adic absolute value, and by $\mu_c$ 
the group of $c$\,th roots of unity in $\mathbb{C}_p$ for $c\in\mathbb{N}$. 
Let $\textbf{P}^{1}(\mathbb{C}_p)=\mathbb{C}_p \cup \{\infty\}$ with $|\infty|_p=\infty$. 

Let 
$\mathbb{Z}_p^\times$ be the unit group in $\mathbb{Z}_p$, and define 
\begin{equation*}
W:=
\begin{cases}
\{\pm 1\} & (\text{if\ }p=2)\\
\{ \xi\in \mathbb{Z}_p^\times\,|\,\xi^{p-1}=1\} & (\text{if\ }p\geqslant 3),
\end{cases}
\end{equation*}
\begin{equation}
q:=
\begin{cases}
4 & (\text{if\ }p=2)\\
p & (\text{if\ }p\geqslant 3).
\end{cases}
\label{q-def}
\end{equation}
It is well-known that 
$W$ forms a set of complete representatives of $\left(\mathbb{Z}_p/q\mathbb{Z}_p\right)^\times$, and 
$$\mathbb{Z}_p^\times = W\times (1+q\mathbb{Z}_p),$$
where we denote this correspondence by $x=\omega(x)\cdot\langle x \rangle$ for $x\in \mathbb{Z}_p^\times$ (see \cite[Section\,3]{Iwasawa72}). 
Noting $\omega(m)\equiv m\ ({\rm mod\ }q\mathbb{Z}_p)$ for $(m,q)=1$, we see that 
\begin{equation}
\omega\,:\,\left(\mathbb{Z}/q\mathbb{Z}\right)^{\times} \to W \label{Teich-Char}
\end{equation}
is a primitive Dirichlet character of conductor $q$, which is called the 
{\it Teichm\"uller character}.


We 
recall Koblitz' definition of 
the twisted Bernoulli numbers.

\begin{definition}[{\cite[p.\,456]{Kob79}}]
For any root of unity $\xi$, 
the {\bf twisted Bernoulli numbers} $\{ \aa_n(\xi)\}$ 
are defined by
\begin{equation}
\hc(t;\xi)=\frac{1}{1-\xi e^{t}}=\sum_{n=-1}^\infty \aa_n(\xi)\frac{t^n}{n!}, \label{def-tw-Ber}
\end{equation}
where we formally let $(-1)!=1$.
\end{definition}

\begin{remark}
Koblitz \cite{Kob79}, 
in fact, 
defined the more general twisted Bernoulli numbers $B_{n,\chi,\xi}$ $(n\in \mathbb{N}_{0})$, 
associated with a primitive 
Dirichlet character $\chi$ of conductor $f$ by
$$\sum_{a=0}^{f-1}\frac{\chi(a)\xi^ate^{at}}{\xi^fe^{ft}-1}=\sum_{n=0}^\infty B_{n,\chi,\xi}\frac{t^n}{n!}.$$
Note that $B_{n,\chi}=B_{n,\chi,1}$ is the generalized Bernoulli number (see \cite[Chapter 4]{Wa}). 
The above $\aa_n(\xi)$ corresponds to $B_{n,\chi_0,\xi}$ for the trivial character $\chi_0$. 
\end{remark}

In the case $\xi=1$, we have 
\begin{equation}
\aa_{-1}(1)=-1,\qquad \aa_n(1)=-\frac{B_{n+1}}{n+1}\quad (n\in \mathbb{N}_0). \label{Ber-01}
\end{equation}

We obtain from \eqref{def-tw-Ber} that $\aa_n(\xi)\in \mathbb{Q}(\xi)$. For example, 
\begin{equation}
\begin{split}
& \aa_0(\xi)=\frac{1}{1-\xi},\quad \aa_1(\xi)=\frac{\xi}{(1-\xi)^2},\quad \aa_2(\xi)=\frac{\xi(\xi+1)}{(1-\xi)^3},\\
& \aa_3(\xi)=\frac{\xi(\xi^2+4\xi+1)}{(1-\xi)^4},\quad \aa_4(\xi)=\frac{\xi(\xi^3+11\xi^2+11\xi+1)}{(1-\xi)^5},\ldots
\end{split}
 \label{TBN-exam}
\end{equation}

Let $\mu_k$ be the group of $k$th roots of unity. 
Using the relation
\begin{equation}
\frac{1}{X-1}-\frac{k}{X^{k}-1}=\sum_{\xi\in \mu_k\atop \xi\not=1}\frac{1}{1-\xi X}\qquad (k\in \mathbb{N}_{>1}) \label{log-der}
\end{equation}
for an indeterminate $X$, 
we obtain the following:

\begin{proposition}
Let $c\in \mathbb{N}_{>1}$. For $n\in \mathbb{N}_0$, 
\begin{equation}
\left(1-c^{n+1}\right)\frac{B_{n+1}}{n+1}=\sum_{\xi^c=1\atop \xi\not=1}\aa_n(\xi). \label{2-0-1}
\end{equation}
\end{proposition}

Let us consider 
the following multiple analogues of twisted Bernoulli numbers (see \cite{FKMT-Desing}): 

\begin{definition}\label{Def-M-Bern}
Let $r\in \mathbb{N}$, 
$\gamma_1,\ldots,\gamma_r\in \mathbb{C}$, 
and let $\xi_1,\ldots,\xi_r\in \mathbb{C}^\times \setminus \{1\}$ be roots of unity. 
Set 
\begin{align}
    \mathfrak{H}_r  (( t_j );( \xi_j); ( \gamma_j))&:=
    \prod_{j=1}^{r} \mathfrak{H}(\gamma_j (\sum_{k=j}^r t_k);\xi_j)=\prod_{j=1}^{r} \frac{1}{1-\xi_j \exp\left(\gamma_j \sum_{k=j}^r t_k\right)}\label{Def-Hr}
\end{align}
and define the 
{\bf twisted multiple Bernoulli numbers}
$\{\aa(n_1,\ldots,n_r;( \xi_j);( \gamma_j))\}$
by 
\begin{align}
    \mathfrak{H}_r  (( t_j );( \xi_j); ( \gamma_j))
    =\sum_{n_1=0}^\infty
    \cdots
    \sum_{n_r=0}^\infty
    \aa(n_1,\ldots,n_r;( \xi_j);( \gamma_j))
    \frac{t_1^{n_1}}{n_1!}
    \cdots
    \frac{t_r^{n_r}}{n_r!}.
\label{Fro-def-r}
\end{align}
In the case $r=1$, we have $\aa_n(\xi_1)=\aa(n;\xi_1;1)$. Note that 
$\mathfrak{H}_r  (( t_j );( \xi_j); ( \gamma_j))$ is holomorphic around the origin with respect to the parameters $t_1,\dots, t_r$, 
and hence, the singular part does not appear on the right-hand side of \eqref{Fro-def-r}.

\end{definition}
We immediately obtain the following from \eqref{def-tw-Ber}, \eqref{Def-Hr} and \eqref{Fro-def-r}:

\begin{proposition}\label{prop-M-Bern}
Let $\gamma_1,\ldots,\gamma_r\in \mathbb{C}$ 
and $\xi_1,\ldots,\xi_r\in \mathbb{C}^\times \setminus \{1\}$ be roots of unity. Then 
$\aa(n_1,\ldots,n_r;( \xi_j);( \gamma_j))$ can be expressed as a polynomial in $\{ \aa_{n}(\xi_j)\,|\,1\leqslant j\leqslant r,~{n\geqslant 0}\}$ and $\{\gamma_1,\ldots,\gamma_r\}$ with $\mathbb{Q}$-coefficients, hence, a rational function in $\{\xi_j\}$ and $\{\gamma_j\}$ with $\mathbb{Q}$-coefficients. 
\end{proposition}

\begin{example}\label{Exam-DH}
Substituting \eqref{def-tw-Ber} into \eqref{Def-Hr} 
for the case $r=2$ and $\xi_j\neq 1$ $(j=1,2)$, 
we have
\begin{align*}
& \mathfrak{H}_2(t_1,t_2;\xi_1,\xi_2;\gamma_1,\gamma_2)=\frac{1}{1-\xi_1 \exp\left(\gamma_1(t_1+t_2)\right)}\frac{1}{1-\xi_2 \exp\left(\gamma_2 t_2\right)}\\
    &=\left(\sum_{m=0}^\infty \aa_m(\xi_1)\frac{\gamma_1^m (t_1+t_2)^m}{m!}\right) \left(\sum_{n=0}^\infty \aa_n(\xi_2)\frac{\gamma_2^n t_2^n}{n!}\right)
    \\
    &=\sum_{m=0}^\infty\sum_{n=0}^\infty \aa_m(\xi_1)\aa_n(\xi_2)\left(\sum_{k,j\geqslant 0 \atop k+j=m}\frac{t_1^k t_2^j}{k!j!}\right)\gamma_1^m\gamma_2^n\frac{t_2^n}{n!}.
\end{align*}
Putting $l:=n+j$, we have 
\begin{align*}
\mathfrak{H}_2(t_1,t_2;\xi_1,\xi_2;\gamma_1,\gamma_2)
    &=\sum_{k=0}^\infty\sum_{l=0}^\infty \sum_{j=0}^{l}\binom{l}{j}\aa_{k+j}(\xi_1)\aa_{l-j}(\xi_2)\gamma_1^{k+j}\gamma_2^{l-j}\frac{t_1^k}{k!}\frac{t_2^l}{l!},
\end{align*}
which gives 
\begin{equation}
\begin{split}
\aa(k,l;\xi_1,\xi_2;\gamma_1,\gamma_2)&=\sum_{j=0}^{l}\binom{l}{j}\aa_{k+j}(\xi_1)\aa_{l-j}(\xi_2)\gamma_1^{k+j}\gamma_2^{l-j}\quad (k,l\in \mathbb{N}_0).
\end{split}
\label{Eur-exp1}
\end{equation}
\end{example}
\begin{definition}\label{define-tilde-H}
For $c\in \mathbb{R}$ and 
$\gamma_1,\ldots,\gamma_r\in \mathbb{C}$ with $\Re \gamma_j >0 \ (1\leqslant j\leqslant r)$, define
\begin{align}
\widetilde{\mathfrak{H}}_r  (( t_j ); ( \gamma_j);c)
& =\prod_{j=1}^{r} \left( \frac{1}{\exp\left(\gamma_j \sum_{k=j}^r t_k\right)-1}-\frac{c}{\exp\left(c\gamma_j \sum_{k=j}^r t_k\right)-1}\right)\notag\\
& =\prod_{j=1}^{r} \left(\sum_{m=1}^\infty  \left(1-c^m\right)B_m\frac{\left(\gamma_j \sum_{k=j}^r t_k\right)^{m-1}}{m!}\right).\label{def-tilde-H}
\end{align}
In particular, when $c\in \mathbb{N}_{>1}$, by use of \eqref{log-der}, we have
\begin{align}
\widetilde{\mathfrak{H}}_r  (( t_j ); ( \gamma_j);c)&=\prod_{j=1}^{r} \sum_{\xi_j^c=1 \atop \xi_j\not=1}\frac{1}{1-\xi_j \exp\left(\gamma_j \sum_{k=j}^r t_k\right)}\notag\\
& =\sum_{\xi_1^c=1 \atop \xi_1\not=1}\cdots \sum_{\xi_r^c=1 \atop \xi_r\not=1}\mathfrak{H}_r  (( t_j );( \xi_j); ( \gamma_j)).\label{tilde-H}
\end{align}
\end{definition}

Now we recall the Kubota-Leopoldt $p$-adic $L$-functions which are defined by use of certain $p$-adic integrals with respect to 
the $p$-adic measure (for the details, see \cite[Chapter 2]{Kob}, 
\cite{Kob79},~\cite[Chapter 5]{Wa}). 
Following Koblitz \cite[Chapter 2]{Kob}, 
we consider the $p$-adic measure $\mk_z$ 
on $\mathbb{Z}_p$ for
$z\in\textbf{P}^{1}(\mathbb{C}_p)$ with $|z-1|_p\geqslant 1$  by
\begin{align}
\mk_z\left(j+p^N\mathbb{Z}_p\right)&:=\frac{z^{j}}{1-z^{p^N}}
\quad\in {\mathcal O}_{\mathbb{C}_p}
\quad (0\leqslant j\leqslant p^N-1). \label{Koblitz-measure}
\end{align}
Corresponding to this measure, 
we can define the $p$-adic integral 
\begin{equation}
\int_{\mathbb{Z}_p} f(x)d{\mathfrak{m}}_z(x)  :=\lim_{N\to \infty}\sum_{a=0}^{p^N-1} f(a) \mk_z\left(a+p^N\mathbb{Z}_p\right)
\quad\in  \mathbb{C}_p   \label{p-adic integ}
\end{equation}
for any continuous function $f\,:\,\mathbb{Z}_p \to \mathbb{C}_p$. 
The $p$-adic integral over a compact open subset $U$ of $\mathbb{Z}_p$ is defined by
$$
\int_{U} f(x)d{\mathfrak{m}}_z(x)  :=
\int_{\mathbb{Z}_p} f(x)\chi_U(x)d{\mathfrak{m}}_z(x) \quad\in \mathbb{C}_p
$$
where $\chi_U(x)$ is the characteristic function of $U$.

For an arbitrary $t\in \mathbb{C}_p$ with $|t|_p<p^{-1/(p-1)}$, we see that $e^{tx}$ is a continuous function in $x\in \mathbb{Z}_p$. We obtain the following proposition by regarding \eqref{def-tw-Ber} as the Maclaurin expansion of $\hc(t;\xi)$ in $\mathbb{C}_p$ and using Koblitz' results \cite[Chapter 2,\ Equations (1.4), (2.4) and (2.5)]{Kob}:

\begin{proposition}\label{Kob-Ch2}
For any primitive root of unity $\xi\in \mathbb{C}_p^\times \setminus \{1\}$ of order prime to $p$,
\begin{equation}
\hc(t;\xi)=\int_{\mathbb{Z}_p}e^{tx}d\mk_\xi(x)
\quad\in \mathbb{C}_p[[t]], \label{gene-01}
\end{equation}
and so 
\begin{equation}
\int_{\mathbb{Z}_p}x^n d\mk_\xi(x)=\aa_n(\xi)
\qquad (n\in \mathbb{N}_0). \label{gene-02}
\end{equation}
\end{proposition}

For any $c\in \mathbb{N}_{>1}$ with $(c,p)=1$,
we consider
\begin{align}
\widetilde{\mathfrak{m}}_c\left(j+p^N\mathbb{Z}_p\right)&:=\sum_{\xi \in \mathbb{C}_p\setminus \{1\} \atop \xi^{c}=1}\mk_\xi\left(j+p^N\mathbb{Z}_p\right), \label{KM-measure}
\end{align}
which is a $p$-adic measure on $\mathbb{Z}_p$ introduced by Mazur 
\cite[Chapter 2]{Kob}. 
Considering the Galois action of ${\rm Gal}(\overline{\mathbb{Q}}_p/{\mathbb{Q}}_p)$, we see that $\widetilde{\mathfrak{m}}_c$ is a $\mathbb{Q}_p$-valued measure, 
and hence, a $\mathbb{Z}_p$-valued measure.

Note that 
\begin{equation}
\begin{split}
\mm\left(pj+p^N\mathbb{Z}_p\right)&=\sum_{\xi \in \mathbb{C}_p\setminus \{1\} \atop \xi^{c}=1}\frac{\xi^{pj}}{1-\xi^{p^N}}\\
&=\sum_{\rho \in \mathbb{C}_p\setminus \{1\} \atop \rho^{c}=1}\frac{\rho^{j}}{1-\rho^{p^{N-1}}}=\mm\left(j+p^{N-1}\mathbb{Z}_p\right),
\end{split}
\label{2-0}
\end{equation}
because $(c,p)=1$. 
From \eqref{2-0-1} and \eqref{gene-02}, we have the following. 

\begin{lemma}[{cf. \cite[Chapter 2,\,(2.6)]{Kob}}]
For any $c\in \mathbb{N}_{>1}$ with $(c,p)=1$,
\begin{align}
\int_{\mathbb{Z}_p} x^{n}d\mm(x) &=\left(c^{n+1}-1\right)\zeta(-n)=\left(1-c^{n+1}\right)\frac{B_{n+1}}{n+1}\in \mathbb{Q}\quad  (n\in \mathbb{N}), \label{2-1}
\end{align}
and
\begin{equation}
\int_{\mathbb{Z}_p} 1\,d\mm(x) =\frac{c-1}{2}\in \mathbb{Z}_p. \label{2-1-2}
\end{equation}
\end{lemma}
Note that $c$ is odd when $p=2$.

Now we recall the definition of {\bf Kubota-Leopoldt $p$-adic $L$-functions} as follows:

\begin{definition}[{\cite[Chapter 2, (4.5)]{Kob}}] \label{Def-Kubota-L}
For $k\in \mathbb{Z}$ and  $c\in \mathbb{N}_{>1}$ with $(c,p)=1$,
\begin{equation*}
\begin{split}
L_p(s;\omega^{k})& :=\frac{1}{\langle c \rangle^{1-s}\omega^{k}(c)-1}\int_{\mathbb{Z}_p^\times}\langle x \rangle^{-s}\omega^{k-1}(x)d\mm(x) 
\end{split}
\end{equation*}
for $s\in \mathbb{C}_p$ with $|s|_p<qp^{-1/(p-1)}$, which is a $\mathbb{C}_p$-valued function.
\end{definition}

\begin{remark}
We note that $\langle x \rangle^{-s}$ can be defined as an ${\mathcal O}_{\mathbb{C}_p}$-valued rigid-analytic function (cf. Notation \ref{rigid-basics}, below) in $s\in \mathbb{C}_p$ with $|s|_p<qp^{-1/(p-1)}$ (see \cite[p.\,54]{Wa}). 
\end{remark}

\begin{proposition}[{\cite[Chapter 2, Theorem]{Kob}, \cite[Theorem 5.11 and Corollary 12.5]{Wa}}] 
For $k\in \mathbb{Z}$, $L_p(s;\omega^{k})$ is rigid-analytic in the sense of Notation 
\ref{rigid-basics}, below (except for a pole at $s=1$ when $k \equiv 0$ mod $p-1$), and satisfies 
\begin{align}
L_p(1-m;\omega^k)& =\left(1-\omega^{k-m}(p)p^{m-1}\right)L\left(1-m,\omega^{k-m}\right)\notag\\
& =-\left(1-\omega^{k-m}(p)p^{m-1}\right)\frac{B_{m,\omega^{k-m}}}{m} \quad (m\in \mathbb{N}). \label{p-L-vals}
\end{align}
\end{proposition}

We know that $B_{m,\omega^{k-m}}=0$ $(m\in \mathbb{N})$ when $k$ is odd. Hence we obtain the following.

\begin{proposition}[{\cite[p.57, Remark]{Wa}}] \label{Prop-zero}
When $k$ is any odd integer, $L_p(s;\omega^k)$ is the zero-function.
\end{proposition}

\begin{remark}\label{Rem-zero-ft}
The above fact will be generalized to the multiple case as functional relations among the $p$-adic multiple $L$-functions defined in the following section (see Theorem \ref{T-6-1} and Example \ref{Rem-zero-2}).
\end{remark}

\subsection{Construction and properties of $p$-adic multiple $L$-functions}\label{sec-3-2}

As a double analogue of $L_p(s;\omega^k)$, the second-, the third- and the fourth-named authors defined a certain $p$-adic double $L$-function associated with the {Teichm\"uller character} for any odd prime $p$ in \cite{KMT-IJNT}. 
In this section, 
we will slightly generalize this by introducing 
a $p$-adic multiple $L$-function 
${L_{p,r}(s_1,\ldots,s_r;\omega^{k_1},\ldots,\omega^{k_r};\gamma_1,\ldots,\gamma_{r};c)}$
(Definition \ref{Def-pMLF}) 
and 
establishing its $p$-adic rigid analyticity with respect to the parameters 
$s_1,\ldots,s_r$ (Theorem \ref{Th-pMLF}), 
and 
its $p$-adic continuity with respect to  the parameter 
$c$ (Theorem \ref{continuity theorem}), 
which yields a non-trivial $p$-adic property of
its special values (Corollary \ref{non-trivial property}).
At the end of this subsection, 
we 
discuss the construction of
a $p$-adic analogue of the entire zeta-function
$\zeta^{\rm des}_r(s_1,\ldots,s_r;\gamma_1,\ldots,\gamma_{r})$
constructed in \cite{FKMT-Desing}.

Let $r\in \mathbb{N}$ and $\gamma_1,\ldots,\gamma_r\in \mathbb{Z}_p$. 
Let $\xi_1,\ldots,\xi_r\in \mathbb{C}_p$ be roots of unity other than $1$
whose orders are prime to $p$.

By combining \eqref{gene-01} and \eqref{Def-Hr}, we then have
\begin{align*}
    & \mathfrak{H}_r  (( t_j ),( \xi_j), ( \gamma_j))=
    \prod_{j=1}^{r} \int_{\mathbb{Z}_p}e^{\gamma_j (\sum_{k=j}^r t_k)x_j}d\mk_{\xi_j}(x_j)\\
    & =\int_{\mathbb{Z}_p^r} \exp\left( \sum_{k=1}^{r} \left(\sum_{j=1}^{k}x_j\gamma_j\right)t_k\right)\prod_{j=1}^{r}d\mk_{\xi_j}(x_j)
    \\
    & =\sum_{n_1,\ldots,n_r=0}^\infty \int_{\mathbb{Z}_p^r} \prod_{k=1}^{r}\left( \sum_{j=1}^{k}x_j\gamma_j \right)^{n_k}\prod_{j=1}^r d\mk_{\xi_j}(x_j)\frac{t_1^{n_1}\cdots t_r^{n_r}}{n_1!\cdots n_r!}.
\end{align*}
Hence, as a multiple analogue of Proposition \ref{Kob-Ch2}, 
we have the following:

\begin{proposition}\label{T-multi-2}
For $n_1,\ldots,n_r\in \mathbb{N}_0$, $\gamma_1,\ldots,\gamma_r\in \mathbb{Z}_p$ and
roots $\xi_1,\ldots,\xi_r\in\mathbb{C}_p^\times\setminus\{1\}$ of unity
whose orders are prime to $p$,
\begin{equation}
\int_{\mathbb{Z}_p^r} \prod_{k=1}^{r}( \sum_{j=1}^{k}x_j\gamma_j )^{n_k}\prod_{j=1}^r d\mk_{\xi_j}(x_j)=\aa((n_j);( \xi_j);( \gamma_j)). \label{multi-val2}
\end{equation}
\end{proposition}

We set
\begin{align}
\left( \mathbb{Z}_p^r\right)'_{\{\gamma_j\}}:=\bigg\{ (x_j)\in \mathbb{Z}_p^r& \,\bigg|\,p\nmid x_1\gamma_1,\ p\nmid (x_1\gamma_1+x_2\gamma_2),\ldots,\ p\nmid \sum_{j=1}^{r}x_j\gamma_j\,\bigg\}, \label{region}
\end{align}
and 
\label{pMLF-region}
\begin{equation}
\mathfrak{X}_r(d):=\left\{(s_1,\ldots,s_r)\in \mathbb{C}_p^r\,\big|\,|s_j|_p<d^{-1}p^{-1/(p-1)}\ (1\leqslant j\leqslant r)\right\}\label{region-d}
\end{equation}
for $d\in \mathbb{R}_{>0}$.

\begin{definition} \label{Def-pMLF}
For $r\in \mathbb{N}$, $k_1,\ldots,k_r,\in \mathbb{Z}$,
$\gamma_1,\ldots,\gamma_r\in \mathbb{Z}_p$
and $c\in \mathbb{N}_{>1}$ with $(c,p)=1$,
the associated {\bf $p$-adic multiple $L$-function} of depth $r$ is the $\mathbb{C}_p$-valued function for 
$(s_j)\in \mathfrak{X}_r\left(q^{-1}\right)$  defined by
\begin{equation}
\begin{split}
&{L_{p,r}(s_1,\ldots,s_r;\omega^{k_1},\ldots,\omega^{k_r};\gamma_1,\ldots,\gamma_{r};c)}\\
&\quad :=\int_{\left( \mathbb{Z}_p^r\right)'_{\{\gamma_j\}}}\langle x_1\gamma_1 \rangle^{-{s_1}}\langle x_1\gamma_1+ x_2\gamma_2 \rangle^{-{s_2}}\cdots \langle \sum_{j=1}^{r}x_{j}\gamma_{j} \rangle^{-{s_r}}\\
& \qquad \qquad \times \omega^{k_1}(x_1\gamma_1)\cdots\omega^{k_r}( \sum_{j=1}^{r}x_{j}\gamma_{j}) \prod_{j=1}^{r}d\mm(x_j).
\end{split}
\label{e-6-1}
\end{equation}
\end{definition}

\begin{remark}\label{Rem-gamma1}
When $\gamma_1\in p\mathbb{Z}_p$, we see that $\left( \mathbb{Z}_p^r\right)'_{\{\gamma_j\}}$ is an empty set, so 
we interpret 
${L_{p,r}((s_j);(\omega^{k_j});(\gamma_j);c)}$ as the zero-function.
\end{remark}

\begin{remark}
Note that we can define, more generally, 
\begin{align*}
{L_{p,r}(s_1,\ldots,s_r;(\omega^{k_j});(\gamma_j);(c_j))}
&:=\int_{\left( \mathbb{Z}_p^r\right)'_{\{\gamma_j\}}}\langle x_1\gamma_1 \rangle^{-{s_1}}\langle x_1\gamma_1+ x_2\gamma_2 \rangle^{-{s_2}}\cdots \langle \sum_{j=1}^{r}x_{j}\gamma_{j} \rangle^{-{s_r}}\\
& \qquad \qquad \times \omega^{k_1}(x_1\gamma_1)\cdots\omega^{k_r}( \sum_{j=1}^{r}x_{j}\gamma_{j}) \prod_{j=1}^{r}d\widetilde{\mathfrak{m}}_{c_j}(x_j)
\end{align*}
for $c_j\in \mathbb{N}_{>1}$ with $(c_j,p)=1$ $(1\leqslant j \leqslant r)$. 
Then we can naturally generalize the 
arguments in the remaining sections to this generalized case.
\end{remark}

In the next section, 
we will see that the above $p$-adic multiple $L$-function can
be seen as a $p$-adic interpolation of
a certain finite sum \eqref{Int-P} of complex multiple zeta functions
(cf. Theorem \ref{T-main-1} and Remark \ref{rem-intpln}).

The following example shows that the Kubota-Leopoldt $p$-adic $L$-functions 
are essentially the same as 
our $p$-adic multiple $L$-functions with $r=1$.

\begin{example}\label{example for r=1}
For $s\in\mathbb{C}_p$ with $|s|_p<qp^{-1/(p-1)}$, $\gamma_1\in \mathbb{Z}_p^\times $, $k\in \mathbb Z$
and $c\in \mathbb{N}_{>1}$ with $(c,p)=1$,
we have
\begin{equation}
\begin{split}
L_{p,1}(s;\omega^{k-1};\gamma_1;c)& =\int_{\mathbb{Z}_p^\times}\langle x\gamma_1\rangle^{-{s}} \omega^{k-1}(x\gamma_1)d\mm(x)\\
& =\langle \gamma_1\rangle^{-s} \omega^{k-1}(\gamma_1) (\langle c\rangle^{1-s}\omega^k(c)-1)\cdot L_p(s;\omega^k). 
\end{split}
\label{1-p-LF-gamma}
\end{equation}
\end{example}

The next example shows that 
when $r=2$, we recover 
the $p$-adic double $L$-functions 
introduced in \cite{KMT-IJNT}.

\begin{example}\label{example for r=2}
Let $p$ be an odd prime, $r=2$, $c=2$ and $\eta\in p\mathbb{Z}_p$. 
Then, for $s_1,s_2\in\mathbb{C}_p$ with $|s_j|_p<p^{1-1/(p-1)}$ $(j=1,2)$ and $k_1,k_2\in \mathbb Z$,
our $p$-adic double $L$-function is given by
\begin{equation}
\begin{split}
& L_{p,2}(s_1,s_2;\omega^{k_1},\omega^{k_2};1,\eta;2)\\
& \quad =\int_{\mathbb{Z}_p^\times\times \mathbb{Z}_p}\langle x_1 \rangle^{-{s_1}} \langle x_1+x_2\eta \rangle^{-{s_2}} \omega^{k_1}(x_1)\omega^{k_2}(x_1+x_2\eta)d\m2(x_1)d\m2(x_2).
\end{split}
\label{double-p-LF-eta}
\end{equation}
\end{example}

The next theorem asserts 
that our $p$-adic multiple $L$-functions are rigid-analytic
(cf. Notation \ref{rigid-basics}, below)
with respect to the parameters $s_1,\ldots,s_r$:

\begin{theorem}\label{Th-pMLF}
Let $k_1,\ldots,k_r\in \mathbb{Z}$,
$\gamma_1,\ldots,\gamma_r\in \mathbb{Z}_p$, 
and $c\in \mathbb{N}_{>1}$ with $(c,p)=1$.
Then 
$L_{p,r}((s_j);(\omega^{k_j});(\gamma_j);c)$
has the following expansion
\begin{align*}
&{L_{p,r}(s_1,\ldots,s_r;\omega^{k_1},\ldots,\omega^{k_r};\gamma_1,\ldots,\gamma_{r};c)}\\
&=\sum_{n_1,\ldots,n_r=0}^\infty \cc\left(n_1,\ldots,n_r;(\omega^{k_j});(\gamma_j);c\right)s_1^{n_1}\cdots s_r^{n_r}
\end{align*}
for $(s_j)\in \mathfrak{X}_r\left(q^{-1}\right)$, where $\cc\left((n_j);(\omega^{k_j});(\gamma_j);c\right)\in \mathbb{Z}_p$ satisfies 
\begin{equation}
\left| \cc\left(n_1,\ldots,n_r;(\omega^{k_j});(\gamma_j);c\right)\right|_p\leqslant \left(qp^{-1/(p-1)}\right)^{-n_1-\cdots -n_r}. \label{C-esti}
\end{equation}
\end{theorem}

In order to prove this theorem, we show 
the following two lemmas, 
which are generalizations of \cite[Proposition 5.8 and its associated lemma]{Wa}:

\begin{lemma}\label{multi-lemma1}
Let $f$ be a continuous function from $\mathbb{Z}_p^r$ to $\mathbb{Q}_p$ defined by 
$$f(X_1,\ldots,X_r)=\sum_{n_1,\ldots,n_r=0}^\infty a(n_1,\ldots,n_r)\prod_{j=1}^{r}\binom{X_j}{n_j},$$
where $a(n_1,\ldots,n_r)\in \mathbb{Z}_p$. Suppose there exist $d,M\in \mathbb{R}_{>0}$ with $d<p^{-1/(p-1)}<1$ such that $|a(n_1,\ldots,n_r)|_p\leqslant Md^{\sum_{j=1}^r n_j}$ for any $(n_j)\in \mathbb{N}_0^r$. Then $f(X_1,\ldots,X_r)$ may be expressed as 
$$f(X_1,\ldots,X_r)=\sum_{n_1,\ldots,n_r=0}^\infty C(n_1,\ldots,n_r)X_1^{n_1}\cdots X_r^{n_r}\in \mathbb{Q}_p[[X_1,\ldots,X_r]],$$
which converges absolutely in $\mathfrak{X}_r(d)$, 
whenever 
$$\left|C(n_1,\ldots,n_r)\right|_p\leqslant M(d^{-1}p^{-1/(p-1)})^{-n_1-\cdots -n_r}.$$
\end{lemma}

\begin{lemma}\label{multi-lemma2}
Let 
$$P_l(X_1,\ldots,X_r)=\sum_{n_1,\ldots,n_r\geqslant 0}a(n_1,\ldots,n_r;l)\prod_{j=1}^{r}X_j^{n_j}\ \ (l\in \mathbb{N}_0)$$
be a sequence of power series with $\mathbb{C}_p$-coefficients which converges 
on a fixed subset $D$ of $\mathbb{C}_p^r$, and suppose that: 
\begin{enumerate}[{\rm (i)}]
\item when $l\to \infty$, $a(n_1,\ldots,n_r;l)\to a(n_1,\ldots,n_r;0)$ for each $(n_j)\in \mathbb{N}_0^r$,
\item for each $(X_j)\in D$ and any $\varepsilon>0$, there exists an $N_0=N_0((X_j),\varepsilon)$ such that 
$$\left|\sum_{(n_j)\in \mathbb{N}_0^r \atop \sum n_j >N_0}a(n_1,\ldots,n_r;l)\prod_{j=1}^{r}X_j^{n_j}\right|_p<\varepsilon$$
uniformly in $l\in \mathbb{N}$. 
\end{enumerate}
Then $P_l((X_j))\to P_0((X_j))$ as $l\to \infty$ for any $(X_j)\in D$.
\end{lemma}

\begin{proof}[Proof of Lemma \ref{multi-lemma2}] 
For any $\varepsilon$ and $(X_j)$, we can choose $N_0$ as above. Then 
$$|P_l((X_j))- P_0((X_j))|_p\leqslant \max_{\sum n_j\leqslant N_0}\left\{\varepsilon, \left|a(n_1,\ldots,n_r;0)-a(n_1,\ldots,n_r;l)\right|_p\cdot \prod_{j=1}^{r}|X_j|_p^{n_j}\right\}\leqslant \varepsilon$$
for any sufficiently large $l$. 
\end{proof}

\begin{proof}[Proof of Lemma \ref{multi-lemma1}] 
We can write
$$\binom{X}{n}=\frac{1}{n!}\sum_{m=0}^n b(n,m)X^m\qquad (n\in \mathbb{N}_0),$$
where $b(n,m)\in \mathbb{Z}$. For any $l\in \mathbb{N}$, let 
\begin{align*}
\pp_l(X_1,\ldots,X_r)&=\sum_{n_1,\ldots,n_r\geqslant 0 \atop n_1+\cdots +n_r\leqslant l}a(n_1,\ldots,n_r)\prod_{j=1}^{r}\binom{X_j}{{n_j}}\\
& =\sum_{0\leqslant m_1,\ldots,m_r\leqslant l \atop m_1+\cdots +m_r\leqslant l}C(m_1,\ldots,m_r;l)\prod_{j=1}^{r}X_j^{m_j},
\end{align*}
say. We 
show that $\pp_l(X_1,\ldots,X_r)$ satisfies the conditions (i) and (ii) for $P_l(X_1,\ldots,X_r)$ in Lemma \ref{multi-lemma2}. 
In fact, using the notation 
\begin{equation*}
\lambda_n(m)=
\begin{cases}
1 & (m\leqslant n)\\
0 & (m>n),
\end{cases}
\end{equation*}
we have
\begin{align*}
\pp_l(X_1,\ldots,X_r)&=\sum_{n_1+\cdots+n_r\leqslant l}\frac{a(n_1,\ldots,n_r)}{n_1!\cdots n_r!}\sum_{m_1\leqslant n_1}\cdots \sum_{m_r\leqslant n_r}\prod_{j=1}^{r}b(n_j,m_j)X_j^{m_j}\\
&=\sum_{n_1+\cdots+n_r\leqslant l}\frac{a(n_1,\ldots,n_r)}{n_1!\cdots n_r!}\sum_{m_1\leqslant l}\cdots \sum_{m_r\leqslant l \atop m_1+\cdots+m_r \leqslant l}\prod_{j=1}^{r}b(n_j,m_j)\lambda_{n_j}(m_j)X_j^{m_j}\\
&=\sum_{m_1,\ldots,m_r\leqslant l \atop m_1+\cdots +m_r\leqslant l}\left(\sum_{m_1\leqslant n_1,\ldots,m_r\leqslant n_r \atop n_1+\cdots+n_r\leqslant l}\frac{a(n_1,\ldots,n_r)}{n_1!\cdots n_r!}\prod_{j=1}^{r}b(n_j,m_j)\right)\prod_{j=1}^{r}X_j^{m_j}.
\end{align*}
Hence, we have
$$C(m_1,\ldots,m_r;l)=\sum_{m_1\leqslant n_1,\ldots,m_r\leqslant n_r \atop n_1+\cdots+n_r\leqslant l}\frac{a(n_1,\ldots,n_r)}{n_1!\cdots n_r!}\prod_{j=1}^{r}b(n_j,m_j).$$
Therefore, noting $b(n_j,m_j)\in \mathbb{Z}$ and using $|n!|_p>p^{-n/(p-1)}$, we obtain
\begin{align}
|C(m_1,\ldots,m_r;l)|_p& \leqslant \max_{m_1\leqslant n_1,\ldots,m_r\leqslant n_r \atop n_1+\cdots+n_r\leqslant l}\frac{Md^{\sum n_j}}{|n_1!\cdots n_r!|_p}\notag\\
& \leqslant \max_{m_1\leqslant n_1,\ldots,m_r\leqslant n_r \atop n_1+\cdots+n_r\leqslant l}M\left(dp^{1/(p-1)}\right)^{\sum n_j}\leqslant M\left(d^{-1}p^{-1/(p-1)}\right)^{-m_1-\cdots -m_r}\label{5-4-1}.
\end{align}
Furthermore, we have
\begin{align*}
& |C(m_1,\ldots,m_r;l+k)-C(m_1,\ldots,m_r;l)|_p\\
& \leqslant \left|\sum_{m_1\leqslant n_1,\ldots,m_r\leqslant n_r \atop l<n_1+\cdots+n_r\leqslant l+k}\frac{a(n_1,\ldots,n_r)}{n_1!\cdots n_r!}\prod_{j=1}^{r}b(n_j,m_j)\right|_p\\
& \leqslant \max_{m_1\leqslant n_1,\ldots,m_r\leqslant n_r \atop l<n_1+\cdots+n_r\leqslant l+k}\frac{Md^{\sum n_j}}{|n_1!\cdots n_r!|_p}\notag\\
& \leqslant \max_{m_1\leqslant n_1,\ldots,m_r\leqslant n_r \atop l<n_1+\cdots+n_r\leqslant l+k}M\left(dp^{1/(p-1)}\right)^{\sum n_j}\leqslant M\left(d^{-1}p^{-1/(p-1)}\right)^{-l-1}
\end{align*}
for any 
$l\in \mathbb{N}$. 
Hence, $\{C(m_1,\ldots,m_r;l)\}$ is a Cauchy sequence in $l$, 
so there exists
$$C(m_1,\ldots,m_r;0)=\lim_{l\to\infty} C(m_1,\ldots,m_r;l)\in \mathbb{Q}_p\quad ((m_j)\in \mathbb{N}_0^r)$$
with $|C(m_1,\ldots,m_r;0)|_p\leqslant M(d^{-1}p^{-1/(p-1)})^{-m_1-\cdots -m_r}$. 
For $(X_j)\in \mathfrak{X}_r(d)$ defined by \eqref{region-d}, let
$$\pp_0(X_1,\ldots,X_r)=\sum_{n_1,\ldots,n_r\geqslant 0}C(n_1,\ldots,n_r;0)X_1^{n_1}\cdots X_r^{n_r}.$$
Then $\pp_0(X_1,\ldots,X_r)$ converges absolutely in $\mathfrak{X}_r(d)$. Moreover, by \eqref{5-4-1}, 
we have
\begin{align*}
& \left|\sum_{n_1,\ldots,n_r\geqslant 0 \atop n_1+\cdots +n_r\geqslant N_0}C(n_1,\ldots,n_r;l)X_1^{n_1}\cdots X_r^{n_r}\right|_p \\
& \quad \leqslant \max_{n_1+\cdots +n_r\geqslant N_0}\left\{M\prod_{j=1}^{r}\left(d^{-1}p^{-1/(p-1)}\right)^{-n_j}|X_j|_p^{n_j}\right\} \to 0\quad (n_1,\ldots,n_r\to \infty)
\end{align*}
uniformly in $l$, for $(X_j)\in \mathfrak{X}_r(d)$. Therefore, by Lemma \ref{multi-lemma2}, we obtain 
$$f(X_1,\ldots,X_r)=\lim_{l\to \infty}\pp_l(X_1,\ldots,X_r)=\pp_0(X_1,\ldots,X_r)$$
for $(X_j)\in \mathfrak{X}_r(d)$, 
which completes the proof. 
\end{proof}

\begin{proof}[Proof of Theorem \ref{Th-pMLF}] 
Considering the binomial expansion in \eqref{e-6-1}, we have, for $(s_j)\in \mathfrak{X}_r=\mathfrak{X}_r\left(q^{-1}\right)$, 
\begin{align*}
&{L_{p,r}(s_1,\ldots,s_r;\omega^{k_1},\ldots,\omega^{k_r};\gamma_1,\ldots,\gamma_{r};c)}\\
&\quad =\sum_{n_1,\ldots,n_r=0}^\infty \bigg\{\int_{\left( \mathbb{Z}_p^r\right)'_{\{\gamma_j\}}}\prod_{\nu=1}^{r} \left(\langle \sum_{j=1}^{\nu}x_{j}\gamma_{j} \rangle-1\right)^{n_\nu}\\
& \qquad \qquad \times \omega^{k_1}(x_1\gamma_1)\cdots\omega^{k_r}( \sum_{j=1}^{r}x_{j}\gamma_{j}) \prod_{j=1}^{r}d\mm(x_j)\bigg\}\prod_{j=1}^{r}\binom{-s_j}{n_j}\\
& \quad =\sum_{n_1,\ldots,n_r=0}^\infty a(n_1,\ldots,n_r)\prod_{j=1}^{r}\binom{-s_j}{n_j},
\end{align*}
say. Since $\langle \sum_{j=1}^{\nu}x_{j}\gamma_{j} \rangle\equiv 1$ mod $q$, we have $|a(n_1,\ldots,n_r)|_p\leqslant q^{-\sum n_j}$. 
Applying Lemma \ref{multi-lemma1} with $d=q^{-1}$ and $M=1$, we obtain the proof of Theorem \ref{Th-pMLF}. 
\end{proof}


Next, we discuss the $p$-adic continuity of our $p$-adic multiple $L$-function
\eqref{e-6-1} with respect to the parameter $c$. 

\begin{theorem}\label{continuity theorem}
Let 
$s_1,\ldots,s_r\in \mathbb{Z}_p$,
$k_1,\ldots,k_r\in \mathbb{Z}$,
$\gamma_1,\ldots,\gamma_r\in\mathbb{Z}_p$ and
$c\in \mathbb{N}_{>1}$ with $(c,p)=1$.
Then 
the map 
$$c\mapsto{L_{p,r}(s_1,\ldots,s_r;\omega^{k_1},\ldots,\omega^{k_r};\gamma_1,\ldots,\gamma_{r};c)}$$
is continuously extended to $c\in\mathbb{Z}_p^\times$
as a $p$-adic continuous function.

Moreover the extension is uniformly continuous
with respect to both parameters $c$ and $(s_j)$.
Namely, 
for any given $\varepsilon>0$,
there always exists $\delta>0$ such that 
$$
|L_{p,r}(s_1,\ldots,s_r;\omega^{k_1},\ldots,\omega^{k_r};\gamma_1,\ldots,\gamma_{r};c)
-L_{p,r}(s'_1,\ldots,s'_r;\omega^{k_1},\ldots,\omega^{k_r};\gamma_1,\ldots,\gamma_{r};c')|_p
<\varepsilon
$$
holds for any $c, c'\in\mathbb Z_p^\times$ with $|c-c'|_p<\delta$, 
and any $s_j, s'_j\in\mathbb Z_p$ with $|s_j-s_j'|_p<\delta$
($1\leqslant j\leqslant r$).
\end{theorem}
To prove 
Theorem \ref{continuity theorem}, we prepare several lemmas and propositions.

  Let ${\rm Meas}_{\mathbb Z_p}({\mathbb Z_p})$ be the space of all 
$\mathbb Z_p$-valued measures on $\mathbb{Z}_p$. 
By \cite[Subsection\,12.2]{Wa}, 
it is identified with the completed group algebra  ${\mathbb Z_p}[[{\mathbb Z_p}]]$;
\begin{equation}\label{identificationI}
{\rm Meas}_{\mathbb Z_p}({\mathbb Z_p})\simeq
{\mathbb Z_p}[[{\mathbb Z_p}]].
\end{equation}
Again by loc.\,cit. Theorem 7.1,
it is identified with the one parameter formal power series ring 
${\mathbb Z_p}[[T]]$ by sending $1\in {\mathbb Z_p}[[{\mathbb Z_p}]]$ to $1+T\in {\mathbb Z_p}[[T]]$:
\begin{equation}\label{identificationII}
{\mathbb Z_p}[[{\mathbb Z_p}]]\simeq{\mathbb Z_p}[[T]].
\end{equation}

By loc.\,cit.\,Subsection\,12.2, we obtain the following:

\begin{lemma}
Let $c\in \mathbb{N}_{>1}$ with $(c,p)=1$.
By the above correspondences,
$\widetilde{\mathfrak{m}}_c\in {\rm Meas}_{\mathbb Z_p}({\mathbb Z_p})$ 
corresponds to $g_c(T)\in {\mathbb Z_p}[[T]]$ 
defined by 
\begin{equation}
g_c(T)=\frac{1}{(1+T)-1}-\frac{c}{(1+T)^c-1}.
\end{equation}
\end{lemma}


\begin{lemma}
The map $c\mapsto g_c(T)$ is uniquely extended 
to a $p$-adic continuous function
$g:\mathbb Z_p^\times\to{\mathbb Z_p}[[T]]$.
\end{lemma}

\begin{proof}
The series $(1+T)^c$ is well-defined 
in ${\mathbb Z_p}[[T]]$
for any $c\in\mathbb Z_p$ and 
is continuous with respect to 
 $c\in\mathbb Z_p$
because 
$$
{\mathbb Z_p}[[T]]\simeq\underset{\leftarrow}{\lim} \ 
{\mathbb Z_p}[T]/
\bigl((1+T)^{p^N}-1\bigr)
$$
(cf. \cite[Theorem 7.1]{Wa}).
When $c\in\mathbb Z_p^\times$, 
$\frac{c}{(1+T)^c-1}$ belongs to $\frac{1}{T}+{\mathbb Z_p}[[T]]$.
Hence, $g_c(T)$ belongs to  ${\mathbb Z_p}[[T]]$.
\end{proof}

For $c\in\mathbb Z_p^\times$, 
we denote 
$\widetilde{\mathfrak{m}}_c$
to be the $\mathbb Z_p$-valued measure on $\mathbb{Z}_p$ 
which corresponds to $g_c(T)$
by \eqref{identificationI} and \eqref{identificationII}.
We note that it coincides with \eqref{KM-measure}
when $c\in \mathbb{N}_{>1}$ with $(c,p)=1$. 

\begin{remark}\label{measure extension}
For the parameters 
$s_1,\ldots,s_r\in \mathbb{C}_p$ with $|s_j|_p<qp^{-1/(p-1)}$ $(1\leqslant j\leqslant r)$,
$k_1,\ldots,k_r\in \mathbb{Z}$,
$\gamma_1,\ldots,\gamma_r\in\mathbb Z_p$, and
$c\in \mathbb Z_p^\times$,
the $p$-adic multiple $L$-function
$${L_{p,r}(s_1,\ldots,s_r;\omega^{k_1},\ldots,\omega^{k_r};\gamma_1,\ldots,\gamma_{r};c)}$$
is defined by \eqref{e-6-1}
with the above measure $\widetilde{\mathfrak{m}}_c$.
\end{remark}

Let $C( \mathbb{Z}^r_p, \mathbb{C}_p)$ be the $ \mathbb{C}_p$-Banach space
of continuous  $\mathbb{C}_p$-valued functions on $ \mathbb{Z}^r_p$
with $||f||:=\sup_{x\in\mathbb Z^r_p} |f(x)|_p$, 
and let $\textrm{Step}( \mathbb{Z}^r_p)$ be the set of $ \mathbb{C}_p$-valued locally constant functions
on $ \mathbb{Z}^r_p$.
The subspace $\textrm{Step}( \mathbb{Z}^r_p)$ is dense in $C( \mathbb{Z}^r_p, \mathbb{C}_p)$ 
 (cf. \cite[Section\,12.1]{Wa}).

\begin{proposition}
For each function $f\in C( \mathbb{Z}_p, \mathbb{C}_p)$,
the map 
$$\Phi_f:\mathbb{Z}_p^\times\to\mathbb{C}_p$$
sending $c\mapsto\int_{\mathbb{Z}_p}f(z)d\widetilde{\mathfrak{m}}_c(z)$
is $p$-adically continuous.
\end{proposition}

\begin{proof}
First, assume that $f=\chi_{j,N}$, where $\chi_{j,N}$
($0\leqslant j<p^N$, $N\in\mathbb N$)
is the characteristic function of
the set $j+p^N\mathbb Z_p$.
Then $\Phi_f(c)$ is calculated to be $a_j(g_c)$ (cf. \cite[Section\,12.2]{Wa}), which is defined by
$$
g_c(T)\equiv\sum_{i=0}^{p^N-1}a_i(g_c)(1+T)^i\pmod{(1+T)^{p^N}-1}.
$$ 
Since the map $c\mapsto g_c$ is continuous, 
$c\mapsto\Phi_f(c)=a_j(g_c)$ is also continuous in this case.
This implies the continuity of $\Phi_f$
in $c$ when $f\in \textrm{Step} (\mathbb Z_p)$.

Next, assume $f\in C( \mathbb{Z}_p, \mathbb{C}_p)$
and $c\in\mathbb Z_p^\times$.
Then for any  $g\in C( \mathbb{Z}_p, \mathbb{C}_p)$ and
$c'\in\mathbb Z_p^\times$,
we have
\begin{align*}
\Phi_f(c) & -\Phi_f(c')=
\int_{\mathbb{Z}_p}f(z)d\widetilde{\mathfrak{m}}_c(z)-
\int_{\mathbb{Z}_p}f(z)d\widetilde{\mathfrak{m}}_{c'}(z) \\
&=\int_{\mathbb{Z}_p}(f(z)-g(z))\cdot
d(\widetilde{\mathfrak{m}}_c(z)-\widetilde{\mathfrak{m}}_{c'}(z))
+\int_{\mathbb{Z}_p}g(z)\cdot
d(\widetilde{\mathfrak{m}}_c(z)-\widetilde{\mathfrak{m}}_{c'}(z)) \\
&=(\Phi_{f-g}(c) -\Phi_{f-g}(c')) +(\Phi_g(c)-\Phi_g(c')).
\end{align*}
Let $\varepsilon>0$. Since $\textrm{Step}(\mathbb Z_p)$ is dense in $C( \mathbb{Z}_p, \mathbb{C}_p)$,
there exists $g_0\in \textrm{Step}(\mathbb Z_p)$ with $||f-g_0||<\varepsilon$. 
Thus, for any $c\in\mathbb Z_p^\times$, 
we have
$$
|\Phi_{f-g_0}(c)|_p=
\Bigl|\int_{\mathbb{Z}_p}(f(z)-g_0(z))\cdot d\widetilde{\mathfrak{m}}_c(z)\Bigr|_p
\leqslant 
||f-g_0|| 
<\varepsilon,
$$
because  $\widetilde{\mathfrak{m}}_c$ is a $\mathbb Z_p$-valued measure.

On the other hand, since $g_0\in \textrm{Step}(\mathbb Z_p)$,
there exists $\delta >0$ such that
$$
|\Phi_{g_0}(c)  -\Phi_{g_0}(c')|_p<\varepsilon
$$
for any $c'\in\mathbb Z_p^\times$ with $|c-c'|_p<\delta$.

Therefore, for $f\in C( \mathbb{Z}_p, \mathbb{C}_p)$, $c\in\mathbb Z_p^\times$
and any $\varepsilon>0$, 
there always exists  $\delta >0$ such that
$$
|\Phi_f(c)  -\Phi_f(c')|_p\leqslant \max \left\{ |\Phi_{f-g_0}(c)|_p ,\,|\Phi_{f-g_0}(c')|_p,\,|\Phi_{g_0}(c)-\Phi_{g_0}(c')|_p\right\}
<\varepsilon 
$$
for any $c'\in\mathbb Z_p^\times$ with $|c-c'|_p<\delta$.
\end{proof}

By generalizing our arguments above, we obtain the following:
\begin{proposition}\label{continuity}
For each function $f\in C( \mathbb{Z}_p^r, \mathbb{C}_p)$,
the map 
$$\Phi^r_f:\mathbb{Z}_p^\times\to\mathbb{C}_p$$
sending $c\mapsto\int_{\mathbb{Z}_p^r}f(x_1,\ldots,x_r)
\prod_{j=1}^{r}d\widetilde{\mathfrak{m}}_c(x_j)
$
is $p$-adically continuous.
\end{proposition}

A proof of Theorem \ref{continuity theorem} is attained by the above proposition.
This is the reason why we restrict Theorem \ref{continuity theorem} 
to the case for $s_1,\ldots,s_r\in \mathbb{Z}_p$.

\begin{proof}[Proof of Theorem \ref{continuity theorem}]

Let us fix notation: Define 
\begin{equation*}
\begin{split}
&
{W(s_1,\dots s_r;x_1,\dots,x_r)}
\\
&:=
\langle x_1\gamma_1 \rangle^{-{s_1}}\langle x_1\gamma_1+ x_2\gamma_2 \rangle^{-{s_2}}\cdots \langle \sum_{j=1}^{r}x_{j}\gamma_{j} \rangle^{-{s_r}}
\cdot
\omega^{k_1}(x_1\gamma_1)\cdots\omega^{k_r}( \sum_{j=1}^{r}x_{j}\gamma_{j}) \cdot
\chi_{_{\left( \mathbb{Z}_p^r\right)'_{\{\gamma_j\}}}}(x_1,\ldots,x_r),
\end{split}
\end{equation*}
where $\chi_{_{\left( \mathbb{Z}_p^r\right)'_{\{\gamma_j\}}}}(x_1,\ldots,x_r)$
is the characteristic function of 
${\left( \mathbb{Z}_p^r\right)'_{\{\gamma_j\}}}$ (cf. Definition \ref{Def-pMLF}).
Then we have
\begin{equation}\label{integration}
{L_{p,r}(s_1,\ldots,s_r;\omega^{k_1},\ldots,\omega^{k_r};\gamma_1,\ldots,\gamma_{r};c)}
:=\int_{\mathbb Z_p^r}
W(s_1,\dots ,s_r;x_1,\dots,x_r)\cdot
\prod_{j=1}^{r}d\widetilde{\mathfrak{m}}_c(x_j).
\end{equation}

Below we will prove that the map
$$
\varPsi:
(s_1,\ldots,s_r, c)\in \mathbb Z_p^r\times \mathbb Z_p^\times\mapsto
L_{p,r}(s_1,\ldots,s_r;\omega^{k_1},\ldots,\omega^{k_r};\gamma_1,\ldots,\gamma_{r};c)
\in \mathbb Z_p
$$
is continuous: 

First, by Proposition \ref{continuity} we see that 
it is continuous
with respect to $c\in\mathbb Z_p^\times$ for each fixed $(s_1,\dots,s_r)\in \mathbb Z_p^r$.
Namely, for each fixed $(s_1,\dots,s_r)\in \mathbb Z_p^r$ and $c\in \mathbb Z_p^\times$, 
and for any given $\varepsilon>0$,
there always exists $\delta>0$ such that
\begin{equation*}
|L_{p,r}(s_1,\ldots,s_r;\omega^{k_1},\ldots,\omega^{k_r};\gamma_1,\ldots,\gamma_{r};c)
-L_{p,r}(s_1,\ldots,s_r;\omega^{k_1},\ldots,\omega^{k_r};\gamma_1,\ldots,\gamma_{r};c')|_p<\varepsilon;
\end{equation*}
equivalently, 
\begin{equation}\label{bound1}
|\varPsi(s_1,\ldots,s_r; c)-\varPsi(s_1,\ldots,s_r; c')|_p<\varepsilon
\end{equation}
for all $c'\in\mathbb Z_p^\times$ with $|c-c'|_p<\delta$.

Next, take $M\in\mathbb N$ such that $\varepsilon> p^{-M}>0$.
Since
it is easy to see that there exists $\delta'>0$
such that
$$(1+u)^d\equiv 1\pmod{p^M}$$
for any $u\in p\mathbb Z_p$ and
$d\in\mathbb Z_p$ with $|d|_p<\delta'$
(
In fact, we may take $\delta'$ as $\delta'<p^{-M}$),
we have
$$x^s\equiv x^{s'}\pmod{p^M}$$
for all  $x \in (1+p\mathbb Z_p)$ and
$s,s'\in\mathbb Z_p$
with $|s-s'|<\delta'$ for such $\delta'$. 
Therefore
$$
W(s_1,\dots, s_r;x_1,\dots,x_r)\equiv W(s'_1,\dots, s'_r;x_1,\dots,x_r) \pmod{p^M}
$$
for all $(x_1,\dots, x_r)\in\mathbb Z_p^r$
when $|s_i-s'_i|_p<\delta'$ ($1\leqslant i\leqslant r$).
Thus, in this case, 
the inequality
\begin{align*}
|L_{p,r}(s_1,\ldots,s_r; \omega^{k_1},\ldots,\omega^{k_r};\gamma_1,\ldots,\gamma_{r};&c)-
L_{p,r}(s'_1,\ldots,s'_r;\omega^{k_1},\ldots,\omega^{k_r};\gamma_1,\ldots,\gamma_{r};c)|_p \\
&\leqslant p^{-M}
<\varepsilon
\end{align*}
holds for any $c\in\mathbb Z_p^\times$, 
because $\widetilde{\mathfrak{m}}_c$ is a $\mathbb Z_p$-valued measure.
Therefore, for any given $\varepsilon>0$,
there always exists $\delta'>0$ such that, 
when $|s_i-s'_i|_p<\delta'$ ($1\leqslant i\leqslant r$),
\begin{equation}\label{bound2}
|\varPsi(s_1,\ldots,s_r;c)-\varPsi(s'_1,\ldots,s'_r; c)|_p<\varepsilon
\end{equation}
holds for all $c\in\mathbb Z_p^\times$.

By \eqref{bound1} and \eqref{bound2}, we see that
for each fixed $(s_1,\dots,s_r)\in \mathbb Z_p^r$ and $c\in \mathbb Z_p^\times$,
and for any given $\varepsilon>0$,
there always exist $\delta, \delta'>0$ such that
\begin{align*}
\bigl|\varPsi&(s_1,\ldots,s_r;c)-\varPsi(s'_1,\ldots,s'_r; c')\bigr|_p \\
&=\bigl|\varPsi(s_1,\ldots,s_r;c)-\varPsi(s_1,\ldots,s_r;c')+\varPsi(s_1,\ldots,s_r;c')
-\varPsi(s'_1,\ldots,s'_r; c')\bigr|_p \\
&\leqslant
\max
\Bigl\{
\bigl|\varPsi(s_1,\ldots,s_r;c)-\varPsi(s_1,\ldots,s_r;c')\bigr|_p, 
\bigl|\varPsi(s_1,\ldots,s_r;c')-\varPsi(s'_1,\ldots,s'_r; c')\bigr|_p
\Bigr\}<\varepsilon
\end{align*}
for  $|s_i-s'_i|_p<\delta'$ ($1\leqslant i\leqslant r$) and $|c-c'|_p<\delta$.
Thus, we get the desired continuity of $\varPsi(s_1,\ldots,s_r; c)$.

The uniform continuity of $\varPsi(s_1,\ldots,s_r; c)$
is almost obvious because 
now we know that the function
is continuous and the source set $\mathbb Z_p^r\times \mathbb Z_p^\times$ is compact.
\end{proof}

As a corollary of Theorem \ref{continuity theorem},
the following non-trivial property of special values of $p$-adic multiple $L$-functions
at non-positive integer points
is obtained.

\begin{corollary}\label{non-trivial property}
The right-hand side of equation
\eqref{Th-main} of Theorem \ref{T-main-1} (in the next section)
is $p$-adically continuous
with respect to $c$. 
\end{corollary}

We conclude this section with
a speculation how to construct a $p$-adic 
analogue of the entire function
$\zeta_r^{\rm des}(s_1,\ldots,s_r;\gamma_1,\ldots,\gamma_{r})$
which was constructed 
as a multiple analogue of the entire function $(1-s)\zeta(s)$ in \cite{FKMT-Desing}. 
By Theorem \ref{T-main-1} and Remark \ref{rem-intpln}, below, 
our
$L_{p,r}(s_1,\ldots,s_r;\omega^{n_1},\ldots,\omega^{n_r};\gamma_1,\ldots,\gamma_{r};c)$
is a $p$-adic interpolation of \eqref{Int-P}.

By $\lim_{c\to 1}g_c(T)=0$, we have,
for $s_1,\ldots,s_r\in \mathbb{Z}_p$,
$k_1,\ldots,k_r\in \mathbb{Z}$, 
and 
$\gamma_1,\ldots,\gamma_r\in\mathbb Z_p$,
\begin{equation*}
\underset{c\in\mathbb Z_p^\times}{\lim_{c\to 1}}{L_{p,r}(s_1,\ldots,s_r;\omega^{k_1},\ldots,\omega^{k_r};\gamma_1,\ldots,\gamma_{r};c)}=0.
\end{equation*}
On the other hand, by
$$\lim_{c\to 1}\frac{g_c(T)}{c-1}\in\mathbb{Q}_p[[T]]\setminus\mathbb{Z}_p[[T]],$$
we cannot say that the limit
$\lim_{c\to 1}
\frac{1}{c-1}{\widetilde{\mathfrak{m}}_c}$ 
converges to a measure.
Thus, 
we encounter the following:

\begin{problem}\label{Main-Prob}
For $s_1,\ldots,s_r\in \mathbb{Z}_p$,
$k_1,\ldots,k_r\in \mathbb{Z}$ and
$\gamma_1,\ldots,\gamma_r\in\mathbb Z_p$, does
\begin{equation}\label{limit}
\underset{c\in\mathbb Z_p^\times\setminus\{1\}}
{\lim_{c\to 1}}\frac{1}{(c-1)^r}L_{p,r}(s_1,\ldots,s_r;\omega^{k_1},\ldots,\omega^{k_r};\gamma_1,\ldots,\gamma_{r};c)
\end{equation}
converge?
\end{problem} 

If the limit \eqref{limit} exists and happens to be a rigid analytic function,
we may call it a $p$-adic analogue of the desingularized zeta function
$\zeta^{\rm des}_r(s_1,\ldots,s_r;\gamma_1,\ldots,\gamma_{r})$.
%
We recall 
that the problem is affirmative 
when 
$r=1$ and $\gamma=1$. 
In fact, by \eqref{1-p-LF-gamma}, 
we have
\begin{equation*}
\lim_{c\to 1}{(c-1)^{-1}L_{p,1}(s;\omega^k;1;c)}=
(1-s)\cdot L_p(s;\omega^{k+1}).
\end{equation*}
We also note  that the limit converges when $(s_1,\ldots,s_r)=(-n_1,\ldots,-n_r)$, 
for $n_1,\ldots,n_r\in\mathbb{N}_0$, 
by Theorem \ref{T-5-gene} in the next section.

\

\section{Special values of $p$-adic multiple $L$-functions at non-positive integers}\label{sec-4}

In this section, we 
consider the special values of our $p$-adic multiple $L$-functions
(Definition \ref{Def-pMLF}) at non-positive integers, 
and shall 
express them in terms of twisted multiple Bernoulli numbers
(Definition \ref{Def-M-Bern}) in Theorems \ref{T-main-1} and \ref{T-5-gene}.
We will see that our  $p$-adic multiple $L$-function
is a $p$-adic interpolation of
a certain sum \eqref{Int-P} of the entire  complex multiple zeta-functions  of generalized Euler-Zagier-Lerch type
in Remark \ref{rem-intpln}.
Based on 
these evaluations, we 
generalize the well-known Kummer congruences for ordinary Bernoulli numbers
to the multiple Kummer congruences for twisted multiple Bernoulli numbers in Theorem \ref{Th-Kummer}.
We 
also show certain functional relations with a parity condition
among $p$-adic multiple $L$-functions in Theorem \ref{T-6-1}. 
These functional relations 
may be seen as multiple generalizations of 
the vanishing property of the Kubota-Leopoldt $p$-adic $L$-functions with odd characters
(cf. Proposition \ref{Prop-zero})  in the single variable case, 
and of the  functional relations for the $p$-adic double $L$-functions (shown in \cite{KMT-IJNT})
in the double variable case under the special condition $c=2$.
Many 
examples will be investigated in this section.

\subsection{Evaluation of $p$-adic multiple $L$-functions at non-positive integers}\label{negative-1}

Based on the consideration in the previous section, we determine 
the values of $p$-adic multiple $L$-functions at non-positive integers as follows: 

\begin{theorem}\label{T-main-1}
For $n_1,\ldots,n_r\in \mathbb{N}_0$, $\gamma_1,\ldots,\gamma_r\in \mathbb{Z}_p$, 
and $c\in \mathbb{N}_{>1}$ with $(c,p)=1$, 
\begin{align}
& L_{p,r}(-n_1,\ldots,-n_r;\omega^{n_1},\ldots,\omega^{n_r};\gamma_1,\ldots,\gamma_{r};c)\notag\\
& =\sum_{\xi_1^c=1 \atop \xi_1\not=1}\cdots\sum_{\xi_r^c=1 \atop \xi_r\not=1}\aa((n_j);(\xi_j);(\gamma_j)) \notag\\
& +\sum_{d=1}^{r}\left(-\frac{1}{p}\right)^d \sum_{1\leqslant i_1<\cdots<i_d\leqslant r}\sum_{\rho_{i_1}^p=1}\cdots\sum_{\rho_{i_d}^p=1}\sum_{\xi_1^c=1 \atop \xi_1\not=1}\cdots\sum_{\xi_r^c=1 \atop \xi_r\not=1}\aa((n_j);((\prod_{j\leqslant i_l}\rho_{i_l})^{\gamma_j}\xi_j);(\gamma_j)), \label{Th-main}
\end{align}
where the empty product is interpreted as $1$.
\end{theorem}

In Definition \ref{Def-M-Bern}, the twisted multiple Bernoulli number $\aa((n_j);(\xi_j);(\gamma_j))$ was defined for 
$\gamma_1,\ldots,\gamma_r\in \mathbb{C}$ and roots of unity $\xi_1,\ldots,\xi_r\in \mathbb{C}$. Here we define $\aa((n_j);(\xi_j);(\gamma_j))$ for $\gamma_1,\ldots,\gamma_r\in \mathbb{Z}_p$ and roots of unity $\xi_1,\ldots,\xi_r\in \mathbb{C}_p$ by \eqref{Def-Hr} and \eqref{Fro-def-r}.

\begin{remark}\label{rem-intpln}
When $\gamma_1,\ldots,\gamma_r\in \mathbb{Z}_p\cap \overline{\mathbb{Q}}$ satisfy 
$\Re \gamma_j >0 \quad (1\leqslant j\leqslant r)$, 
we obtain from the above theorem 
that
\begin{align*}
& L_{p,r}(-n_1,\ldots,-n_r;\omega^{n_1},\ldots,\omega^{n_r};\gamma_1,\ldots,\gamma_{r};c)
=(-1)^{r+n_1+\cdots+n_r}
\Bigl\{\sum_{\xi_1^c=1 \atop \xi_1\not=1}\cdots\sum_{\xi_r^c=1 \atop \xi_r\not=1}
\zeta_{r}((-n_j);(\xi_j);(\gamma_j)) \notag\\
&\qquad\qquad
 +\sum_{d=1}^{r}\left(-\frac{1}{p}\right)^d \sum_{1\leqslant i_1<\cdots<i_d\leqslant r}\sum_{\rho_{i_1}^p=1}\cdots\sum_{\rho_{i_d}^p=1}\sum_{\xi_1^c=1 \atop \xi_1\not=1}\cdots\sum_{\xi_r^c=1 \atop \xi_r\not=1}
\zeta_{r}((-n_j);((\prod_{j\leqslant i_l}\rho_{i_l})^{\gamma_j}\xi_j);(\gamma_j))\Bigr\},
\end{align*}
because 
\begin{equation*}
\zeta_r((-n_j);(\xi_j);(\gamma_j))=(-1)^{r+n_1+\cdots+n_r}\aa((n_j);( \xi_j^{-1});( \gamma_j)) 
\end{equation*}
(see \cite[Theorem 4.1]{FKMT-Desing}). 
Hence, 
the $p$-adic multiple $L$-function $L_{p,r}((s_j);(\omega^{k_1});(\gamma_j);c)$
is a $p$-adic interpolation of 
the  following finite sum of multiple zeta-functions $\zeta_r((s_j);(\xi_j);(\gamma_j))$
which are  entire:
\begin{equation}
\sum_{\xi_1^c=1 \atop \xi_1\not=1}\cdots \sum_{\xi_r^c=1 \atop \xi_r\not=1}
\zeta_r((s_j);(\xi_j);(\gamma_j)) \label{Int-P}.
\end{equation}
\end{remark}

\begin{proof}[Proof of Theorem \ref{T-main-1}]
We see that for any $\rho\in \mu_p$ and $\xi\in \mu_c\setminus\{1\}$, 
\begin{align*}
\mk_{\rho\xi}(j+p^N\mathbb{Z}_p)& =\frac{(\rho\xi)^j}{1-(\rho\xi)^{p^{N}}}=\rho^{j}\mk_{\xi}(j+p^N\mathbb{Z}_p) \qquad (N\geqslant 1). 
\end{align*}
Hence, using 
\begin{equation*}
\sum_{\rho^p=1}\rho^n=
\begin{cases}
0 & (p\nmid n)\\
p & (p\mid n),
\end{cases}
\end{equation*}
we have
\begin{align*}
& L_{p,r}(-n_1,\ldots,-n_r;\omega^{n_1},\ldots,\omega^{n_r};\gamma_1,\ldots,\gamma_{r};c)\\
& =\int_{\mathbb{Z}_p^r}(x_1\gamma_1)^{n_1} \cdots (\sum_{j=1}^{r}x_{j}\gamma_{j})^{n_r}\prod_{i=1}^{r}\left(1-\frac{1}{p}\sum_{\rho_i^p=1}\rho_i^{\sum_{\nu=1}^{i}x_\nu \gamma_\nu}\right) \prod_{j=1}^{r}\sum_{\xi_j^c=1 \atop \xi_j\not=1}d\mk_{\xi_j}(x_j)\\
& =\int_{\mathbb{Z}_p^r}(x_1\gamma_1)^{n_1} \cdots (\sum_{j=1}^{r}x_{j}\gamma_{j})^{n_r}\sum_{\xi_1^c=1 \atop \xi_1\not=1}\cdots \sum_{\xi_r^c=1 \atop \xi_r\not=1}\prod_{j=1}^{r}d\mk_{\xi_j}(x_j)\\
& \qquad +\sum_{d=1}^{r}\left(-\frac{1}{p}\right)^d \sum_{1\leqslant i_1<\cdots<i_d\leqslant r}\sum_{\rho_{i_1}^p=1}\cdots\sum_{\rho_{i_d}^p=1}\sum_{\xi_1^c=1 \atop \xi_1\not=1}\cdots\sum_{\xi_r^c=1 \atop \xi_r\not=1}\\
& \qquad \qquad \times 
\int_{\mathbb{Z}_p^r}\prod_{l=1}^{r} (\sum_{j=1}^{l}x_{j}\gamma_{j})^{n_l} \prod_{l=1}^{d}\rho_{i_l}^{\sum_{\nu=1}^{i_l}x_\nu\gamma_{\nu}}\prod_{j=1}^{r}d\mk_{\xi_j}(x_j)\\
& =\int_{\mathbb{Z}_p^r}(x_1\gamma_1)^{n_1} \cdots (\sum_{j=1}^{r}x_{j}\gamma_{j})^{n_r}\sum_{\xi_1^c=1 \atop \xi_1\not=1}\cdots \sum_{\xi_r^c=1 \atop \xi_r\not=1}\prod_{j=1}^{r}d\mk_{\xi_j}(x_j)\\
& \qquad +\sum_{d=1}^{r}\left(-\frac{1}{p}\right)^d \sum_{1\leqslant i_1<\cdots<i_d\leqslant r}\sum_{\rho_{i_1}^p=1}\cdots\sum_{\rho_{i_d}^p=1}\sum_{\xi_1^c=1 \atop \xi_1\not=1}\cdots\sum_{\xi_r^c=1 \atop \xi_r\not=1}\\
& \qquad \qquad \times 
\int_{\mathbb{Z}_p^r}\prod_{l=1}^{r} (\sum_{j=1}^{l}x_{j}\gamma_{j})^{n_l} \prod_{j=1}^{r}d\mk_{\{(\prod_{j\leqslant i_l}\rho_{i_l})^{\gamma_j}\xi_j\}}(x_j).
\end{align*}
Thus, 
the desired identity has been shown by Proposition \ref{T-multi-2}.
\end{proof}

\begin{remark}
We reiterate Corollary \ref{non-trivial property}: 
The right-hand side of 
equation \eqref{Th-main} is $p$-adically continuous, 
not only
with respect to $n_1,\ldots,n_r$, 
but also with respect to $c$.
\end{remark}


Considering the Galois action of ${\rm Gal}(\overline{\mathbb{Q}}_p/{\mathbb{Q}}_p)$, 
we obtain the following result from Theorems \ref{Th-pMLF} and \ref{continuity theorem}.

\begin{corollary}\label{C-main-0}
For $n_1,\ldots,n_r\in \mathbb{N}_0$, $\gamma_1,\ldots,\gamma_r\in \mathbb{Z}_p$
and $c\in \mathbb{Z}_{p}^\times$,
\begin{align*}
& L_{p,r}(-n_1,\ldots,-n_r;\omega^{n_1},\ldots,\omega^{n_r};\gamma_1,\ldots,\gamma_{r};c)\in \mathbb{Z}_p.
\end{align*}
Hence, the right-hand side of \eqref{Th-main} is in $\mathbb{Z}_p$, though it includes the terms like $\left(-\frac{1}{p}\right)^d$. 
In particular, when $\gamma_1,\ldots,\gamma_{r}\in \mathbb{Z}_{(p)}:=\{\frac{a}{b}\in \mathbb{Q}\,|\,a,b\in \mathbb{Z},\ (b,p)=1\}=\mathbb{Z}_p\cap \mathbb{Q}$, then 
\begin{align*}
& L_{p,r}(-n_1,\ldots,-n_r;\omega^{n_1},\ldots,\omega^{n_r};\gamma_1,\ldots,\gamma_{r};c)\in \mathbb{Z}_{(p)}.
\end{align*}
\end{corollary}

The following examples are special cases ($r=1$ and $r=2$) of Theorem \ref{T-main-1}.

\begin{example}
For $n\in \mathbb{N}_0$ 
and $c\in \mathbb{N}_{>1}$ with $(c,p)=1$,

\begin{align}
L_{p,1}(-n;\omega^n;1;c)&=(1-p^n)\sum_{\xi^c=1 \atop \xi\neq 1}\aa_n(\xi),\label{KL-LP}
\end{align}
which recovers the known fact (see \cite[Theorem 5.11]{Wa})
$$
L_p(-n;\omega^{n+1})
=-(1-p^n)\frac{B_{n+1}}{n+1},$$
thanks to \eqref{2-0-1} and \eqref{p-L-vals}. 
\end{example}

\begin{example}\label{C-main-1}
For $n_1,n_2\in \mathbb{N}_0$, $\gamma_1,\gamma_2\in \mathbb{Z}_p$, 
and $c\in \mathbb{N}_{>1}$ with $(c,p)=1$, 
\begin{align}
L_{p,2}&(-n_1,-n_2;\omega^{n_1},\omega^{n_2};\gamma_1,\gamma_{2};c)\notag\\
& =\sum_{\xi_1^c=1 \atop \xi_1\not=1}\sum_{\xi_2^c=1 \atop \xi_2\not=1}\aa(n_1,n_2;\xi_1,\xi_2;\gamma_1,\gamma_2) \notag\\
& \quad -\frac{1}{p} \sum_{\rho_{1}^p=1}\sum_{\xi_1^c=1 \atop \xi_1\not=1}\sum_{\xi_2^c=1 \atop \xi_2\not=1}\aa(n_1,n_2;\rho_{1}^{\gamma_1}\xi_1,\xi_2;\gamma_1,\gamma_2)\notag\\
& \quad -\frac{1}{p} \sum_{\rho_{2}^p=1}\sum_{\xi_1^c=1 \atop \xi_1\not=1}\sum_{\xi_2^c=1 \atop \xi_2\not=1}\aa(n_1,n_2;\rho_{2}^{\gamma_1}\xi_1, \rho_{2}^{\gamma_2}\xi_2;\gamma_1,\gamma_2)\notag\\
& \quad +\frac{1}{p^2} \sum_{\rho_{1}^p=1}\sum_{\rho_{2}^p=1}\sum_{\xi_1^c=1 \atop \xi_1\not=1}\sum_{\xi_2^c=1 \atop \xi_2\not=1}\aa(n_1,n_2;(\rho_{1}\rho_2)^{\gamma_1}\xi_1,\rho_{2}^{\gamma_2}\xi_2;\gamma_1,\gamma_2). \label{C-1-formula}
\end{align}
\end{example}

More generally, we consider the generating function of $\{ L_{p,r}((-n_j);(\omega^{n_j});(\gamma_j);c)\}$, that is, 
\begin{align*}
&F_{p,r}(t_1,\ldots,t_r;(\gamma_j);c)=\sum_{n_1=0}^\infty \cdots\sum_{n_r=0}^\infty L_{p,r}((-n_j);(\omega^{n_j});(\gamma_j);c)\prod_{j=1}^{r}\frac{t_j^{n_j}}{n_j!}.
\end{align*}
Then we have the following:

\begin{theorem}\label{T-5-gene}
Let $c\in\mathbb Z_p^\times$.
For $\gamma_1,\ldots,\gamma_r\in \mathbb{Z}_p$,
\begin{align*}
&F_{p,r}(t_1,\ldots,t_r;(\gamma_j);c)\notag\\
& =\prod_{j=1}^{r}\left(\frac{1}{\exp\left(\gamma_j\sum_{k=j}^r t_k\right)-1}-\frac{c}{\exp\left(c\gamma_j\sum_{k=j}^r t_k\right)-1}\right)\notag\\
& +\sum_{d=1}^{r}\left(-\frac{1}{p}\right)^d \sum_{1\leqslant i_1<\cdots<i_d\leqslant r}\sum_{\rho_{i_1}^p=1}\cdots\sum_{\rho_{i_d}^p=1}\notag\\
& \quad\times \prod_{j=1}^{r}\left(\frac{1}{(\prod_{j\leqslant i_l}\rho_{i_l})^{\gamma_j}\exp\left(\gamma_j\sum_{k=j}^r t_k\right)-1}-\frac{c}{(\prod_{j\leqslant i_l}\rho_{i_l})^{c\gamma_j}\exp\left(c\gamma_j\sum_{k=j}^r t_k\right)-1}\right)\\
& =\prod_{j=1}^{r}\left(\sum_{m_j=0}^\infty \left(1-c^{m_j+1}\right)\frac{B_{m_j+1}}{m_j+1}\frac{\left(\gamma_j\sum_{k=j}^r t_k\right)^{m_j}}{m_j!}\right)\notag\\
&\quad +(-1)^r\sum_{d=1}^{r}\left(-\frac{1}{p}\right)^d \sum_{1\leqslant i_1<\cdots<i_d\leqslant r}\sum_{\rho_{i_1}^p=1}\cdots\sum_{\rho_{i_d}^p=1}\notag\\
& \qquad\times \prod_{j=1}^{r}\left(\sum_{m_j=0}^\infty (1-c^{m_j+1})\frak{B}_{m_j}((\prod_{j\leqslant i_l}\rho_{i_l})^{\gamma_j})\frac{\left(\gamma_j\sum_{k=j}^r t_k\right)^{m_j}}{m_j!}\right). 
\end{align*}
In particular, when $\gamma_1\in \mathbb{Z}_p^\times$ and $\gamma_j\in p\mathbb{Z}_p$ $(2\leqslant j\leqslant r)$, 
\begin{align}
F_{p,r}(t_1,\ldots,t_r;(\gamma_j);c)
& =\bigg(\frac{1}{\exp\left(\gamma_1\sum_{k=1}^r t_k\right)-1}-\frac{c}{\exp\left(c\gamma_1\sum_{k=1}^r t_k\right)-1} \notag\\
& \qquad - \frac{1}{\exp\left(p\gamma_1\sum_{k=1}^r t_k\right)-1}+\frac{c}{\exp\left(cp\gamma_1\sum_{k=1}^r t_k\right)-1}\bigg)\notag\\
& \quad \times \prod_{j=2}^{r}\left(\frac{1}{\exp\left(\gamma_j\sum_{k=j}^r t_k\right)-1}-\frac{c}{\exp\left(c\gamma_j\sum_{k=j}^r t_k\right)-1}\right)\notag\\
& =\left( \sum_{n_1=0}^\infty \left(1-p^{n_1}\right)\left(1-c^{n_1+1}\right)\frac{B_{n_1+1}}{n_1+1}\frac{\left(\gamma_1\sum_{k=1}^{r}t_k\right)^{n_1}}{n_1!}\right)\notag\\
& \quad \times \prod_{j=2}^{r}\left(\sum_{n_j=0}^\infty \left(1-c^{n_j+1}\right)\frac{B_{n_j+1}}{n_j+1}\frac{\left(\gamma_j\sum_{k=j}^{r}t_k\right)^{n_j}}{n_j!}\right).\label{generating}
\end{align}
\end{theorem}

We stress here that
Theorem \ref{T-main-1} holds for  $c\in \mathbb{N}_{>1}$ with $(c,p)=1$; 
in contrast,
Theorem \ref{T-5-gene} holds, more generally, for $c\in\mathbb Z_p^\times$.

\begin{proof}
For the moment, assume that $c\in \mathbb{N}_{>1}$ with $(c,p)=1$. 
Combining \eqref{Fro-def-r} and \eqref{Th-main}, we have
\begin{align*}
F_{p,r}((t_j);(\gamma_j);c)
& = \sum_{\xi_1^c=1 \atop \xi_1\not=1}\cdots \sum_{\xi_r^c=1 \atop \xi_r\not=1}\prod_{j=1}^{r} \frac{1}{1-\xi_j \exp\left(\gamma_j \sum_{k=j}^r t_k\right)}\\
& \quad +\sum_{d=1}^{r}\left(-\frac{1}{p}\right)^d \sum_{1\leqslant i_1<\cdots<i_d\leqslant r}\sum_{\rho_{i_1}^p=1}\cdots\sum_{\rho_{i_d}^p=1}\sum_{\xi_1^c=1 \atop \xi_1\not=1}\cdots \sum_{\xi_r^c=1 \atop \xi_r\not=1}\notag\\
& \quad\times \prod_{j=1}^{r}\frac{1}{1-(\prod_{j\leqslant i_l}\rho_{i_l})^{\gamma_j}\xi_j\exp\left(\gamma_j\sum_{k=j}^r t_k\right)}\\
& = \prod_{j=1}^{r} \sum_{\xi_j^c=1 \atop \xi_j\not=1}\frac{1}{1-\xi_j \exp\left(\gamma_j \sum_{k=j}^r t_k\right)}\\
& \quad +\sum_{d=1}^{r}\left(-\frac{1}{p}\right)^d \sum_{1\leqslant i_1<\cdots<i_d\leqslant r}\sum_{\rho_{i_1}^p=1}\cdots\sum_{\rho_{i_d}^p=1}\\
& \quad \times \prod_{j=1}^{r}\left(\sum_{\xi_j^c=1 \atop \xi_j\not=1}\frac{1}{1-(\prod_{j\leqslant i_l}\rho_{i_l})^{\gamma_j}\xi_j\exp\left(\gamma_j\sum_{k=j}^r t_k\right)}\right).
\end{align*}
Hence, by \eqref{log-der}, we obtain the first assertion. Next, we assume that 
$\gamma_1\in \mathbb{Z}_p^\times$ and $\gamma_j\in p\mathbb{Z}_p$ $(2\leqslant j\leqslant r)$. Then 
$$(\prod_{j\leqslant i_l}\rho_{i_l})^{\gamma_j}=1\quad (2\leqslant j\leqslant r),$$
from which it follows that 
\begin{align*}
F_{p,r}((t_j);(\gamma_j);c)
& =\prod_{j=1}^{r}\left(\frac{1}{\exp\left(\gamma_j\sum_{k=j}^r t_k\right)-1}-\frac{c}{\exp\left(c\gamma_j\sum_{k=j}^r t_k\right)-1}\right)\notag\\
& \ +\sum_{d=1}^{r}\left(-\frac{1}{p}\right)^d \sum_{1\leqslant i_1<\cdots<i_d\leqslant r}\sum_{\rho_{i_1}^p=1}\cdots\sum_{\rho_{i_d}^p=1}\notag\\
& \quad\times \left(\frac{1}{(\prod_{i_l}\rho_{i_l})\exp\left(\gamma_1\sum_{k=1}^r t_k\right)-1}-\frac{c}{(\prod_{i_l}\rho_{i_l})^{c}\exp\left(c\gamma_1\sum_{k=1}^r t_k\right)-1}\right)\\
& \qquad \times \prod_{j=2}^{r}\left(\frac{1}{\exp\left(\gamma_j\sum_{k=j}^r t_k\right)-1}-\frac{c}{\exp\left(c\gamma_j\sum_{k=j}^r t_k\right)-1}\right).
\end{align*}
Using \eqref{log-der}, we can rewrite the second term on the right-hand side as
\begin{align*}
& \prod_{j=2}^{r}\left(\frac{1}{\exp\left(\gamma_j\sum_{k=j}^r t_k\right)-1}-\frac{c}{\exp\left(c\gamma_j\sum_{k=j}^r t_k\right)-1}\right)\\
& \quad \times \sum_{d=1}^{r}\left(-\frac{1}{p}\right)^d \sum_{1\leqslant i_1<\cdots<i_d\leqslant r}p^d\left(\frac{1}{\exp\left(p\gamma_1\sum_{k=1}^r t_k\right)-1}-\frac{c}{\exp\left(cp\gamma_1\sum_{k=1}^r t_k\right)-1}\right).
\end{align*}
Noting 
$$\sum_{d=1}^{r}\left(-\frac{1}{p}\right)^d \sum_{1\leqslant i_1<\cdots<i_d\leqslant r}p^d=\sum_{d=1}^{r}\binom{r}{d}(-1)^d=(1-1)^r-1=-1,$$
we obtain the second assertion.

In summary, we have proven 
the formulas in the theorem for $c\in \mathbb{N}_{>1}$ with $(c,p)=1$.
It should be noted that each coefficient of the left-hand side of each formula is 
expressed as a polynomial in $c$.
On the other hand, each coefficient of the right-hand side of each formula is 
continuous in $c$, by Theorem \ref{continuity theorem}.
Therefore, 
these identities can be extended 
to $c\in\mathbb Z_p^\times$.
\end{proof}


We consider the case $r=2$, with $(\gamma_1,\gamma_2)=(1,\eta)$ for $\eta\in \mathbb{Z}_p$, 
and $c\in \mathbb{N}_{>1}$ with $(c,p)=1$. 
From \eqref{Eur-exp1} and Example \ref{C-main-1}, we have 
\begin{align*}
& L_{p,2}(-k_1,-k_2;\omega^{k_1},\omega^{k_2};1,\eta;c)\\
& =\sum_{\xi_1^c=1 \atop \xi_1\not=1}\sum_{\xi_2^c=1 \atop \xi_2\not=1}\sum_{\nu=0}^{k_2}\binom{k_2}{\nu}\aa_{k_1+\nu}(\xi_1)\aa_{k_2-\nu}(\xi_2)\eta^{k_2-\nu}\\
& \ -\frac{1}{p} \sum_{\rho_{1}^p=1}\sum_{\xi_1^c=1 \atop \xi_1\not=1}\sum_{\xi_2^c=1 \atop \xi_2\not=1}\sum_{\nu=0}^{k_2}\binom{k_2}{\nu}\aa_{k_1+\nu}(\rho_{1}\xi_1)\aa_{k_2-\nu}(\xi_2)\eta^{k_2-\nu}\\
& \ -\frac{1}{p} \sum_{\rho_{2}^p=1}\sum_{\xi_1^c=1 \atop \xi_1\not=1}\sum_{\xi_2^c=1 \atop \xi_2\not=1}\sum_{\nu=0}^{k_2}\binom{k_2}{\nu}\aa_{k_1+\nu}(\rho_{2}\xi_1)\aa_{k_2-\nu}(\rho_{2}^{\eta}\xi_2)\eta^{k_2-\nu}\\
& \ +\frac{1}{p^2} \sum_{\rho_{1}^p=1}\sum_{\rho_{2}^p=1}\sum_{\xi_1^c=1 \atop \xi_1\not=1}\sum_{\xi_2^c=1 \atop \xi_2\not=1}\sum_{\nu=0}^{k_2}\binom{k_2}{\nu}\aa_{k_1+\nu}(\rho_{1}\rho_2\xi_1)\aa_{k_2-\nu}(\rho_{2}^{\eta}\xi_2)\eta^{k_2-\nu}
\end{align*}
for $k_1,k_2\in \mathbb{N}_0$. 
By the argument similar to 
\eqref{2-0-1}, we have
\begin{equation}
\sum_{\xi^c=1 \atop \xi\not=1} \aa_n(\alpha \xi)=
\begin{cases}
\left(1-c^{n+1}\right)\frac{B_{n+1}}{n+1} & (\alpha=1)\\
c^{n+1}\aa_{n}(\alpha^c)-\aa_n(\alpha) & (\alpha\not=1)
\end{cases}
\label{p-xi-0}
\end{equation}
for $n\in \mathbb{N}_0$.
In addition, 
using \eqref{log-der} with $k=p$, we have 
\begin{equation}
\sum_{\rho^p=1}\sum_{\xi^c=1 \atop \xi\not=1}\aa_n(\rho\xi)=\sum_{\rho^p=1}\left(c^{n+1}\aa_{n}(\rho^c)-\aa_{n}(\rho)\right)=p^{n+1}\left(1-c^{n+1}\right)\frac{B_{n+1}}{n+1}\label{p-xi}
\end{equation}
for $n\in \mathbb{N}_0$. 
Hence, by 
continuity with respect to $c$, we have the following.

\begin{example} \label{C-DZV-1}
For $k_1,k_2\in \mathbb{N}_0$, $\eta\in \mathbb{Z}_p$ and $c\in \mathbb{Z}_p^\times$,
\begin{align}
& L_{p,2}(-k_1,-k_2;\omega^{k_1},\omega^{k_2};1,\eta;c)\notag\\
& = \sum_{\nu=0}^{k_2}\binom{k_2}{\nu}\left(1-p^{k_1+\nu}\right)\left(1-c^{k_1+\nu+1}\right)\frac{B_{k_1+\nu+1}}{k_1+\nu+1}\left(1-c^{k_2-\nu+1}\right)\frac{B_{k_2-\nu+1}}{k_2-\nu+1}\eta^{k_2-\nu}\notag\\
& \ -\frac{1}{p} \sum_{\rho_{2}^p=1}\sum_{\nu=0}^{k_2}\binom{k_2}{\nu}\left(c^{k_1+\nu+1}\aa_{k_1+\nu}(\rho_{2}^c)-\aa_{k_1+\nu}(\rho_{2})\right)\left(c^{k_2-\nu+1}\aa_{k_2-\nu}(\rho_{2}^{c\eta})-\aa_{k_2-\nu}(\rho_{2}^{\eta})\right)\eta^{k_2-\nu}\notag\\
& \ +\frac{1}{p}\sum_{\rho_{2}^p=1}\sum_{\nu=0}^{k_2}\binom{k_2}{\nu}p^{k_1+\nu}\left(1-c^{k_1+\nu+1}\right)\frac{B_{k_1+\nu+1}}{k_1+\nu+1}\left(c^{k_2-\nu+1}\aa_{k_2-\nu}(\rho_{2}^{c\eta})-\aa_{k_2-\nu}(\rho_{2}^{\eta})\right)\eta^{k_2-\nu}. \label{DV-01}
\end{align}
In particular, when $\eta\in p\mathbb{Z}_p$, 
then by \eqref{Ber-01} and \eqref{p-xi},
\begin{align}
& L_{p,2}(-k_1,-k_2;\omega^{k_1},\omega^{k_2};1,\eta;c)\notag\\
& =\sum_{\nu=0}^{k_2}\binom{k_2}{\nu}\left(1-p^{k_1+\nu}\right)\left(1- c^{k_1+\nu+1} \right)\left( 1-c^{k_2-\nu+1}\right)\frac{B_{k_1+\nu+1}B_{k_2-\nu+1}}{(k_1+\nu+1)(k_2-\nu+1)}\eta^{k_2-\nu}. \label{L2-val}
\end{align}
Note that the case $c=2$ was 
obtained in our previous paper \cite[Section 4]{KMT-IJNT}.
In the special case when $k_1\in \mathbb{N}$, $k_2\in \mathbb{N}_0$ and $k_1+k_2$ is odd, 
\begin{equation}
L_{p,2}(-k_1,-k_2;\omega^{k_1},\omega^{k_2};1,\eta;c)=\displaystyle{\frac{c-1}{2}}\left( 1-c^{k_1+k_2+1}\right)\left(1-p^{k_1+k_2}\right)\frac{B_{k_1+k_2+1}}{k_1+k_2+1}. \label{odd-val}
\end{equation}
\end{example}

Furthermore, computing the coefficient of $t_1^{n_1}t_2^{n_2}t_3^{n_3}$ in \eqref{generating} with $r=3$, we obtain the following:

\begin{example}
For $k_1,k_2,k_3\in \mathbb{N}_0$, $\eta_1,\eta_2\in p\mathbb{Z}_p$ 
and $c\in \mathbb{Z}_p^\times$, 
\begin{align}
& L_{p,3}(-k_1,-k_2,-k_3;\omega^{k_1},\omega^{k_2},\omega^{k_3};1,\eta_1,\eta_2;c)\notag\\
& =\sum_{\nu_1=0}^{k_3}\sum_{\nu_2=0}^{k_3-\nu_1}\sum_{\kappa=0}^{k_2}\binom{k_2}{\kappa}\binom{k_3}{\nu_1,\,\nu_2}\left(1-p^{k_1+\nu_2+\kappa}\right)\left(1-c^{k_1+\nu_2+\kappa+1}\right)\left(1-c^{k_2+\nu_1-\kappa+1}\right)\notag\\
& \qquad \times \left(1-c^{k_3-\nu_1-\nu_2+1}\right)\frac{B_{k_1+\nu_2+\kappa+1}}{k_1+\nu_2+\kappa+1}\frac{B_{k_2+\nu_1-\kappa+1}}{k_2+\nu_1-\kappa+1}\frac{B_{k_3-\nu_1-\nu_2+1}}{k_3-\nu_1-\nu_2+1}\notag\\
& \qquad \times 
\eta_1^{k_2+\nu_1-\kappa}\eta_2^{k_3-\nu_1-\nu_2},\label{tri-comv}
\end{align}
where $\binom{N}{\nu_1,\,\nu_2}=\frac{N!}{\nu_1!\ \nu_2!\ (N-\nu_1-\nu_2)!}$.
\end{example}

\subsection{Multiple Kummer congruences
}\label{Kummer}

We will show certain congruence relations 
for the 
twisted multiple Bernoulli numbers,
which are certain generalizations of 
Kummer congruences for Bernoulli numbers (see \eqref{ord-Kummer}, below), from the viewpoint of $p$-adic multiple $L$-functions 
(Theorem \ref{Th-Kummer}).
Then we will extract specific congruences for 
ordinary Bernoulli numbers
in the depth $2$ case (Example \ref{DKC}). 

First, we recall the ordinary Kummer congruences. By \eqref{1-p-LF-gamma}, we know from \cite[p.\,31]{Kob} that 
\begin{equation}
L_{p,1}(1-m;\omega^{m-1};1;c) \equiv L_{p,1}(1-n;\omega^{n-1};1;c)\quad ({\rm mod\ }p^l) \label{L-K-r-1}
\end{equation}
for $m,n \in \mathbb{N}_0$ and $l\in \mathbb{N}$ 
such that $m \equiv n\ ({\rm mod\ }(p-1)p^{l-1})$ and $c\in \mathbb{N}_{>1}$. Therefore, from \eqref{KL-LP}, we see that 
\begin{equation}
(1-c^m)(1-p^{m-1})\frac{B_m}{m}\equiv (1-c^n)(1-p^{n-1})\frac{B_n}{n}\quad ({\rm mod\ }p^l). \label{Ord-KC}
\end{equation}
Note that $c\in \mathbb{N}_{>1}$ can be chosen arbitrarily under the condition $(c,p)=1$. 
In particular, when $p>2$,  
for even positive integers $m$ and $n$ 
such that $m \equiv n\ ({\rm mod\ }(p-1)p^{l-1})$ and $n \not\equiv 0\ ({\rm mod\ }p-1)$, 
we can choose $c$ so that 
$1-c^m \equiv 1-c^n\ ({\rm mod\ }p^l)$ and $(1-c^m,p)=(1-c^n,p)=1$, 
which yields the ordinary Kummer congruences (see \cite[p.\,32]{Kob}):
\begin{equation}
(1-p^{m-1})\frac{B_m}{m}\equiv (1-p^{n-1})\frac{B_n}{n}\quad ({\rm mod\ }p^l). \label{ord-Kummer}
\end{equation}

As a multiple analogue of \eqref{Ord-KC}, we have the following:

\begin{theorem}[{\bf Multiple Kummer congruences}]
\label{Th-Kummer}
Let $m_1,\ldots,m_r,n_1,\ldots,n_r\in \mathbb{N}_0$ with 
$m_j\equiv n_j$\ $\text{\rm mod}\ (p-1)p^{l_j-1}$ 
for $l_j\in \mathbb{N}$ $(1\leqslant j\leqslant r)$. 
Then, for $\gamma_1,\ldots,\gamma_r\in \mathbb{Z}_p$ 
and $c\in \mathbb{N}_{>1}$ with $(c,p)=1$, 
\begin{align}
& L_{p,r}(-m_1,\ldots,-m_r;\omega^{m_1},\ldots,\omega^{m_r};\gamma_1,\ldots,\gamma_{r};c)\notag\\
& \equiv L_{p,r}(-n_1,\ldots,-n_r;\omega^{n_1},\ldots,\omega^{n_r};\gamma_1,\ldots,\gamma_{r};c)\quad  \left({\rm mod}\  p^{\min \{l_j\,|\,{1\leqslant j\leqslant r}\}}\right). \label{L-val}
\end{align}
In other words, 
\begin{align}
& \sum_{\xi_1^c=1 \atop \xi_1\not=1}\cdots\sum_{\xi_r^c=1 \atop \xi_r\not=1}\aa((m_j);(\xi_j);(\gamma_j)) \notag\\
& +\sum_{d=1}^{r}\left(-\frac{1}{p}\right)^d \sum_{1\leqslant i_1<\cdots<i_d\leqslant r}\sum_{\rho_{i_1}^p=1}\cdots\sum_{\rho_{i_d}^p=1}\sum_{\xi_1^c=1 \atop \xi_1\not=1}\cdots\sum_{\xi_r^c=1 \atop \xi_r\not=1}\aa((m_j);((\prod_{j\leqslant i_l}\rho_{i_l})^{\gamma_j}\xi_j);(\gamma_j))\notag\\
& \equiv \sum_{\xi_1^c=1 \atop \xi_1\not=1}\cdots\sum_{\xi_r^c=1 \atop \xi_r\not=1}\aa((n_j);(\xi_j);(\gamma_j)) \notag\\
& \qquad +\sum_{d=1}^{r}\left(-\frac{1}{p}\right)^d \sum_{1\leqslant i_1<\cdots<i_d\leqslant r}\sum_{\rho_{i_1}^p=1}\cdots\sum_{\rho_{i_d}^p=1}\sum_{\xi_1^c=1 \atop \xi_1\not=1}\cdots\sum_{\xi_r^c=1 \atop \xi_r\not=1}\aa((n_j);((\prod_{j\leqslant i_l}\rho_{i_l})^{\gamma_j}\xi_j);(\gamma_j))\notag\\
& \qquad \qquad \left({\rm mod}\  p^{\min \{l_j\,|\,{1\leqslant j\leqslant r}\}}\right). \label{Kummer-main}
\end{align}
\end{theorem}

\begin{proof}
The proof works in the same way as 
the case of $L_p(s;\chi)$ stated above (for the details, see \cite[Chapter 2,\,Section\,3]{Kob}).
As for the integrand of \eqref{integration},
we have
\begin{equation}
\left(\sum_{\nu=1}^{j}x_{\nu}\gamma_{\nu}\right)^{m_j}\equiv \left(\sum_{\nu=1}^{j}x_{\nu}\gamma_{\nu}\right)^{n_j}\quad \left(\textrm{mod\ } p^{l_j}\right)\quad (1\leqslant j\leqslant r) \label{MKC1}
\end{equation}
for $(x_j)\in \left(\mathbb{Z}_p^r\right)'_{\{\gamma_j\}}$. 
Hence, \eqref{MKC1} 
implies 
\eqref{L-val}. 
It follows from Theorem \ref{T-main-1} that \eqref{Kummer-main} holds. 
\end{proof}

\begin{remark}
In the case $p=2$, since 
$$\left(\mathbb{Z}/2^l\mathbb{Z}\right)^\times\simeq \mathbb{Z}/2\mathbb{Z} \times \mathbb{Z}/2^{l-2}\mathbb{Z}\quad (l\in \mathbb{N}_{>2}),$$
\eqref{MKC1} holds under the condition
\begin{equation}
m_j\equiv n_j\ \text{\rm mod}\ 2^{l_j-2}\ \ (l_j\in \mathbb{N}_{>2};\ 1\leqslant j\leqslant r), \label{2-condition}
\end{equation}
and 
so does \eqref{L-val} under \eqref{2-condition}. 
Hence the following examples also hold under \eqref{2-condition} in the case $p=2$.
\end{remark}




In the case $r=1$, Theorem \ref{Th-Kummer} is nothing but \eqref{L-K-r-1},
a reformulation of \eqref{Ord-KC}.

In the case $r=2$, we obtain the following: 

\begin{example}\label{Cor-Kummer-2}
For $m_1,m_2,n_1,n_2\in \mathbb{N}_0$ with $m_j\equiv n_j$\ $\text{\rm mod}\ (p-1)p^{l_j-1}$ $(l_1,l_2\in \mathbb{N})$, $\eta\in \mathbb{Z}_p$ and $c\in \mathbb{Z}_p^\times$, 
\begin{align}
& \sum_{\nu=0}^{m_2}\binom{m_2}{\nu}\left(1-p^{m_1+\nu}\right)\left(1-c^{m_1+\nu+1}\right)\frac{B_{m_1+\nu+1}}{m_1+\nu+1}\left(1-c^{m_2-\nu+1}\right)\frac{B_{m_2-\nu+1}}{m_2-\nu+1}\eta^{m_2-\nu}\notag\\
& \ -\frac{1}{p} \sum_{\rho_{2}^p=1}\sum_{\nu=0}^{m_2}\binom{m_2}{\nu}\left(c^{m_1+\nu+1}\aa_{m_1+\nu}(\rho_{2}^c)-\aa_{m_1+\nu}(\rho_{2})\right)\left(c^{m_2-\nu+1}\aa_{m_2-\nu}(\rho_{2}^{c\eta})-\aa_{m_2-\nu}(\rho_{2}^{\eta})\right)\eta^{m_2-\nu}\notag\\
& \ +\frac{1}{p}\sum_{\rho_{2}^p=1}\sum_{\nu=0}^{m_2}\binom{m_2}{\nu}p^{m_1+\nu}\left(1-c^{m_1+\nu+1}\right)\frac{B_{m_1+\nu+1}}{m_1+\nu+1}\left(c^{m_2-\nu+1}\aa_{m_2-\nu}(\rho_{2}^{c\eta})-\aa_{m_2-\nu}(\rho_{2}^{\eta})\right)\eta^{m_2-\nu}\notag\\
&\ \equiv \sum_{\nu=0}^{n_2}\binom{n_2}{\nu}\left(1-p^{n_1+\nu}\right)\left(1-c^{n_1+\nu+1}\right)\frac{B_{n_1+\nu+1}}{n_1+\nu+1}\left(1-c^{n_2-\nu+1}\right)\frac{B_{n_2-\nu+1}}{n_2-\nu+1}\eta^{n_2-\nu}\notag\\
& \ -\frac{1}{p} \sum_{\rho_{2}^p=1}\sum_{\nu=0}^{n_2}\binom{n_2}{\nu}\left(c^{n_1+\nu+1}\aa_{n_1+\nu}(\rho_{2}^c)-\aa_{n_1+\nu}(\rho_{2})\right)\left(c^{n_2-\nu+1}\aa_{n_2-\nu}(\rho_{2}^{c\eta})-\aa_{n_2-\nu}(\rho_{2}^{\eta})\right)\eta^{n_2-\nu}\notag\\
& \ +\frac{1}{p}\sum_{\rho_{2}^p=1}\sum_{\nu=0}^{n_2}\binom{n_2}{\nu}p^{n_1+\nu}\left(1-c^{n_1+\nu+1}\right)\frac{B_{n_1+\nu+1}}{n_1+\nu+1}\left(c^{n_2-\nu+1}\aa_{n_2-\nu}(\rho_{2}^{c\eta})-\aa_{n_2-\nu}(\rho_{2}^{\eta})\right)\eta^{n_2-\nu}\notag\\
& \qquad \qquad \left({\rm mod}\  p^{\min \{l_1,l_2 \}}\right),\label{gen-Kummer}
\end{align}
which follows 
by setting $(\gamma_1,\gamma_2)=(1,\eta)$ for $\eta\in \mathbb{Z}_p$ in \eqref{L-val} and using Example \ref{C-DZV-1}. 
\end{example}

Next we consider certain congruences among ordinary Bernoulli numbers: 

\begin{example}\label{DKC}
We set $\eta=p$ and $c\in \mathbb{N}_{>1}$ with $(c,p)=1$
in \eqref{gen-Kummer}. 
Note that this case was already calculated in \eqref{L2-val} with $\eta=p$. 
Hence, 
we similarly obtain the following  {\bf double Kummer congruence}
for ordinary Bernoulli numbers:
\begin{align}
& \sum_{\nu=0}^{m_2}\binom{m_2}{\nu}\left(1-p^{m_1+\nu}\right)\left(1- c^{m_1+\nu+1} \right)\left( 1-c^{m_2-\nu+1}\right)\frac{B_{m_1+\nu+1}B_{m_2-\nu+1}}{(m_1+\nu+1)(m_2-\nu+1)}p^{m_2-\nu}\notag \\
& \equiv \sum_{\nu=0}^{n_2}\binom{n_2}{\nu}\left(1-p^{n_1+\nu}\right)\left(1- c^{n_1+\nu+1} \right)\left( 1-c^{n_2-\nu+1}\right)\frac{B_{n_1+\nu+1}B_{n_2-\nu+1}}{(n_1+\nu+1)(n_2-\nu+1)}p^{n_2-\nu}\notag \\
& \qquad \left({\rm mod}\  p^{\min \{l_1,l_2 \}}\right) \label{Kum-mod-d1}
\end{align}
for $m_1,m_2,n_1,n_2\in \mathbb{N}_0$ with $m_j\equiv n_j$\ $\text{\rm mod}\ (p-1)p^{l_j-1}$ $(l_j\in \mathbb{N};\ j=1,2)$.
In particular, when $m_1$ and $m_2$ are of different parity, 
we have from \eqref{odd-val} that 
\begin{align}
&\displaystyle{\frac{c-1}{2}}\left( 1-c^{m_1+m_2+1}\right)\left(1-p^{m_1+m_2}\right)\frac{B_{m_1+m_2+1}}{m_1+m_2+1}\notag \\
& \equiv \displaystyle{\frac{c-1}{2}}\left( 1-c^{n_1+n_2+1}\right)\left(1-p^{n_1+n_2}\right)\frac{B_{n_1+n_2+1}}{n_1+n_2+1}\ \ \left({\rm mod}\  p^{\min \{l_1,l_2 \}}\right),  \label{Kum-mod-d2}
\end{align}
which is equivalent to \eqref{Ord-KC} in the case $(c,p)=1$ with $p^2\nmid (c-1)$. 
Therefore, choosing $c$ suitably, we obtain the ordinary Kummer congruence \eqref{ord-Kummer}. 
From this observation, we can regard \eqref{Kum-mod-d1} as a certain general form of 
Kummer congruence including \eqref{ord-Kummer}. 
We do not know whether 
our \eqref{Kum-mod-d1} is a 
new congruence for 
Bernoulli numbers, or 
whether 
it follows from the ordinary Kummer congruence \eqref{Ord-KC}.
\end{example}

In the case $r=3$, we obtain the following:

\begin{example}\label{triple-KC}
Using \eqref{tri-comv} with $(\gamma_1,\gamma_2,\gamma_3)=(1,p,p)$, we can similarly obtain the following {\bf triple Kummer congruence}
for ordinary Bernoulli numbers:
\begin{align}
& \sum_{\nu_1=0}^{m_3}\sum_{\nu_2=0}^{m_3-\nu_1}\sum_{\kappa=0}^{m_2}\binom{m_2}{\kappa}\binom{m_3}{\nu_1\ \nu_2}\left(1-p^{m_1+\nu_2+\kappa}\right)\left(1-c^{m_1+\nu_2+\kappa+1}\right)\left(1-c^{m_2+\nu_1-\kappa+1}\right)\notag\\
& \quad \times \left(1-c^{m_3-\nu_1-\nu_2+1}\right)\frac{B_{m_1+\nu_2+\kappa+1}}{m_1+\nu_2+\kappa+1}\frac{B_{m_2+\nu_1-\kappa+1}}{m_2+\nu_1-\kappa+1}\frac{B_{m_3-\nu_1-\nu_2+1}}{m_3-\nu_1-\nu_2+1}p^{m_2+m_3-\kappa-\nu_2}\notag\\
& \equiv \sum_{\nu_1=0}^{n_3}\sum_{\nu_2=0}^{n_3-\nu_1}\sum_{\kappa=0}^{n_2}\binom{n_2}{\kappa}\binom{n_3}{\nu_1\ \nu_2}\left(1-p^{n_1+\nu_2+\kappa}\right)\left(1-c^{n_1+\nu_2+\kappa+1}\right)\left(1-c^{n_2+\nu_1-\kappa+1}\right)\notag\\
& \qquad \times \left(1-c^{n_3-\nu_1-\nu_2+1}\right)\frac{B_{n_1+\nu_2+\kappa+1}}{n_1+\nu_2+\kappa+1}\frac{B_{n_2+\nu_1-\kappa+1}}{n_2+\nu_1-\kappa+1}\frac{B_{n_3-\nu_1-\nu_2+1}}{n_3-\nu_1-\nu_2+1}p^{n_2+n_3-\kappa-\nu_2}\notag\\
& \qquad \left({\rm mod}\  p^{\min \{l_1,l_2,l_3 \}}\right) \label{Kum-mod-triple}
\end{align}
for $m_1,m_2,m_3,n_1,n_2,n_3\in \mathbb{N}_0$ with $m_j\equiv n_j$\ $\text{\rm mod}\ (p-1)p^{l_j-1}$ $(l_j\in \mathbb{N};\ j=1,2,3)$. It is also unclear whether this is new or follows from \eqref{Ord-KC}.
\end{example}

\subsection{Functional relations for $p$-adic multiple $L$-functions}\label{Funct-rel}
In this subsection, we will prove certain functional relations with a parity condition
among $p$-adic multiple $L$-functions (Theorem \ref{T-6-1}).
They are extensions of the vanishing property of the
Kubota-Leopoldt $p$-adic $L$-functions with odd characters
in the single variable case (cf. Proposition \ref{Prop-zero} and Example \ref{Rem-zero-2}), 
as well as 
the functional relations among
$p$-adic double $L$-functions proved by  \cite{KMT-IJNT}
in the double variable case under the condition $c=2$
(cf. Example \ref{P-Func-eq-1}). 
They may also 
be regarded as certain $p$-adic analogues of the parity result for MZVs (see Remark \ref{Rem-Parity-result}). 
It should be noted that the following functional relations 
are peculiar to the $p$-adic case,
which are derived from the relation \eqref{parity-2}, below, among Bernoulli numbers, 
and cannot be observed in the complex case.

Let $r \geqslant 2$. 
For each $J\subset\{1,\ldots,r\}$ with $1\in J$, we write 
$$J=\{ j_1(J),j_2(J),\ldots,j_{|J|}(J)\}\quad (1=j_1(J)<\cdots<j_{|J|}(J)),$$
where $|J|$ implies the number of elements of $J$. 
In addition, we formally let $j_{|J|+1}(J)=r+1$.

\begin{theorem}\label{T-6-1}
For $r\in \mathbb{N}_{>1}$, let $k_1,\ldots,k_r\in \mathbb{Z}$ with $k_1+\cdots+k_r\not\equiv r\pmod 2$, $\gamma_{1}\in \mathbb{Z}_p$, $\gamma_2,\ldots,\gamma_{r}\in p\mathbb{Z}_p$ and $c\in \mathbb{Z}_p^\times$. Then, for $s_1,\ldots,s_r\in \mathbb{Z}_p$, 
\begin{align}
& L_{p,r} (s_1,\ldots,s_r;\omega^{k_1},\ldots,\omega^{k_r};\gamma_1,\gamma_2,\ldots,\gamma_{r};c)\notag\\
& \ =-\sum_{J\subsetneq\{1,\ldots,r\} \atop 1\in J}
    \Bigr(\frac{1-c}{2}\Bigl)^{r-|J|} L_{p,|J|}\Bigl(\Bigl({\sum}_\mu s_l\Bigr)_{\mu=1}^{|J|};\bigl(\omega^{{\sum}_\mu k_l}\bigr)_{\mu=1}^{|J|};(\gamma_{j})_{j\in J};c\Bigr), \label{main-1}
\end{align}
where the symbol ${\sum}_\mu$ means the
sum over $j_\mu(J)\leqslant l<j_{\mu+1}(J)$.
\end{theorem}

To prove Theorem \ref{T-6-1}, we first define certain multiple 
analogues of Bernoulli numbers, and discuss their properties.

We consider $\widetilde{\mathfrak{H}}_1(t;\gamma;c)$ in $\mathbb{C}_p$, that is,  
\begin{equation}
\begin{split}
\widetilde{\mathfrak{H}}_1(t;\gamma;c)&=\frac{1}{e^{\gamma t}-1}-\frac{c}{e^{c\gamma t}-1}\\
&=\sum_{m=0}^\infty \left(1-c^{m+1}\right)B_{m+1}\frac{(\gamma t)^m}{(m+1)!}\\
&=\frac{c-1}{2}+\sum_{m=1}^\infty \left(c^{m+1}-1\right)\zeta(-m)\frac{(\gamma t)^m}{m!}
\end{split}
\label{3-3-2}
\end{equation}
for $c\in \mathbb{Z}_{p}^\times$ and $\gamma\in \mathbb{Z}_p$. 
For simplicity, we set 
\begin{equation}
\bb (n;\gamma;c):=\left(1-c^{n+1}\right)\frac{B_{n+1}}{n+1}\gamma^{n}\quad (n \in \mathbb{N}). \label{Ber-def}
\end{equation}
It follows from \eqref{2-1} and \eqref{2-1-2} that
\begin{equation}
\widetilde{\mathfrak{H}}_1(t;\gamma;c)=\sum_{m=0}^\infty \int_{\mathbb{Z}_p} x^{m}d\mm(x) \frac{(\gamma t)^m}{m!}
=\int_{\mathbb{Z}_p} e^{x\gamma t} d\mm(x),\label{3-4}
\end{equation}
\begin{equation}
\begin{split}
  \gh_1(t;\gamma;c):& =\widetilde{\mathfrak{H}}_1(t;\gamma;c)-\frac{c-1}{2}=\sum_{m=1}^\infty \bb (m;\gamma;c)\frac{t^m}{m!}\\
  &=\int_{\mathbb{Z}_p}(e^{x\gamma t}-1)d\widetilde{\mathfrak{m}}_c(x). 
\end{split}
\label{e-3-5}
\end{equation}
We know that $B_{2m+1}=0$, namely $\bb ({2m};\gamma;c)=0$ for $m\in \mathbb{N}$. Hence we have
\begin{equation}
\gh_1(-t;\gamma;c)=-\gh_1(t;\gamma;c). \label{e-3-6}
\end{equation}
For $r\in \mathbb{N}$ 
and $\gamma_1,\ldots,\gamma_r\in \mathbb{Z}_p$, we set
\begin{equation*}
  \begin{split}
    \gh_r(t_1,\ldots,t_r;\gamma_1,\ldots,\gamma_r;c)
    &:=
    \prod_{j=1}^{r} \gh_1\Bigl( \sum_{k=j}^r t_k;\gamma_j;c\Bigr)
    \\
    &\left(=
    \prod_{j=1}^{r} \left\{\widetilde{\mathfrak{H}}_1\Bigl( \sum_{k=j}^r t_k;\gamma_j;c\Bigr)-\frac{c-1}{2}\right\}\right),
\end{split}
\end{equation*}
and define the multiple analogues of Bernoulli numbers $\bb (n_1,\ldots,n_r;\gamma_1,\ldots,\gamma_r;c)$ by the Taylor expansion 
\begin{equation}
\label{Ber-def-r}
\begin{split}
  \gh_r(t_1,\ldots,t_r;\gamma_1,\ldots,\gamma_r;c)
  &=
    \sum_{n_1=1}^\infty
    \cdots
    \sum_{n_r=1}^\infty
    \bb (n_1,\ldots,n_r;\gamma_1,\ldots,\gamma_r;c)
    \frac{t_1^{n_1}}{n_1!}
    \cdots
    \frac{t_r^{n_r}}{n_r!}.
  \end{split}
\end{equation}
By \eqref{e-3-6}, we have
\begin{equation}
  \gh_r(-t_1,\ldots,-t_r;\gamma_1,\ldots,\gamma_r;c)
  =(-1)^r
  \gh_r(t_1,\ldots,t_r;\gamma_1,\ldots,\gamma_r;c).
\label{parity-1}
\end{equation}
Hence, when $n_1+\cdots+n_r\not\equiv r\pmod 2$, we obtain
\begin{equation}
  \bb (n_1,\ldots,n_r;\gamma_1,\ldots,\gamma_r;c)=0.
\label{parity-2}
\end{equation}

\begin{remark}
From \eqref{def-tilde-H}, we have
\begin{align}
    & \widetilde{\mathfrak{H}}_r((t_j);(\gamma_j);c) 
    =
    \prod_{j=1}^{r} \widetilde{\mathfrak{H}}_1\Bigl(\sum_{k=j}^r t_k;\gamma_j;c\Bigr). \label{Ber-def-r-2}
\end{align}
Comparing \eqref{Ber-def-r} with \eqref{Ber-def-r-2}, we see that $\gh_r((t_j);(\gamma_j);c)$ is defined as a slight modification of $\widetilde{\mathfrak{H}}_r((t_j);(\gamma_j);c)$, 
and it can be verified that a formula of the same form
as \eqref{parity-1} holds for $\widetilde{\mathfrak{H}}_r((t_j);(\gamma_j);c)$, as well.
\end{remark}

It follows from \eqref{Ber-def}, \eqref{e-3-5} and \eqref{Ber-def-r} that $\bb (n_1,\ldots,n_r;\gamma_1,\ldots,\gamma_r;c)$ can be expressed as a polynomial in $\{ B_{n}\}_{n\geqslant 1}$ and $\{\gamma_1,\ldots,\gamma_r\}$ with $\mathbb{Q}$-coefficients. 

\begin{lemma} \label{L-6-2}
For $n_1,\ldots,n_r\in \mathbb{N}$, $\gamma_1,\ldots,\gamma_r\in \mathbb{Z}_p$ and $c\in \mathbb{Z}_p^\times$, 
\begin{align}
\label{e-6-2}
  \bb (n_1,\ldots,n_r;\gamma_1,\ldots,\gamma_r;c)
    &=
    \sum_{J\subset\{1,\ldots,r\}\atop 1\in J}
    \Bigr(\frac{1-c}{2}\Bigl)^{r-|J|}
    \int_{\mathbb{Z}_p^{|J|}}
    \prod_{l=1}^r 
    (\sum_{\substack{j\in J\\j\leqslant l}}x_j \gamma_j)^{n_l}
    \prod_{j\in J} d\widetilde{\mathfrak{m}}_c(x_j),
\end{align}
where the empty product is interpreted as $1$.
\end{lemma}

\begin{proof}
From \eqref{e-3-5} and \eqref{Ber-def-r}, we have
\begin{align*}
  \gh_r&(t_1,\ldots,t_r;\gamma_1,\ldots,\gamma_r;c)\\
    &=
    \int_{\mathbb{Z}_p^r}\prod_{j=1}^r(e^{x_j \gamma_j (\sum_{l=j}^r t_l)}-1)\prod_{j=1}^r d\widetilde{\mathfrak{m}}_c(x_j)
    \\
    &=
    \int_{\mathbb{Z}_p^r}
    \sum_{J\subset\{1,\ldots,r\}}(-1)^{r-|J|}
    \exp\Bigl(\sum_{j\in J}x_j \gamma_j (\sum_{l=j}^r t_l)\Bigr)
    \prod_{j=1}^r d\widetilde{\mathfrak{m}}_c(x_j)
    \\
    &=
    \int_{\mathbb{Z}_p^r}
    \sum_{J\subset\{1,\ldots,r\}}(-1)^{r-|J|}
    \exp\Bigr(\sum_{l=1}^r t_l(\sum_{j\in J \atop j\leqslant l}x_j \gamma_j)\Bigl)
    \prod_{j=1}^r d\widetilde{\mathfrak{m}}_c(x_j)
    \\
    &=
    \sum_{n_1=0}^\infty
    \cdots
    \sum_{n_r=0}^\infty
    \sum_{J\subset\{1,\ldots,r\}}
    (-1)^{r-|J|}
    \int_{\mathbb{Z}_p^r}
    \prod_{l=1}^r 
    \frac{\left(t_l(\sum_{j\in J \atop j\leqslant l}x_j \gamma_j)\right)^{n_l}}{n_l!}
    \prod_{j=1}^r d\widetilde{\mathfrak{m}}_c(x_j)
    \\
    &=
    \sum_{n_1=0}^\infty
    \cdots
    \sum_{n_r=0}^\infty
    \sum_{J\subset\{1,\ldots,r\}}
    (-1)^{r-|J|}
    \int_{\mathbb{Z}_p^r}
    \prod_{l=1}^r 
    (\sum_{\substack{j\in J \\ j\leqslant l}}x_j \gamma_j)^{n_l}
    \prod_{j=1}^r d\widetilde{\mathfrak{m}}_c(x_j)
    \frac{t_1^{n_1}}{n_1!}
    \cdots
    \frac{t_r^{n_r}}{n_r!}.
\end{align*}
Here, we consider each coefficient of $\frac{t_1^{n_1}}{n_1!} \cdots \frac{t_r^{n_r}}{n_r!}$ for $n_1,\ldots,n_r\in \mathbb{N}$. If $1\not\in J$ then 
$$\sum_{j\in J \atop j\leqslant 1}x_j \gamma_j$$
is an empty sum, which implies $0$. Hence, we obtain from \eqref{2-1-2} that 
\begin{align*}
  \bb &(n_1,\ldots,n_r;\gamma_1,\ldots,\gamma_r;c)\\
    &=
    \sum_{J\subset\{1,\ldots,r\}}
    (-1)^{r-|J|}
    \int_{\mathbb{Z}_p^r}
    \prod_{l=1}^r 
    (\sum_{\substack{j\in J\\j\leqslant l}}x_j \gamma_j)^{n_l}
    \prod_{j=1}^r d\widetilde{\mathfrak{m}}_c(x_j)
    \\
    &=
    \sum_{J\subset\{1,\ldots,r\} \atop 1\in J}
    (-1)^{r-|J|}
    \int_{\mathbb{Z}_p^r}
    \prod_{l=1}^r 
    (\sum_{\substack{j\in J\\j\leqslant l}}x_j \gamma_j)^{n_l}
    \prod_{j=1}^r d\widetilde{\mathfrak{m}}_c(x_j)
    \\
    &=
    \sum_{J\subset\{1,\ldots,r\} \atop 1\in J}
    \Bigr(\frac{1-c}{2}\Bigl)^{r-|J|}
    \int_{\mathbb{Z}_p^{|J|}}
    \prod_{l=1}^r 
    (\sum_{\substack{j\in J\\j\leqslant l}}x_j \gamma_j)^{n_l}
    \prod_{j\in J} d\widetilde{\mathfrak{m}}_c(x_j)
\end{align*}
for $n_1,\ldots,n_r\in \mathbb{N}$, 
and this completes the proof.
\end{proof}

\begin{proof}[Proof of Theorem \ref{T-6-1}] 
If $\gamma_1\in p\mathbb{Z}_p$, then it follows from Remark \ref{Rem-gamma1} that all functions on 
both sides of \eqref{main-1} are zero-functions. Hence, \eqref{main-1} holds trivially, 
so we only consider the case 
$\gamma_1\in \mathbb{Z}_p^\times,\ \gamma_2,\ldots,\gamma_r\in p\mathbb{Z}_p$ 
in \eqref{e-6-2}. 

First, we consider the case $\gamma_1=1$. 
Setting $\gamma_1=1$ and $\gamma_2,\ldots,\gamma_r\in p\mathbb{Z}_p$ in \eqref{e-6-2}, 
we have
\begin{align*}
  \bb &(n_1,\ldots,n_r;\gamma_1,\gamma_2,\ldots,\gamma_r;c)\\
    &=
    \sum_{J\subset\{1,\ldots,r\} \atop 1\in J}
    \Bigr(\frac{1-c}{2}\Bigl)^{r-|J|}
    \int_{\mathbb{Z}_p^{|J|}}
    \prod_{l=1}^r 
    (\sum_{\substack{j\in J\\j\leqslant l}}x_j \gamma_j)^{n_l}
    \prod_{j\in J} d\widetilde{\mathfrak{m}}_c(x_j)
\end{align*}
for $n_1,\ldots,n_r\in \mathbb{N}$. From the condition $\gamma_1=1$ and $\gamma_2,\ldots,\gamma_r\in p\mathbb{Z}_p$, we 
see that
\begin{align*}
(\mathbb{Z}_p^{|J|})'_{\{\gamma_j\}_{j\in J}}&=\mathbb{Z}_p^\times \times \mathbb{Z}_p^{|J|-1}=\mathbb{Z}_p^{|J|}\setminus \left(p\mathbb{Z}_p \times \mathbb{Z}_p^{|J|-1}\right)
\end{align*}
for any $J\subset\{1,\ldots,r\}$ with $1\in J$. 
Therefore, we have
\begin{align}
&   \sum_{J\subset\{1,\ldots,r\} \atop 1\in J}    \Bigr(\frac{1-c}{2}\Bigl)^{r-|J|}L_{p,|J|}\Bigl(\Bigl(-{\sum}_{\mu}n_l\Bigr)_{\mu=1}^{|J|};\bigl(\omega^{{\sum}_{\mu}n_l}\bigr)_{\mu=1}^{|J|};(\gamma_{j})_{j\in J};c\Bigr)\notag\\
    &
    =
    \sum_{J\subset\{1,\ldots,r\}\atop 1\in J}
    \Bigr(\frac{1-c}{2}\Bigl)^{r-|J|}
    \int_{\mathbb{Z}_p^\times \times \mathbb{Z}_p^{|J|-1}}
    \prod_{l=1}^r 
    (\sum_{\substack{j\in J\\ j\leqslant l}}x_j \gamma_j)^{n_l}
    \prod_{j\in J} d\widetilde{\mathfrak{m}}_c(x_j)\notag\\
    &=
    \sum_{J\subset\{1,\ldots,r\} \atop 1\in J}
    \Bigr(\frac{1-c}{2}\Bigl)^{r-|J|}
    \int_{\mathbb{Z}_p^{|J|}}
    \prod_{l=1}^r 
    (\sum_{\substack{j\in J\\i\leqslant l}}x_j \gamma_j)^{n_l}
    \prod_{j\in J} d\widetilde{\mathfrak{m}}_c(x_j)\notag\\
&   \ \ -p^{\sum_{l=1}^{r}n_l}\sum_{J\subset\{1,\ldots,r\}\atop 1\in J}
    \Bigr(\frac{1-c}{2}\Bigl)^{r-|J|}
    \int_{\mathbb{Z}_p^{|J|}}
    \prod_{l=1}^r 
    (x_1+\sum_{\substack{j\in J\\1<j\leqslant l}}x_j \gamma_j/p)^{n_l}
    \prod_{j\in J} d\widetilde{\mathfrak{m}}_c(x_j)\notag\\
& =  \bb (n_1,\ldots,n_r;1,\gamma_2,\ldots,\gamma_r;c)
    -p^{\sum_{l=1}^{r}n_l}
     \bb (n_1,\ldots,n_r;1,\gamma_2/p,\ldots,\gamma_r/p;c), \label{eq-zero}
\end{align}
which is equal to $0$ when $n_1+\cdots+n_r\not\equiv r\pmod 2$ because of \eqref{parity-2}. 
Let $k_1,\ldots,k_r\in \mathbb{Z}$ with $k_1+\cdots+k_r\not\equiv r\pmod 2$. Then the above consideration 
implies that
\begin{align}
&   \sum_{J\subsetneq\{1,\ldots,r\}\atop 1\in J}
    \Bigr(\frac{1-c}{2}\Bigl)^{r-|J|}L_{p,|J|}\Bigl(\Bigl(-{\sum}_{\mu}n_l\Bigr)_{\mu=1}^{|J|};\bigl(\omega^{{\sum}_{\mu}k_l}\bigr)_{\mu=1}^{|J|};(\gamma_{j})_{j\in J};c\Bigr)\notag\\
& \qquad +L_{p,r}(-n_1,\ldots,-n_r;\omega^{k_1},\ldots,\omega^{k_r};\gamma_1,\gamma_2,\ldots,\gamma_{r};c)=0 \label{L-parity}
\end{align}
holds on 
$$S_{\{k_j\}}=\{(n_1,\ldots,n_r) \in \mathbb{N}^{r}\mid n_j\in k_j+(p-1)\mathbb{Z}\ (1\leqslant j \leqslant r)\}.$$
Since $S_{\{k_j\}}$ is dense in $\mathbb{Z}_p^r$, we obtain \eqref{main-1} in the case $\gamma_1=1$. 

Next we consider the case $\gamma_1\in \mathbb{Z}_p^\times$. 
We can easily check that if $\gamma_1\in \mathbb{Z}_p^\times$, 
\begin{align*}
& L_{p,r}((s_j);(\omega^{k_j});(\gamma_j);c) \\
& \quad  =\langle \gamma_1\rangle^{-s_1-\cdots -s_r} \omega^{k_1+\cdots +k_r}(\gamma_1) 
    \ L_{p,r}((s_j);(\omega^{k_j});(\gamma_j/\gamma_1);c),\\
& L_{p,|J|}\Bigl(\Bigl({\sum}_{\mu}s_l\Bigr)_{\mu=1}^{|J|};\bigl(\omega^{{\sum}_{\mu}k_l}\bigr)_{\mu=1}^{|J|};(\gamma_{j})_{j\in J};c\Bigr)  \\
& =\langle \gamma_1\rangle^{-s_1-\cdots -s_r} \omega^{k_1+\cdots +k_r}(\gamma_1) L_{p,|J|}\Bigl(\Bigl({\sum}_{\mu}s_l\Bigr)_{\mu=1}^{|J|};\bigl(\omega^{{\sum}_{\mu}k_l}\bigr)_{\mu=1}^{|J|};(\gamma_{j}/\gamma_1)_{j\in J};c\Bigr)
\end{align*}
for $J\subsetneq \{1,\ldots,r\}$ with $1\in J$. 
Since we already proved \eqref{main-1} for 
the case $\{\gamma_j/\gamma_1\}_{j=1}^{r}$, 
we consequently obtain \eqref{main-1} for 
$\gamma_1\in \mathbb{Z}_p^\times$ by multiplying 
both sides by
$$\langle \gamma_1\rangle^{-s_1-\cdots -s_r} \omega^{k_1+\cdots +k_r}(\gamma_1). $$
Thus we obtain the proof of Theorem \ref{T-6-1}.
\end{proof}

\begin{remark}\label{Rem-Parity-result}
Note that each $p$-adic multiple $L$-function appearing on the right-hand side of \eqref{main-1} is of depth lower  than $r$. 
The relation \eqref{main-1} reminds us of the \textit{parity result} for MZVs which implies that 
the MZV whose depth and weight are of different parity can be expressed as a polynomial in MZVs of lower depth with $\mathbb{Q}$-coefficients.
In fact, \eqref{L-parity} shows that 
$L_{p,r}((-n_j);(\omega^{k_j});(\gamma_j);c)$ can be expressed as a polynomial in $p$-adic multiple $L$-values of lower depth than $r$ at non-positive integers with $\mathbb{Q}$-coefficients, when $n_1+\cdots+n_r\not\equiv r\pmod 2$. This can be regarded as a $p$-adic version of the parity result for $p$-adic multiple $L$-values. 
By the density property, we obtain its continuous version, that is, the functional relation \eqref{main-1}.
\end{remark}

The following result corresponds to the case $r=1$ in Theorem \ref{T-6-1}.

\begin{example}\label{Rem-zero-2}
Setting $r=1$ in \eqref{eq-zero}, we obtain from Definition \ref{Def-Kubota-L} that 
\begin{align*}
    &
    \Bigr(\frac{1-c}{2}\Bigl)^{r-1}
    \int_{\mathbb{Z}_p^\times}
    x_1^{n_1}
    d\widetilde{\mathfrak{m}}_c(x_1)\left(=\Bigr(\frac{1-c}{2}\Bigl)^{r-1}
    \left(c^{n_1+1}-1\right)L_p(-n_1;\omega^{n_1+1})\right)\notag\\
& =  \bb (n_1;1;c)
    -p^{n_1}
     \bb (n_1;1;c)=0
\end{align*}
for $n_1\in \mathbb{N}$ such that 
$n_1\not\equiv 1\pmod 2$, namely $n_1+1$ is odd, 
due to \eqref{parity-2}. Hence, if $k$ is odd, 
then $L_p(-n_1;\omega^{k})=0$ for $n_1\in \mathbb{N}$ such that $n_1+1\equiv k\ ({\rm mod}\ p-1)$. This implies 
$$
L_p(s,\omega^k)\equiv 0
$$
when $k$ is odd, the statement of Proposition \ref{Prop-zero}. Thus, we can regard Theorem \ref{T-6-1} as a multiple version of Proposition \ref{Prop-zero}.
\end{example}

Next, we consider the case $r=2$.

\begin{example} \label{P-Func-eq-1}
Let $k,l\in \mathbb{N}_0$ with $2\nmid (k+l)$,
$\gamma_1\in \mathbb{Z}_p^\times$, $\gamma_2\in p\mathbb{Z}_p$ and $c\in \mathbb{Z}_p^\times $. Then, for $s_1,s_2\in \mathbb{Z}_p$, we obtain from \eqref{main-1} and \eqref{1-p-LF-gamma} that 
\begin{align}
& L_{p,2}(s_1,s_2;\omega^{k},\omega^{l};\gamma_1,\gamma_2;c)\notag\\
& ={\frac{c-1}{2}}\langle \gamma_1\rangle^{-s_1-s_2} \omega^{k+l}(\gamma_1)\left( \langle c\rangle^{1-s_1-s_2}\omega^{k+l+1}(c)-1\right)L_p(s_1+s_2;\omega^{k+l+1}). \label{Func-eq-1}
\end{align}
Note that this functional relation in the case $p>2$, $c=2$ and $\gamma_1=1$ was already proved in \cite{KMT-IJNT} by a different method. 
\end{example}


\begin{remark}\label{R-3-2} 
As we noted above, when $k+l$ is even, $L_p(s;\omega^{k+l+1})$ is the zero-function, so is the right-hand side of \eqref{Func-eq-1}. On the other hand, even if $k+l$ is even, the left-hand side of \eqref{Func-eq-1} is not necessarily the zero-function. In fact, it follows from \eqref{L2-val} that
\begin{align*}
& L_{p,2}(-1,-1;\omega,\omega;1,\gamma_2;c)=\frac{(1-c^2)^2(1-p)}{4}B_2^2\gamma_2
\end{align*}
for $\gamma_2\in p\mathbb{Z}_p$, which does not vanish if $\gamma_2\not=0$, 
so $L_{p,2}(s_1,s_2;\omega,\omega;1,\gamma_2;c)$ 
is not the zero-function.
As a result, 
$L_{p,2}(s_1,s_2;\omega^{k},\omega^l;1,\gamma_2;c)$ 
seems to have more information 
than the Kubota-Leopoldt $p$-adic $L$-functions.
\end{remark}

Lastly, we consider the case $r=3$.

\begin{example}\label{E-6-3}
Let 
$k_1,k_2,k_3\in \mathbb{Z}$ with $2\mid (k_1+k_2+k_3)$, $\gamma_1\in \mathbb{Z}_p^\times$, $\gamma_2,\gamma_3\in p\mathbb{Z}_p$ and $c\in \mathbb{Z}_p^\times $. Then for $s_1,s_2,s_3\in \mathbb{Z}_p$, we obtain from \eqref{main-1} that 
\begin{align*}
& {L_{p,3}(s_1,s_2,s_3;\omega^{k_1},\omega^{k_2},\omega^{k_3};\gamma_1,\gamma_2,\gamma_3;c)} \notag\\
& \quad {=\frac{1-c}{2}L_{p,2}(s_1,s_2+s_3;\omega^{k_1},\omega^{k_2+k_3};\gamma_1,\gamma_2;c)} {+\frac{1-c}{2}L_{p,2}(s_1+s_2,s_3;\omega^{k_1+k_2},\omega^{k_3};\gamma_1,\gamma_3;c)} \notag\\
& \quad \quad 
-\left(\frac{1-c}{2}\right)^2\langle \gamma_1\rangle^{-s_1-s_2-s_3}\omega^{k_1+k_2+k_3}(\gamma_1)\\
& \qquad \quad 
\times \left(\langle c\rangle^{1-s_1-s_2-s_3}\omega^{k_1+k_2+k_3+1}(c)-1\right)L_{p}(s_1+s_2+s_3;\omega^{k_1+k_2+k_3+1}).
\end{align*}
Note that from \eqref{1-p-LF-gamma} the third term on the right-hand side vanishes because $L_p(s;\omega^{k})$ is the zero function when $k$ is odd. 
\end{example}

Using Theorem \ref{T-6-1} and Examples \ref{P-Func-eq-1} and \ref{E-6-3}, we can immediately obtain the following result by induction on $r\geqslant 2$. 

\begin{corollary}\label{C-6-4} 
Let $r\in \mathbb{N}_{\geqslant 2}$, $k_1,\ldots,k_r\in \mathbb{Z}$ with $k_1+\cdots+k_r\not\equiv r\pmod 2$, $\gamma_1\in \mathbb{Z}_p$, $\gamma_2,\ldots,\gamma_{r}\in p\mathbb{Z}_p$ and $c\in \mathbb{Z}_p^\times$. Then 
$$L_{p,r} (s_1,\ldots,s_r;\omega^{k_1},\ldots,\omega^{k_r};\gamma_1,\ldots,\gamma_{r};c)$$
can be expressed as a polynomial in $p$-adic $j$-ple $L$-functions for $j\in \{1,2,\ldots,r-1\}$ satisfying $j \not\equiv r \pmod 2$, with $\mathbb{Q}$-coefficients.
\end{corollary}

\

\section{Special values of $p$-adic multiple $L$-functions at positive integers}\label{sec-5}
In the 
previous section, particularly in Theorem \ref{T-main-1}, we saw that 
the special values of our $p$-adic multiple $L$-functions at non-positive integers
are expressed  in terms of the twisted multiple Bernoulli numbers (Definition \ref{Def-M-Bern}),
which are
the special values of the complex multiple 
zeta-functions of generalized Euler-Zagier-Lerch type 
at non-positive integers
(cf. \cite{FKMT-Desing}). 

In this section, we will discuss their special values 
at positive integers.
In the complex case, 
the special values 
of multiple zeta functions (the function \eqref{Barnes-Lerch} with  all 
$\xi_j$ and $\gamma_j$ are $1$)
at positive integers  
are given by the special values of multiple polylogarithms 
at unity (cf. \cite{FKMT-Desing}). 
To consider the $p$-adic case, 
we will introduce 
specific $p$-adic functions in 
each subsection.

Our main result is Theorem \ref{L-Li theorem-2}
where we will show a $p$-adic analogue of the equality \eqref{zeta=Li}. 
The result generalizes 
Coleman's work \cite{C}. 
We will establish in Theorem \ref{L-Li theorem-2} a close relationship of our $p$-adic multiple $L$-functions
with the $p$-adic TMSPLs
\footnote{
We remind that TMSPL stands for the twisted multiple star polylogarithm 
(see \S \ref{sec-1}).
},
which are star-variants of
generalizations of  $p$-adic  multiple polylogarithm
introduced by the first-named author \cite{Fu1, F2}
and Yamashita \cite{Y} 
for 
investigating 
the $p$-adic realization of certain motivic fundamental group.
It is achieved 
by showing that the special values of $p$-adic multiple $L$-functions at positive integers
are described by  
$p$-adic twisted multiple $L$-star values;
the special values at unity
of the $p$-adic TMSPLs (Definition \ref{Def-TMSPL}, see also \cite{Fu1,Y})
constructed by Coleman's $p$-adic iterated integration theory \cite{C}.
To relate 
$p$-adic multiple $L$-functions with $p$-adic TMSPLs,
we shall 
introduce $p$-adic rigid TMSPLs (Definition \ref{def of pMMPL})
and their partial versions (Definition \ref{def of pPMPL}) as intermediate objects
and investigate their several basic properties mainly in 
Subsections \ref{sec-5-3} and \ref{sec-5-4}. 

\subsection{$p$-adic rigid twisted multiple star polylogarithms} \label{sec-5-3}
In this 
subsection, we 
introduce $p$-adic rigid TMSPLs
(Definition \ref{def of pMMPL}), 
and 
give a description of special values of $p$-adic multiple $L$-functions at positive integers 
via special values of $p$-adic rigid TMSPLs at roots of unity
 (Theorem \ref{L-ell theorem}), 
which extends Coleman's result \eqref{Coleman equality}.

First, we briefly review the necessary 
basics of  rigid analysis in our specific case.

\begin{notation}[cf. \cite{BGR,FvP}]\label{rigid-basics}
\label{basics on rigid}
Let $\alpha_1, \dots, \alpha_n\in {\mathbb C}_p$, 
$\rs_1,\dots,\rs_n\in{\mathbb Q}_{>0}$, 
and $\rs_0\in{\mathbb Q}_{\geqslant 0}$.
The space
\begin{equation}\label{affinoid presentation}
X=\left\{z\in {\bf P}^1({\mathbb C}_p)\bigm| |z-\alpha_i|_p\geqslant \rs_i \ (i=1,\dots, n),
|z|_p\leqslant 1/\rs_0
\right\}
\end{equation}
is equipped with 
the structure of an {\it affinoid}, a special type of rigid analytic space. 
A {\it rigid analytic function} on $X$ is 
a function $f(z)$ on $X$
which admits  the convergent expansion
$$
f(z)=\sum_{m\geqslant 0}a_{m}(\infty;f)z^{m}
+\sum_{i=1}^n\sum_{m>0}\frac{a_{m}(\alpha_i;f)}{(z-\alpha_i)^{m}}
$$
with $\mathbb C_p$-coefficients.
The  expressions are 
unique 
(the Mittag-Leffler decompositions), 
which can be  shown from, for example, \cite[I.1.3]{FvP}.
We denote the algebra of rigid analytic functions on $X$ by $A^{\mathrm{rig}}(X)$.
\end{notation}

\begin{notation}
For $a$ in $\textbf{P}^{1}(\mathbb{C}_p)$, 
we let $\overline{a}={\rm red}(a)$, where 
the reduction map red is defined as 
$${\rm red}:\textbf{P}^{1}(\mathbb{C}_p) \to \textbf{P}^{1}({\overline{\mathbb{F}}_p})\left(=\overline{\mathbb{F}}_p\cup \{\overline{\infty}\}\right),$$
where $\overline{\mathbb{F}}_p$ is the algebraic closure of ${\mathbb{F}}_p$. 
For a finite subset $D\subset \textbf{P}^{1}(\mathbb{C}_p)$, we define $\overline{D}={\rm red}(D)\subset \textbf{P}^{1}(\overline{\mathbb{F}}_p)$.
For $a_0 \in \textbf{P}^{1}(\overline{\mathbb{F}}_p)$, 
we define 
its tubular neighborhood $]a_0[={\rm red}^{-1}(a_0)$. 
Namely, 
$]\overline{a}[=\{x\in \textbf{P}^{1}(\mathbb{C}_p)\,|\,|x-a|_p<1\}$
for $a\in {\mathbb{C}}_p$,  $]\overline{0}[={\frak M}_{\mathbb{C}_p}$
and $]\overline{\infty}[=\textbf{P}^{1}(\mathbb{C}_p)\setminus {\mathcal O}_{\mathbb{C}_p}$. 
For a finite subset 
$S\subset  \textbf{P}^{1}(\overline{\mathbb{F}}_p)$,
we define $]S[:={\rm red}^{-1}(S)\subset \textbf{P}^{1}(\mathbb{C}_p)$.
By abuse of notation, 
we denote $A^{\mathrm{rig}}({\bf P}^1({\mathbb C}_p)- ]S[)$
by $A^{\mathrm{rig}}({\bf P}^1\setminus S)$.
\end{notation}

We recall 
two fundamental properties of rigid analytic functions:
\begin{proposition}[{\cite[Chapter 6]{BGR}, etc.}]\label{two fundamental properties of rigid analytic functions}
Let $X$ be as in \eqref{affinoid presentation}.
Then the following holds.

{\rm (i)}\ {\it The coincidence principle}:
If two rigid analytic functions $f(z)$ and $g(z)$ coincide in a subaffinoid of $X$,
then  they coincide on the whole of $X$. 

{\rm (ii)}\ The algebra $A^{\mathrm{rig}}(X)$ forms a Banach algebra  with the supremum norm.
\end{proposition}

The following function plays the main role in this subsection:

\begin{definition}\label{def of pMMPL}
Let $n_1,\dots,n_r\in \mathbb{N}$ and $\xi_1,\dots,\xi_{r}\in\mathbb{C}_p$ with $|\xi_j|_p\leqslant 1$ ($1\leqslant j\leqslant r$).
The {\bf $p$-adic rigid TMSPL} 
$\ell^{(p),\star}_{n_1,\dots,n_r}(\xi_1,\dots,\xi_{r};z)$ is defined by the following 
$p$-adic power series:
\begin{equation}\label{series expression for ell}
\ell^{(p),\star}_{n_1,\dots,n_r}(\xi_1,\dots,\xi_{r};z):=
\underset{(k_1,p)=\cdots=(k_r,p)=1}{\underset{0<k_1\leqslant\cdots\leqslant k_{r}}{\sum}}
\frac{\xi_1^{k_1}\cdots\xi_r^{k_r}}{k_1^{n_1}\cdots k_r^{n_r}}z^{k_r},
\end{equation}
which converges for $z\in ]\bar{0}[=\{x\in\mathbb{C}_p\bigm| \ |x|_p<1\}$, 
by the assumption $|\xi_j|_p\leqslant 1$ for $1\leqslant j\leqslant r$.
\end{definition}

It will be proved 
that  it is rigid analytic 
(Proposition \ref{rigidness}) 
and, 
furthermore, overconvergent 
(Theorem \ref{rigidness III}).
We remark that when $r=1$, $\ell^{(p),\star}_{n}(1;z)$
is equal to the $p$-adic polylogarithm
$\ell_n^{(p)}(z)$
in \cite[p.196]{C}.
The following integral expressions 
are generalizations of \cite[Lemma 7.2]{C}:

\begin{theorem}\label{integral theorem}
Let $n_1,\dots,n_r\in \mathbb{N}$, and $\xi_1,\dots,\xi_{r}\in\mathbb{C}_p$ with $|\xi_j|_p\leqslant 1$ ($1\leqslant j\leqslant r$).
Define the 
finite subset $S$ of  $\textbf{P}^{1}(\overline{\mathbb{F}}_p)$ by
\footnote{
Here, we ignore the multiplicity.
}
\begin{equation}\label{S0}
S=\{
\overline{\xi_{r}^{-1}},\overline{(\xi_{r-1}\xi_{r})^{-1}},
\dots,\overline{(\xi_1\cdots\xi_{r-1})^{-1}}\}.
\end{equation}
Then the $p$-adic rigid TMSPL
$\ell^{(p),\star}_{n_1,\dots,n_r}(\xi_1,\dots,\xi_{r};z)$ is 
extended 
to $\textbf{P}^{1}({\mathbb{C}}_p)- ]S[$
as a function in 
$z$ by  the following $p$-adic integral expression:
\begin{align}\label{integral expression}
\ell&^{(p),\star}_{ n_1,\dots,n_r}  (\xi_1,\dots,\xi_{r};z)= \notag \\
&\int_{(\mathbb{Z}_p^r)'_{\{1\}}}\langle x_1\rangle^{-n_1}\langle x_1+x_2\rangle^{-n_2}\cdots\langle x_1+\cdots+x_r\rangle^{-n_r}\cdot 
\omega(x_1)^{-n_1}\omega(x_1+x_2)^{-n_2}\cdots\omega(x_1+\cdots+x_r)^{-n_r} \notag \\
&\qquad\qquad\qquad \times d{\frak m}_{\xi_1\cdots\xi_{r}z}(x_1)\cdots d{\frak m}_{\xi_{r} z}(x_{r}), 
\end{align}
where $(\mathbb{Z}_p^r)'_{\{1\}}=\Bigl\{(x_1,\dots,x_r)\in\mathbb{Z}_p^r\Bigm|p\nmid x_1, p\nmid (x_1+x_2),\dots, p\nmid (x_1+\cdots+x_r)\Bigr\}$
(cf. \eqref{region}).
\end{theorem}

\begin{proof}
Since $\langle x\rangle\cdot \omega (x)=x\neq 0$
for $x\in \mathbb{Z}_p^\times$,
the right-hand side of \eqref{integral expression} is
\begin{align}
\int_{(\mathbb{Z}_p^r)'_{\{1\}}}
& x_1^{-n_1}(x_1+x_2)^{-n_2}\cdots (x_1+\cdots+x_r)^{-n_r}
d{\frak m}_{\xi_1\cdots\xi_{r}z}(x_1)\cdots d{\frak m}_{\xi_{r} z}(x_{r}) 
\notag\\
&=\lim_{M\to\infty}
\underset{(l_1,p)=\cdots=(l_1+\cdots+l_r,p)=1}{\underset{0\leqslant l_1, \dots,l_{r}<p^M}{\sum}}
\frac{\xi_1^{l_1}\xi_2^{l_1+l_2}\cdots\xi_{r}^{l_1+\cdots+l_{r}}}{l_1^{n_1}(l_1+l_2)^{n_2}\cdots (l_1+\cdots+l_r)^{n_r}}z^{l_1+\cdots+l_r} \notag \\
&\qquad \qquad \times \frac{1}{1-(\xi_1\cdots\xi_{r} z)^{p^M}}
\cdots 
\frac{1}{1-(\xi_{r-1}\xi_r z)^{p^M}}\cdot
\frac{1}{1-(\xi_r z)^{p^M}}.
\label{limit expression}\\
&=\lim_{M\to\infty} \qquad g^M_{n_1,\dots,n_r}(\xi_1,\dots,\xi_{r};z)
\qquad \qquad\qquad \text{(say).} \notag
\end{align}
By direct calculation,\ 
it can be shown that it is equal to
the right-hand side of \eqref{series expression for ell} 
when $|z|_p<1$.
\end{proof}

As for $g^M_{n_1,\dots,n_r}(\xi_1,\dots,\xi_{r};z)$
defined in the above proof, we have

\begin{lemma}\label{congruence}
Fix  $n_1,\dots, n_r, M\in \mathbb{N}$, and $\xi_1,\dots,\xi_{r}\in\mathbb{C}_p$ with $|\xi_j|_p=1$ ($1\leqslant j\leqslant r$).
Then for 
$
z_0\in 
{\bf P}^1({\mathbb C}_p) -]\overline{\xi_{r}^{-1}},\overline{(\xi_{r-1}\xi_r)^{-1}},\dots,\overline{(\xi_1\cdots\xi_{r})^{-1}}[ 
$,  we have
$$
g^M_{n_1,\dots,n_r}(\xi_1,\dots,\xi_{r};z_0)\in{\mathcal O}_{{\mathbb C}_p}
$$
and
$$
g^{M+1}_{n_1,\dots,n_r}(\xi_1,\dots,\xi_{r};z_0)\equiv g^M_{n_1,\dots,n_r}(\xi_1,\dots,\xi_{r};z_0) \pmod{p^M}.
$$
\end{lemma}

\begin{proof}
When $z_0\in{\bf P}^1({\mathbb C}_p) 
-]\overline{\xi_{r}^{-1}},\overline{(\xi_{r-1}\xi_r)^{-1}},\dots,\overline{(\xi_1\cdots\xi_{r})^{-1}},
\overline{\infty}[$,
namely when $z_0\in{\mathcal O}_{{\mathbb C}_p}
-]S[
$, it is clear that
$g^M_{n_1,\dots,n_r}(\xi_1,\dots,\xi_{r};z_0)\in{\mathcal O}_{{\mathbb C}_p}$.
In the definition of $g^{M+1}_{n_1,\dots,n_r}(\xi_1,\dots,\xi_{r};z)$, 
$l_j$ ($1\leqslant j\leqslant r$) 
belongs to 
the interval
$[0,p^{M+1})$.
Writing $l_j=l_j^{\prime}+kp^M$ ($0\leqslant l_j^{\prime}<p^M$ and $1\leqslant k\leqslant p-1$),
we have

\begin{align*}
g^{M+1}_{n_1,\dots,n_r}&(\xi_1,\dots,\xi_{r};z_0)
\equiv
\underset{(l_1^{\prime} ,p)=\cdots=(l_1^{\prime} +\cdots+l_r^{\prime} ,p)=1}
{\underset{0\leqslant l_1^{\prime} , \dots,l_{r}^{\prime} <p^M}{\sum}} 
\frac{\xi_1^{l_1^{\prime} }\xi_2^{l_1^{\prime} +l_2^{\prime} }
\cdots\xi_{r}^{l_1^{\prime} +\cdots+l_{r}^{\prime} }z_0^{l_1^{\prime} +\cdots+l_r^{\prime} }}{l_1^{\prime n_1}(l_1^{\prime} +l_2^{\prime} )^{n_2}\cdots (l_1^{\prime} +\cdots+l_r^{\prime} )^{n_r}} \\ 
&\times \{1+(\xi_1\cdots\xi_{r} z_0)^{p^M}+(\xi_1\cdots\xi_{r} z_0)^{2p^M}+\cdots+(\xi_1\cdots\xi_{r} z_0)^{(p-1)p^M}\} \\
&\times \{1+(\xi_2\cdots\xi_{r} z_0)^{p^M}+(\xi_2\cdots\xi_{r} z_0)^{2p^M}+\cdots+(\xi_2\cdots\xi_{r} z_0)^{(p-1)p^M}\} \\
&\qquad\qquad\qquad\qquad \cdots \\
&\times \{1+(\xi_{r} z_0)^{p^M}+(\xi_{r} z_0)^{2p^M}+\cdots+(\xi_{r} z_0)^{(p-1)p^M}\}  \\
&\times \frac{1}{1-(\xi_1\cdots\xi_{r} z)^{p^{M+1}}}
\cdots 
\frac{1}{1-(\xi_{r-1}\xi_r z)^{p^{M+1}}}\cdot
\frac{1}{1-(\xi_r z)^{p^{M+1}}}
\pmod{p^M} \\
&
\qquad
=g^M_{n_1,\dots,n_r}(\xi_1,\dots,\xi_{r};z).
\end{align*}

When $z_0\in]\overline{\infty}[$, 
let $\varepsilon=\frac{1}{z_0}\in]\bar{0}[$.
By direct calculation,
$g^M_{n_1,\dots,n_r}(\xi_1,\dots,\xi_{r};z_0)\in{\mathcal O}_{{\mathbb C}_p}$.
We then 
have
\begin{align*}
g^{M}_{n_1,\dots,n_r}&(\xi_1,\dots,\xi_{r};z_0)
=
\underset{(l_1,p)=\cdots=(l_1+\cdots+l_r,p)=1}{\underset{0\leqslant l_1, \dots,l_{r}<p^M}{\sum}} 
\frac{(\frac{\xi_1\xi_2\cdots\xi_{r}}{\varepsilon})^{l_1}(\frac{\xi_2\cdots\xi_{r}}{\varepsilon})^{l_2}\cdots (\frac{\xi_{r-1}\xi_r}{\varepsilon})^{l_{r-1}}
(\frac{\xi_r}{\varepsilon})^{l_r}}
{l_1^{n_1}(l_1+l_2)^{n_2}\cdots (l_1+\cdots+l_r)^{n_r}} \\
&\qquad\qquad\times\frac{1}{1-(\frac{\xi_1\cdots\xi_{r}}{\varepsilon})^{p^M}}
\cdots 
\frac{1}{1-(\frac{\xi_{r-1}\xi_r}{\varepsilon})^{p^M}}\cdot
\frac{1}{1-(\frac{\xi_r}{\varepsilon})^{p^M}} \\
&=(-1)^r
\underset{(l_1,p)=\cdots=(l_1+\cdots+l_r,p)=1}{\underset{0\leqslant l_1, \dots,l_{r}<p^M}{\sum}} 
\frac{(\frac{\varepsilon}{\xi_1\xi_2\cdots\xi_{r}})^{p^M-l_1}(\frac{\varepsilon}{\xi_2\cdots\xi_{r}})^{p^M-l_2}\cdots (\frac{\varepsilon}{\xi_{r-1}\xi_r})^{p^M-l_{r-1}}
(\frac{\varepsilon}{\xi_r})^{p^M-l_r}}
{l_1^{n_1}(l_1+l_2)^{n_2}\cdots (l_1+\cdots+l_r)^{n_r}} \\
&\qquad\qquad\times\frac{1}{1-(\frac{\varepsilon}{\xi_1\cdots\xi_{r}})^{p^M}}
\cdots 
\frac{1}{1-(\frac{\varepsilon}{\xi_{r-1}\xi_r})^{p^M}}\cdot
\frac{1}{1-(\frac{\varepsilon}{\xi_r})^{p^M}} \\
&=(-1)^r
\underset{(l_1,p)=\cdots=(l_1+\cdots+l_r,p)=1}{\underset{0\leqslant l_1, \dots,l_{r}<p^M}{\sum}} 
\frac{(\frac{\varepsilon}{\xi_1\xi_2\cdots\xi_{r}})^{l_1}(\frac{\varepsilon}{\xi_2\cdots\xi_{r}})^{l_2}\cdots (\frac{\varepsilon}{\xi_{r-1}\xi_r})^{l_{r-1}}
(\frac{\varepsilon}{\xi_r})^{l_r}}
{(p^M-l_1)^{n_1}(2p^M-l_1-l_2)^{n_2}\cdots (rp^M-l_1-\cdots-l_r)^{n_r}} \\
&\qquad\qquad\cdot\frac{1}{1-(\frac{\varepsilon}{\xi_1\cdots\xi_{r}})^{p^M}}
\cdots 
\frac{1}{1-(\frac{\varepsilon}{\xi_{r-1}\xi_r})^{p^M}}\cdot
\frac{1}{1-(\frac{\varepsilon}{\xi_r})^{p^M}} \\
&\equiv (-1)^{r+n_1+\cdots+n_r}
\underset{(l_1,p)=\cdots=(l_1+\cdots+l_r,p)=1}{\underset{0\leqslant l_1, \dots,l_{r}<p^M}{\sum}} 
\frac{(\frac{\varepsilon}{\xi_1\xi_2\cdots\xi_{r}})^{l_1}(\frac{\varepsilon}{\xi_2\cdots\xi_{r}})^{l_2}\cdots (\frac{\varepsilon}{\xi_{r-1}\xi_r})^{l_{r-1}}
(\frac{\varepsilon}{\xi_r})^{l_r}}
{{l_1}^{n_1}(l_1+l_2)^{n_2}\cdots (l_1+\cdots+l_r)^{n_r}} \\
&\qquad\qquad\cdot\frac{1}{1-(\frac{\varepsilon}{\xi_1\cdots\xi_{r}})^{p^M}}
\cdots 
\frac{1}{1-(\frac{\varepsilon}{\xi_{r-1}\xi_r})^{p^M}}\cdot
\frac{1}{1-(\frac{\varepsilon}{\xi_r})^{p^M}} \pmod{p^M} \\
&= (-1)^{r+n_1+\cdots+n_r}
g^{M}_{n_1,\dots,n_r}(\xi_1^{-1},\dots,\xi_{r}^{-1};\varepsilon).
\end{align*}
Therefore, by our previous argument, and by $|\xi_j|_p=|\xi_j^{-1}|_p=1$,
it follows that 
\begin{align*}
g^{M+1}_{n_1,\dots,n_r}&(\xi_1,\dots,\xi_{r};z_0)
\equiv  (-1)^{r+n_1+\cdots+n_r}
g^{M+1}_{n_1,\dots,n_r}(\xi_1^{-1},\dots,\xi_{r}^{-1};\varepsilon) \\
&\equiv (-1)^{r+n_1+\cdots+n_r}
g^{M}_{n_1,\dots,n_r}(\xi_1^{-1},\dots,\xi_{r}^{-1};\varepsilon)
\equiv
g^{M}_{n_1,\dots,n_r}(\xi_1,\dots,\xi_{r};z_0)
\pmod{p^M}.
\end{align*}
\end{proof}


Theorem \ref{integral theorem} and Lemma \ref{congruence} imply the following:

\begin{proposition}\label{rigidness}
Fix  $n_1,\dots,n_r\in \mathbb{N}$ and $\xi_1,\dots,\xi_{r}\in\mathbb{C}_p$ with $|\xi_j|_p=1$ ($1\leqslant j\leqslant r$).
By the 
integral formula \eqref{integral expression},
the  $p$-adic rigid TMSPL 
$\ell^{(p),\star}_{n_1,\dots,n_r}(\xi_1,\dots,\xi_{r};z)$ is a rigid analytic function on
$
{\bf P}^1({\mathbb C}_p) - ]S[
$.
Namely, 
$$
\ell^{(p),\star}_{n_1,\dots,n_r}(\xi_1,\dots,\xi_{r};z)
\in A^{\rm{rig}}( {\bf P}^1\setminus S
).
$$
\end{proposition}

\begin{proof}
Since the space ${\bf P}^1({\mathbb C}_p) - ]S[$
is an affinoid and
the algebra 
$A^{\mathrm{rig}}( {\bf P}^1\setminus S)$
forms a Banach algebra 
with the supremum norm (cf. Notation \ref{basics on rigid}),
by Lemma \ref{congruence}, the rational functions
$$
g^{M}_{n_1,\dots,n_r}(\xi_1,\dots,\xi_{r};z)\in
A^{\mathrm{rig}}( {\bf P}^1\setminus S)
$$
uniformly converge to a rigid analytic function 
$\ell(z)\in A^{\mathrm{rig}}( {\bf P}^1\setminus S)$
when $M$ tends to $\infty$, thanks to Proposition \ref{two fundamental properties of rigid analytic functions}.
It is easy to see that the restriction of $\ell(z)$ 
to $\textbf{P}^{1}({\mathbb{C}}_p)- ]S[$ coincides with \eqref{limit expression}, 
and hence, with $\ell^{(p),\star}_{n_1,\dots,n_r}(\xi_1,\dots,\xi_{r};z)$.
Therefore, the analytic continuation of  $\ell^{(p),\star}_{n_1,\dots,n_r}(\xi_1,\dots,\xi_{r};z)$
is given by $\ell(z)$.
\end{proof}

From now on, we will employ the same symbol 
$\ell^{(p),\star}_{n_1,\dots,n_r}(\xi_1,\dots,\xi_{r};z)$
to denote its analytic continuation.

We note that, by \eqref{limit expression},

\begin{lemma}\label{ell at infinity}
For $n_1,\dots,n_r\in \mathbb{N}$, and $\xi_1,\dots,\xi_{r}\in\mathbb{C}_p$ with $|\xi_j|_p=1$ ($1\leqslant j\leqslant r$),
we have
$$
\ell^{(p),\star}_{n_1,\dots,n_r}(\xi_1,\dots,\xi_{r};\infty)=0.$$
\end{lemma}

\begin{proof}
By direct 
calculation, $g^{M}_{n_1,\dots,n_r}(\xi_1,\dots,\xi_{r};\infty)=0$.
The claim 
then follows 
because $\ell^{(p),\star}_{n_1,\dots,n_r}(\xi_1,\dots,\xi_{r};z)$ is defined to be the limit of
$g^{M}_{n_1,\dots,n_r}(\xi_1,\dots,\xi_{r};z)$.
\end{proof}

The special values of $p$-adic multiple $L$-functions at positive integer points
are described in terms of the special values of
$p$-adic rigid TMSPLs
at roots of unity, as follows:

\begin{theorem}\label{L-ell theorem}
For  $n_1,\dots,n_r\in \mathbb{N}$
and $c\in \mathbb{N}_{>1}$ with $(c,p)=1$,
\begin{equation}\label{L-ell-formula}
L_{p,r}(n_1,\dots,n_r;\omega^{-n_1},\dots,\omega^{-n_r};1,\dots,1;c)= \\
\underset{\xi_1\cdots\xi_r\neq 1, \ \dots, \ \xi_{r-1}\xi_r\neq 1, \ \xi_r\neq 1}{\sum_{\xi_1^c=\cdots=\xi_r^c=1}}
\ell^{(p),\star}_{n_1,\dots,n_r}(\xi_1,\dots,\xi_r;1).
\end{equation}
\end{theorem}

\begin{proof}
%
By the definition,
\begin{align*}
L_{p,r}(n_1,\dots,&n_r;\omega^{-n_1},\dots,\omega^{-n_r};1,\dots,1;c)= \\
&\int_{(\mathbb{Z}_p^r)'_{\{1\}}}\langle x_1\rangle^{-n_1}\langle x_1+x_2\rangle^{-n_2}\cdots\langle x_1+\cdots+x_r\rangle^{-n_r}  \\
&\qquad\qquad
\cdot\omega(x_1)^{-n_1}\omega(x_1+x_2)^{-n_2}\cdots\omega(x_1+\cdots+x_r)^{-n_r} 
d{\widetilde{\frak m}}_c(x_1)\cdots d{\widetilde{\frak m}}_c(x_r),
\end{align*}
where 
${\widetilde{\frak m}}_c=\underset{\xi\neq 1}{\underset{{\xi^c=1}}{\sum}}{\frak m}_\xi$.
By \eqref{integral expression} and
\begin{align*}
({\widetilde{\frak m}}_c)^r
&=\Bigl\{\underset{\xi_1^{\prime}\neq 1}{\underset{{\xi_1^{\prime c}=1}}{\sum}}{\frak m}_{\xi_1^{\prime}} \Bigr\}\times
\Bigl\{\underset{\xi_2^{\prime}\neq 1}{\underset{{\xi_2^{\prime c}=1}}{\sum}}{\frak m}_{\xi_2^{\prime}} \Bigr\}\times\cdots\times
\Bigl\{\underset{\xi_r^{\prime}\neq 1}{\underset{{\xi_r^{\prime c}=1}}{\sum}}{\frak m}_{\xi_r^{\prime}} \Bigr\}
\\
&=\underset{\xi_1\cdots\xi_r\neq 1,\dots,\xi_{r-1}\xi_r\neq 1,\xi_r\neq 1}{\sum_{\xi_1^c=\cdots=\xi_r^c=1}}
{\frak m}_{\xi_1\cdots\xi_r}\times{\frak m}_{\xi_2\cdots\xi_r}\times
\cdots\times
{\frak m}_{\xi_{r-1}\xi_r}\times{\frak m}_{\xi_r},
\end{align*}
the formula follows.
\end{proof}

\begin{remark}
It is 
worth noting 
that  the right-hand side of \eqref{L-ell-formula} is $p$-adically
continuous, not only with respect to $n_1,\ldots,n_r$, but also with respect to $c$, 
by Theorem \ref{continuity theorem}.
\end{remark}

As a special case when $r=1$ of Theorem \ref{L-ell theorem},
we recover Coleman's formula in \cite[p.203]{C}, below, 
by Example \ref{example for r=1}.

\begin{example}
For $n\in \mathbb{N}_{>1}$
and $c\in \mathbb{N}_{>1}$ with $(c,p)=1$,
\begin{equation}\label{Coleman equality}
(c^{1-n}-1)\cdot L_{p}(n;\omega^{1-n})
=\underset{\xi\neq 1}{\sum_{\xi^c=1}}\ell^{(p),\star}_{n}(\xi;1).
\end{equation}
\end{example}

When $r=2$, we have: 

\begin{example}
For  $n_1,n_2\in \mathbb{N}$
and $c\in \mathbb{N}_{>1}$ with $(c,p)=1$,
\begin{equation*}
L_{p,2}(n_1,n_2;\omega^{-n_1},\omega^{-n_2};1,1;c)= \\
\underset{\xi_1\xi_2\neq 1, \ \xi_2\neq 1}{\sum_{\xi_1^c=\xi_2^c=1}}
\ell^{(p),\star}_{n_1,n_2}(\xi_1,\xi_2;1).
\end{equation*}
\end{example}

\subsection{$p$-adic partial twisted multiple star polylogarithms}\label{sec-5-4}
We will prove that
$p$-adic rigid TMSPLs  (Definition \ref{def of pMMPL})
are overconvergent in Theorem \ref{rigidness II}.
In order to do so, 
$p$-adic partial TMSPLs will be introduced in Definition \ref{def of pPMPL}, 
and their properties will be investigated.

First, we recall the notion of overconvergent functions and rigid cohomologies
in our particular  case
(consult \cite{Ber} for a general theory)
%
%
%
%

\begin{notation}\label{overconvergent functions and associated cohomologies}
Let
$S=\{s_0,\dots,s_d\}$ (all $s_i$'s are distinct)
be a finite subset of $\textbf{P}^{1}(\overline{\mathbb{F}}_p)$.
Let $\widehat s_i$ be a lift of $s_i$. 
An {\it overconvergent function} on  $\textbf{P}^{1}\setminus S$
is a function belonging to the ${\mathbb C}_p$-algebra
$$
A^\dag(  {\bf P}^1\setminus S)
:=
\underset{\lambda\to 1^{-}}{\rm ind\text{-}lim} \ A^\mathrm{rig}(U_\lambda),
$$
where
$U_\lambda$ is the affinoid  
obtained by removing all closed discs
of radius $\lambda$ around $\widehat s_i$ 
from
${\bf P}^1({\mathbb C}_p)$, i.e.
\begin{equation}\label{removing all closed discs}
U_\lambda:={\bf P}^1({\mathbb C}_p)\setminus \bigcup_{0\leqslant i \leqslant d} 
z_i^{-1}\left(\{\alpha\in{\mathbb C}_p\bigm||\alpha|_p\leqslant\lambda\}\right)
\end{equation}
and $z_i$ is a local parameter
\begin{equation}\label{local parameter}
z_i:]s_i[\overset{\sim}{\to} ]\bar{0}[.
\end{equation}
(Note that the above $\widehat s_i$ is equal to $z_i^{-1}(0)$.)
An overconvergent function on  $ {\bf P}^1\setminus S$ is, in short, 
a function 
which can be analytically extended 
to an affinoid which is bigger than
${\bf P}^1({\mathbb C}_p)- ]s_0,s_1,\dots,s_d[$.
We note that the definition of the space of
overconvergent functions 
does not depend on the
choice of  local parameter $z_i$.

\end{notation}

Denote by $\mathbb{C}_p[[x,y]]$ the ring of formal power series in $x,y$ with $\mathbb{C}_p$-coefficients. The following lemmas are quite useful:

\begin{lemma}\label{Lemma-3-14}
Assume $s_0=\overline{\infty}$ and take $\widehat{s_0}=\infty$.
Then 
we have a description:
\begin{equation}\label{description of A1}
\begin{split}
A^\dag(  {\bf P}^1\setminus S) \simeq
\Biggl\{f(z)=\sum_{r\geqslant 0}a_{r}(\widehat s_0;f)z^{r}
+\sum_{l=1}^d\sum_{m>0}\frac{a_{m}(\widehat s_l;f)}{(z-\widehat{s_l})^{m}}
\in{\mathbb C}_p[[z,\frac{1}{z-\widehat{s_l}}]] \Biggm| 
&\\ 
\frac{|a_{m}(\widehat s_l;f)|_p}{\lambda^{m}}\to 0 \  (m\to \infty) 
 \ \text{  for some  } 0<\lambda <1 
\quad (0\leqslant l \leqslant & d) 
\Biggr\}, \\
\end{split}
\end{equation}
and 
\begin{equation}\label{description of A2}
\begin{split}
A^{\mathrm{rig}}({\bf P}^1\setminus S) \simeq
\Biggl\{f(z)=&\sum_{m\geqslant 0}a_{m}(\widehat s_0;f)z^{m}
+\sum_{l=1}^d\sum_{m>0}\frac{a_{m}(\widehat s_l;f)}{(z-\widehat{s_l})^{m}}
\in{\mathbb C}_p[[z,\frac{1}{z-\widehat{s_l}}]] \Biggm| \\
&
|a_{m}(\widehat s_l;f)|_p\to 0 \  (m\to \infty) 
\text{ for }  0\leqslant l \leqslant d
\Biggr\}. 
\end{split}
\end{equation}
\end{lemma}

\begin{proof}
They follow from the definitions (cf. Notation \ref{basics on rigid}).
\end{proof}

We note that 
$$
A^\dag(  {\bf P}^1\setminus S)\subset
A^{\mathrm{rig}}({\bf P}^1\setminus S) .
$$
The following is one of the most important properties of overconvergent functions.
\begin{lemma}\label{useful example}
Let $f(z)\in A^\dag(  {\bf P}^1\setminus S)$.
Under the 
assumption in Lemma \ref{Lemma-3-14},
there exists a unique (up to 
additive constant)
solution $F(z)\in A^\dag(  {\bf P}^1\setminus S)$ of the differential equation
$$
dF(z)=f(z)dz
$$
if and only if the {\it residues} of the differential $1$-form $f(z)dz$,
i.e. $a_1(\widehat s_l;f)$ ($1\leqslant l\leqslant d$) are all $0$.
\end{lemma}

\begin{proof}
When $a_1(\widehat s_l;f)$ ($1\leqslant l\leqslant d$) are all $0$,
integrations of $f(z)$ in \eqref{description of A1} 
are formally given by the following power series
$$
\sum_{r\geqslant 1}\frac{a_{r-1}(\widehat s_0;f)}{r}\cdot z^{r}
+\sum_{l=1}^d\sum_{m>0}\frac{a_{m+1}(\widehat s_l;f)}{-m}\cdot\frac{1}{(z-\widehat{s_l})^{m}}+
\text{constant}.
$$
By replacing the $\lambda$  by $\lambda'$ such that $\lambda<\lambda'<1$, 
and using $\text{ord}_p(n)=O(\log n/\log p )$,
we then get
$$
\frac{|a_{m+1}(\widehat s_l;f)|_p}{|m|_p}\cdot \frac{1}{\lambda^{\prime m}}\to 0 \quad  (m\to \infty),
$$
whence $F(z)$ belongs to  $A^\dag(  {\bf P}^1\setminus S)$.
The \lq if'-part is obtained.
The \lq only if'-part is easy.
\end{proof}

The lemma actually is a consequence of the fact that
$\dim H^1_\dag( {\bf P}^1\setminus S)=d$.

The following function is the main object in this subsection:

\begin{definition}\label{def of pPMPL}
Let  $n_1,\dots,n_r\in \mathbb{N}$ and $\xi_1,\dots,\xi_{r}\in\mathbb{C}_p$ 
such that $|\xi_j|_p\leqslant 1$ ($1\leqslant j\leqslant r$).
Let  $\alpha_1,\dots,\alpha_r\in \mathbb{N}$ with $0<\alpha_j<p$
 ($1\leqslant j\leqslant r$).
The {\bf $p$-adic partial TMSPL} 
$\ell^{\equiv (\alpha_1,\dots,\alpha_r),(p),\star}_{n_1,\dots,n_r}(\xi_1,\dots\xi_{r};z)$ is defined by the following 
$p$-adic power series:
\begin{equation}\label{series expression for partial double}
\ell^{\equiv (\alpha_1,\dots,\alpha_r),(p),\star}_{n_1,\dots,n_r}(\xi_1,\dots,\xi_{r};z):=
\underset{k_1\equiv \alpha_1,\dots ,k_r\equiv\alpha_r \bmod p
}{\underset{0<k_1\leqslant\cdots \leqslant k_r}{\sum}}
\frac{\xi_1^{k_1}\cdots\xi_{r}^{k_{r}}}{k_1^{n_1}\cdots k_r^{n_r}}z^{k_r},
\end{equation}
which converges for $z\in ]\bar{0}[$.
\end{definition}

Similar 
to \eqref{limit expression}, we have the following 
limit expression of them.

\begin{proposition}
Let  $n_1,\dots,n_r\in \mathbb{N}$,
$\xi_1,\dots,\xi_{r}\in\mathbb{C}_p$ such that 
$|\xi_j|_p\leqslant 1$ ($1\leqslant j\leqslant r$), 
and $\alpha_1,\dots,\alpha_r\in \mathbb{N}$ such that $0<\alpha_j<p$
($1\leqslant j\leqslant r$).
When $z\in]\overline{0}[$, the function $\ell^{\equiv (\alpha_1,\dots,\alpha_r),(p),\star}_{n_1,\dots,n_r}(\xi_1,\dots,\xi_{r};z)$ can be 
expressed as
\begin{align}\label{limit expression for partial double}
\ell^{\equiv (\alpha_1,\dots,\alpha_r),(p),\star}_{n_1,\dots,n_r}&
(\xi_1,\dots,\xi_{r};z) 
=\lim_{M\to\infty}
\underset{l_1\equiv \alpha_1,l_1+l_2\equiv\alpha_2,\dots ,l_1+\cdots+l_r\equiv\alpha_r \bmod p}{\underset{0\leqslant l_1, \dots,l_{r}<p^M}{\sum}}
\frac{\xi_1^{l_1}\xi_2^{l_1+l_2}\cdots\xi_{r}^{l_1+\cdots+l_{r}}z^{l_1+\cdots+l_r}}
{l_1^{n_1}(l_1+l_2)^{n_2}\cdots (l_1+\cdots+l_r)^{n_r}} \notag \\
&\qquad \qquad \cdot\frac{1}{1-(\xi_1\cdots\xi_{r} z)^{p^M}}
\cdots 
\frac{1}{1-(\xi_{r-1}\xi_r z)^{p^M}}\cdot
\frac{1}{1-(\xi_r z)^{p^M}}.
\end{align}
\end{proposition}

\begin{proof}
It can be proved by direct calculation.
\end{proof}

It is 
worth noting that here the condition 
$\alpha_j\neq 0$ ($1\leqslant j\leqslant r$)
is necessary to make the limit convergent.

\begin{proposition}\label{rigid extension}
Let  $n_1,\dots,n_r\in \mathbb{N}$, 
$\xi_1,\dots,\xi_{r}\in\mathbb{C}_p$ such that $|\xi_j|_p= 1$ ($1\leqslant j\leqslant r$), 
and $\alpha_1,\dots,\alpha_r\in \mathbb{N}$ such that $0<\alpha_j<p$
($1\leqslant j\leqslant r$).
Take $S$
as in \eqref{S0}.
Then the function $\ell^{\equiv (\alpha_1,\dots,\alpha_r),(p),\star}_{n_1,\dots,n_r}(\xi_1,\dots,\xi_{r};z)$
can be 
analytically extended 
to 
${\bf P}^1({\mathbb C}_p) - ]S[
$
as a rigid analytic function.
Namely, 
$$\ell^{\equiv (\alpha_1,\dots,\alpha_r),(p),\star}_{n_1,\dots,n_r}(\xi_1,\dots,\xi_{r};z)\in 
A^{\mathrm{rig}}( {\bf P}^1\setminus S).
$$
\end{proposition}

\begin{proof}
%
The relation
\begin{equation}\label{partial ell-ell}
\ell^{\equiv (\alpha_1,\dots,\alpha_r),(p),\star}_{n_1,\dots,n_r}(\xi_1,\dots,\xi_{r};z)
=\frac{1}{p^r}\sum_{\rho_1^p=\cdots=\rho_r^p=1} \rho_1^{-\alpha_1}\cdots\rho_r^{-\alpha_r}
\ell^{(p),\star}_{n_1,\dots,n_r}(\rho_1\xi_1,\dots,\rho_{r}\xi_{r}; z)
\end{equation}
holds on $]\overline{0}[$.
Since $\ell^{(p),\star}_{n_1,\dots,n_r}(\rho_1\xi_1,\dots,\rho_{r}\xi_{r}; z)$ 
is rigid analytic on the above space
by Proposition \ref{rigidness},
the function $\ell^{\equiv (\alpha_1,\dots,\alpha_r),(p),\star}_{n_1,\dots,n_r}(\xi_1,\dots,\xi_{r};z)$
can be extended there as a rigid analytic function.
\end{proof}

From now on, we will employ the same symbol
$\ell^{\equiv (\alpha_1,\dots,\alpha_r),(p),\star}_{n_1,\dots,n_r}(\xi_1,\dots,\xi_{r};z)$
to denote its analytic continuation.

The following formulas are necessary to prove our Theorem \ref{rigidness II}.

\begin{lemma}\label{differential equations}
Let  $n_1,\dots,n_r\in \mathbb{N}$, 
$\xi_1,\dots,\xi_{r}\in\mathbb{C}_p$ such that 
$|\xi_j|_p= 1$ ($1\leqslant j\leqslant r$), 
and $\alpha_1,\dots,\alpha_r\in \mathbb{N}$ such that 
$0<\alpha_j<p$
($1\leqslant j\leqslant r$).

{\rm (i)}\ For $n_r\neq 1$, 
$$\frac{d}{dz}\ell^{\equiv (\alpha_1,\dots,\alpha_r),(p),\star}_{n_1,\dots,n_r}(\xi_1,\dots,\xi_{r};z)
=\frac{1}{z}\ell^{\equiv (\alpha_1,\dots,\alpha_r),(p),\star}_{n_1,\dots,n_{r-1},n_r-1}(\xi_1,\dots,\xi_{r};z).
$$

{\rm (ii)}\ For $n_r=1$ and $r\neq 1$,
$$\frac{d}{dz}\ell^{\equiv (\alpha_1,\dots,\alpha_r),(p),\star}_{n_1,\dots,n_r}(\xi_1,\dots,\xi_{r};z)=
\begin{cases}
&\frac{\xi_r(\xi_r z)^{\alpha_r-\alpha_{r-1}-1}}{1-(\xi_r z)^p}
\ell^{\equiv (\alpha_1,\dots,\alpha_{r-1}),(p),\star}_{n_1,\dots,n_{r-1}}(\xi_1,\dots,\xi_{r-2},\xi_{r-1};\xi_r z)  \\
&\qquad\qquad\qquad\qquad\qquad\qquad \text{if}\quad \alpha_r\geqslant\alpha_{r-1} , \\
&\frac{\xi_r(\xi_r z)^{\alpha_r-\alpha_{r-1}+p-1}}{1-(\xi_r z)^p}\ell^{\equiv (\alpha_1,\dots,\alpha_{r-1}),(p),\star}_{n_1,\dots,n_{r-1}}(\xi_1,\dots,\xi_{r-2},\xi_{r-1};\xi_r z)  \\
&\qquad\qquad\qquad\qquad\qquad\qquad \text{if}\quad \alpha_r<\alpha_{r-1}. \\
\end{cases}
$$

{\rm (iii)}\ For $n_r=1$ and $r=1$ with $\xi_1=\xi$ and  $\alpha_1=\alpha$,
 $$\frac{d}{dz}\ell^{\equiv \alpha, (p),\star}_{1}(\xi;z)=
\frac{\xi (\xi z)^{\alpha-1}}{1-(\xi z)^p}.
$$

\end{lemma}

\begin{proof}
They can be proved by direct calculation.
\end{proof}

By Lemma \ref{ell at infinity} and \eqref{partial ell-ell}, 
we remark that:

\begin{remark}\label{easy remark}
For $n_1,\dots,n_r\in \mathbb{N}$, 
$\xi_1,\dots,\xi_{r}\in\mathbb{C}_p$ such that 
$|\xi_j|_p= 1$ ($1\leqslant j\leqslant r$)
and $\alpha_1,\dots,\alpha_r\in \mathbb{N}$ such that 
$0<\alpha_j<p$
($1\leqslant j\leqslant r$),
the following hold:

%
(i)\  $\ell^{\equiv (\alpha_1,\dots,\alpha_r),(p),\star}_{n_1,\dots,n_r}(\xi_1,\dots,\xi_{r};0)
=\ell^{\equiv (\alpha_1,\dots,\alpha_r),(p),\star}_{n_1,\dots,n_r}(\xi_1,\dots,\xi_{r};\infty)=0$.

(ii)\  $\ell^{(p),\star}_{n_1,\dots,n_r}(\xi_1,\dots,\xi_{r};z)=
\sum_{0<\alpha_1,\dots, \alpha_r<p}\ell^{\equiv (\alpha_1,\dots,\alpha_r),(p),\star}_{n_1,\dots,n_r}(\xi_1,\dots,\xi_{r};z).$
\end{remark}

Next, we discuss a new property of 
$\ell^{\equiv (\alpha_1,\dots,\alpha_r),(p),\star}_{n_1,\dots,n_r}(\xi_1,\dots,\xi_{r}; z)$. 

\begin{theorem}\label{rigidness II}
Let  $n_1,\dots,n_r\in \mathbb{N}$, 
$\xi_1,\dots,\xi_{r}\in\mathbb{C}_p$ such that 
$|\xi_j|_p= 1$ ($1\leqslant j\leqslant r$),
and $\alpha_1,\dots,\alpha_r\in \mathbb{N}$ such that 
$0<\alpha_j<p$
($1\leqslant j\leqslant r$).
Take 
$S$ as in \eqref{S0}.
The function 
$\ell^{\equiv (\alpha_1,\dots,\alpha_r),(p),\star}_{n_1,\dots,n_r}(\xi_1,\dots,\xi_{r}; z)$
is an overconvergent function on 
${\bf P}^1\setminus S$.
Namely, 
$$\ell^{\equiv (\alpha_1,\dots,\alpha_r),(p),\star}_{n_1,\dots,n_r}(\xi_1,\dots,\xi_{r};z)\in
 A^\dag( {\bf P}^1\setminus S
).$$
\end{theorem}

\begin{proof}
It is achieved by induction on the 
weight $n_1+\cdots+n_r$.

(i) Assume that the weight is equal to $1$, i.e. $r=1$ and $n_1=1$.
By 
 changing variable to $w(z)=\frac{1}{\xi z-1}$, we see that 
${\bf P}^1 \setminus\{{\overline{\xi^{-1}}}\}$
is identified with ${\bf P}^1 \setminus\{{\overline{\infty}}\}$.
By direct calculation,
it can be checked that  
$\frac{(\xi z)^{\alpha-1}}{1-(\xi z)^p}\cdot\frac{dz}{dw}$, as a function in $w$, 
belongs to $A^\dag({\bf P}^1\setminus\{{\overline{\infty}}\})$.
\footnote{We note that it  also follows from the fact that it is a rational function on $w$
whose poles are all of the form $w=\frac{1}{\zeta_p-1}$ with $\zeta_p\in\mu_p$.}
Hence, by Lemma \ref{useful example}, 
there exists a unique (modulo constant) function
$F(w)$ in $A^\dag({\bf P}^1
\setminus\{{\overline{\infty}}\})$
such that 
$$
\frac{dF}{dw}=\frac{\xi (\xi z)^{\alpha-1}}{1-(\xi z)^p}\cdot\frac{dz}{dw}.
$$
Therefore, there exists a unique function $F(z)$ in $A^\dag({\bf P}^1
\setminus\{{\overline{\xi^{-1}}}\})$
such that 
$$
F(0)=0 \quad \text{and} \quad
\frac{dF(z)}{dz}=\frac{\xi (\xi z)^{\alpha-1}}{1-(\xi z)^p}.
$$

By Proposition \ref{rigid extension} and Lemma \ref{differential equations} (iii),
 $\ell^{\equiv \alpha,(p),\star}_{1}(\xi;z)$ is 
the unique function in
$A^{\mathrm{rig}}( {\bf P}^1
\setminus \{\overline{\xi^{-1}}\} )$
satisfying the above properties, and
$F(z)|_{{\bf P}^1({\mathbb C}_p) -]{{\overline{\xi^{-1}}}}[}$ belongs to
$A^{\mathrm{rig}}( {\bf P}^1 
 \setminus \{ {\overline{\xi^{-1}}}\} )$.
Thus, we have
$$
F(z)|_{{\bf P}^1({\mathbb C}_p) -]{{\overline{\xi^{-1}}}}[}\equiv\ell^{\equiv \alpha,(p),\star}_{1}(\xi;z).
$$
By the coincidence principle of rigid analytic functions, 
we deduce 
that  $\ell^{\equiv \alpha,(p),\star}_{1}(\xi;z)$ can be uniquely
extended 
to a rigid analytic space bigger than
${\bf P}^1({\mathbb C}_p) -]{{\overline{\xi^{-1}}}}[ $
by $F(z)\in A^\dag({\bf P}^1
\setminus\{{\overline{\xi^{-1}}}\} )$.

(ii) Assume that  $n_r \neq 1$.
We let 
\begin{equation*}
S_\infty=S\cup \{\overline{\infty}\}
\qquad \text{ and } \qquad
S_{\infty,0}=S\cup \{\overline{\infty}\}\cup\{\bar 0\},
\end{equation*}
and take a lift $\{\widehat{s_0}, \widehat{s_1}, \dots, \widehat{s_d}\}$  
of $S_{\infty,0}$,  with
\begin{equation*}
\widehat{s_0}=\infty \qquad \text{ and } \qquad
\widehat{s_1}=0.
\end{equation*}

For a finite subset $\widetilde{S}$ of $\textbf{P}^{1}(\overline{\mathbb{F}}_p)$, let 
$\Omega^{\dag,1}({\bf P}^1\setminus \widetilde{S})$ 
be the space of 
overconvergent differential $1$-forms there, i.e.
$\Omega^{\dag,1}( {\bf P}^1\setminus \widetilde{S})=
A^{\dag,1}( {\bf P}^1\setminus \widetilde{S})dz$.

By the 
assumption
$$
\ell^{\equiv (\alpha_1,\dots,\alpha_r),(p),\star}_{n_1,\dots,n_{r-1},n_r-1}(\xi_1,\dots,\xi_{r-1},\xi_r;z)\in A^\dag( {\bf P}^1\setminus S
),
$$
and by $\frac{dz}{z}\in 
\Omega^{\dag,1}({\bf P}^1
\setminus\{\overline{\infty},\overline{0}\})$,
we have 
\begin{equation*}
\ell^{\equiv (\alpha_1,\dots,\alpha_r),(p),\star}_{n_1,\dots,n_{r-1},n_r-1}(\xi_1,\dots,\xi_{r-1},\xi_r;z)\frac{dz}{z}\in
\Omega^{\dag,1}( {\bf P}^1\setminus S_{\infty,0}
 ).
\end{equation*}
Let 
\begin{equation}\label{put}
f(z):=\frac{1}{z}\ell^{\equiv (\alpha_1,\dots,\alpha_r),(p),\star}_{n_1,\dots,n_{r-1},n_r-1}(\xi_1,\dots,\xi_{r-1},\xi_r; z)\in
A^{\dag}( {\bf P}^1\setminus S_{\infty,0}
).
\end{equation}

Since $\ell^{\equiv (\alpha_1,\dots,\alpha_r),(p),\star}_{n_1,\dots,n_{r-1},n_r}(\xi_1,\dots,\xi_{r};z)$
belongs to
$
A^{\mathrm{rig}}( {\bf P}^1\setminus S
) 
\Bigl(\subset
A^{\mathrm{rig}}( {\bf P}^1\setminus S_{\infty,0}
 ) 
\Bigr) 
$
by Proposition \ref{rigid extension}, 
and it satisfies the differential equation in Lemma \ref{differential equations} (i),
i.e. its differential  is equal to $f(z)$,
we have, in particular, that
\begin{equation}\label{expansion at 0}
a_m(\widehat s_1;f)=0 \qquad (m>0)
\end{equation}
(recall that $\widehat{s_1}=0$) and
\begin{equation}\label{other residues}
a_1(\widehat s_l;f)=0 \qquad (2\leqslant l\leqslant d),
\end{equation}
by \eqref{description of A1} and \eqref{description of A2}.

By \eqref{put} and \eqref{expansion at 0},
$$
f(z)\in A^{\dag}( {\bf P}^1\setminus S_\infty
).
$$

By \eqref{other residues} and Lemma \ref{useful example},
there exists a unique function
$F(z)$ in 
$
A^{\dag}( {\bf P}^1\setminus S_\infty
)
$,
i.e. a function $F(z)$ which is rigid analytic on an affinoid  $V$ of
$$
{\bf P}^1({\mathbb C}_p) - ]S_\infty[  \quad = \quad
{\bf P}^1({\mathbb C}_p) - ]\overline{\infty},S[
$$
such that 
\begin{equation}\label{differential property}
F(0)=0 \quad \text{and} \quad
dF(z)=f(z)dz.
\end{equation}

Since $\ell^{\equiv (\alpha_1,\dots,\alpha_r),(p),\star}_{n_1,\dots,n_{r-1},n_r}(\xi_1,\dots,\xi_{r}; z)$
is the 
unique function in
$
A^{\mathrm{rig}}( {\bf P}^1\setminus S
) 
$
satisfying \eqref{differential property},
the restrictions of both $F(z)$ and
$\ell^{\equiv (\alpha_1,\dots,\alpha_r),(p),\star}_{n_1,\dots,n_{r-1},n_r}(\xi_1,\dots,\xi_{r}; z)$
to the subspace
${\bf P}^1({\mathbb C}_p) -]S_\infty[
$ must coincide, i.e.
$$
F(z)\Bigm|_{{\bf P}^1({\mathbb C}_p) - ]S_\infty[
}
\equiv
\ell^{\equiv (\alpha_1,\dots,\alpha_r),(p),\star}_{n_1,\dots,n_{r-1},n_r}(\xi_1,\dots,\xi_{r};z)\Bigm|_{{\bf P}^1({\mathbb C}_p) - ]S_\infty[
}.
$$

Hence by the coincidence principle of rigid analytic functions,
there is a rigid analytic function $G(z)$ on the union of $V$ and
${\bf P}^1({\mathbb C}_p) - ]S[$
whose restriction to $V$ is equal to $F(z)$
and whose restriction to 
${\bf P}^1({\mathbb C}_p) -]S[
$
is equal to
$\ell^{\equiv (\alpha_1,\dots,\alpha_r),(p),\star}_{n_1,\dots,n_{r-1},n_r}(\xi_1,\dots,\xi_{r}; z)$.
Therefore, 
$$\ell^{\equiv (\alpha_1,\dots,\alpha_r),(p),\star}_{n_1,\dots,n_{r-1},n_r}(\xi_1,\dots,\xi_{r};z)\in
A^{\mathrm{rig}}( {\bf P}^1\setminus S
) 
$$
can be rigid analytically
extended 
to a bigger rigid analytic space
by $G(z)$,
namely, 
$$\ell^{\equiv (\alpha_1,\dots,\alpha_r),(p),\star}_{n_1,\dots,n_r}(\xi_1,\dots,\xi_{r};z)\in
 A^\dag( {\bf P}^1\setminus S
).$$


(iii) Assume that $n_r=1$ ($r\geqslant 2$), and let 
$$
\beta(z):=
\begin{cases}
\frac{\xi_r (\xi_r z)^{\alpha_r-\alpha_{r-1}-1}}{1-(\xi_r z)^p} &\text{if}\quad \alpha_r\geqslant\alpha_{r-1} ,\\
\frac{\xi_r (\xi_r z)^{\alpha_r-\alpha_{r-1}+p-1}}{1-(\xi_r z)^p}
&\text{if}\quad \alpha_r<\alpha_{r-1} .\\
\end{cases}
$$
By the 
assumption
$$
\ell^{\equiv (\alpha_1,\dots,\alpha_{r-1}),(p),\star}_{n_1,\dots,n_{r-1}}(\xi_1,\dots,\xi_{r-2},\xi_{r-1};\xi_r z) 
\in A^\dag( {\bf P}^1 
\setminus\{
\overline{\xi_{r}^{-1}},\dots,\overline{(\xi_1\cdots\xi_{r})^{-1}}\} ),
$$
and by
$\beta(z)dz\in
\Omega^{\dag,1}({\bf P}^1
\setminus\{\overline{\infty},\overline{\xi_r^{-1}}\})$,
we have
\begin{equation*}
\ell^{\equiv (\alpha_1,\dots,\alpha_{r-1}),(p),\star}_{n_1,\dots,n_{r-1}}(\xi_1,\dots,\xi_{r-2},\xi_{r-1};\xi_r z) \cdot \beta(z)dz
\in\Omega^{\dag,1}( {\bf P}^1 \setminus S_\infty
).
\end{equation*}
Let 
$$
f(z):=\ell^{\equiv (\alpha_1,\dots,\alpha_{r-1}),(p),\star}_{n_1,\dots,n_{r-1}}(\xi_1,\dots,\xi_{r-2},\xi_{r-1};\xi_r z) \cdot \beta(z)
\in
A^{\dag}( {\bf P}^1 \setminus S_\infty
).
$$
Then it follows from the same arguments as 
in (ii) that
$$\ell^{\equiv (\alpha_1,\dots,\alpha_r),(p),\star}_{n_1,\dots,n_r}(\xi_1,\dots,\xi_{r}; z)\in
 A^\dag( {\bf P}^1\setminus S
).$$
\end{proof}


We saw, in Proposition \ref{rigidness}, that
$\ell^{(p),\star}_{n_1,\dots,n_r}(\xi_1,\dots,\xi_{r}; z)$
is a rigid analytic function: 
$
\ell^{(p),\star}_{n_1,\dots,n_r}(\xi_1,\dots,\xi_{r}; z)
\in A^{\mathrm{rig}}( {\bf P}^1\setminus S
).
$
But, in fact, 
we can say more:

\begin{theorem}\label{rigidness III}
Let  $n_1,\dots,n_r\in \mathbb{N}$, and 
$\xi_1,\dots,\xi_{r}\in\mathbb{C}_p$ such that 
$|\xi_j|_p= 1$ ($1\leqslant j\leqslant r$).
Take 
$S$ as in \eqref{S0}.
Then the function 
$\ell^{(p),\star}_{n_1,\dots,n_r}(\xi_1,\dots,\xi_{r}; z)$
is an overconvergent function on 
${\bf P}^1\setminus S
$:
$$\ell^{(p),\star}_{n_1,\dots,n_r}(\xi_1,\dots,\xi_{r}; z)\in
 A^\dag( {\bf P}^1\setminus S
).$$

%
\end{theorem}

\begin{proof}
By 
Theorem \ref{rigidness II}, 
we have
$$\ell^{\equiv (\alpha_1,\dots,\alpha_r),(p),\star}_{n_1,\dots,n_r}(\xi_1,\dots,\xi_{r}; z)\in
 A^\dag( {\bf P}^1\setminus S
).$$
Then by Remark \ref{easy remark} (ii), we have
$\ell^{(p),\star}_{n_1,\dots,n_r}(\xi_1,\dots,\xi_{r};z)\in A^\dag( {\bf P}^1\setminus S)$, since 
$A^\dag( {\bf P}^1\setminus S)$ forms an algebra.
\end{proof}

\begin{example}
In particular, when $r=1$, 
we have
$$\ell^{(p),\star}_n(1;z)\in A^\dag( {\bf P}^1\setminus \{\bar 1\}).$$
In fact, in \cite[Proposition 6.2]{C}, Coleman showed that
$$
\ell^{(p),\star}_n(1;z)\in A^{\mathrm{rig}}(\widetilde X),
$$
with
$\widetilde X=\left\{z\in {\bf P}^1({\mathbb C}_p)\bigm| |z-1|_p >p^{\frac{-1}{p-1}}\right\}$.
\end{example}

\begin{remark}
In \cite[Definition 2.13]{F2}, the first-named author introduced the overconvergent
$p$-adic MPL $Li^\dag_{n_1,\dots,n_r}(z)$.
The relationship between 
$Li^\dag_{n_1,\dots,n_r}(z)$ and 
$\ell^{(p),\star}_{n_1,\dots,n_r}(\xi_1,\dots,\xi_{r};z)$ 
is still under investigation.
\end{remark}

\subsection{$p$-adic twisted multiple star polylogarithms}\label{sec-5-5}
In \cite{Fu1, Y}, 
$p$-adic TMSPLs (Definition \ref{Def-TMSPL}) were introduced 
as 
multiple generalizations 
of Coleman's $p$-adic polylogarithms \cite{C}, and the 
$p$-adic multiple $L$-values are defined 
as their special values at $1$. 
The aim of 
this subsection is to clarify the 
relationship between $p$-adic rigid TMSPLs
(Definition \ref{def of pMMPL})
and $p$-adic TMSPLs  in Theorem \ref{ell-Li theorem}, 
and to establish the 
relationship between 
special values at positive integers of our $p$-adic multiple $L$-functions
(see Definition \ref{Def-pMLF})
and the $p$-adic twisted multiple $L$-star values,
the special values of $p$-adic TMSPLs at $1$,
in Theorem \ref{L-Li theorem}.
Theorem \ref{ell-Li theorem} extends the previous result
\eqref{Coleman L-Li} of Coleman in the depth $1$ case, shown in \cite{C}.

First we recall Coleman's $p$-adic iterated integration theory \cite{C}
for 
our particular case.
For other nice expositions of his theory,
see \cite[Section 5]{B}, \cite[Subsection 2.2.1]{Br} and \cite[Subsection 2.1]{Fu1}.
The integration here is different from the integration using a certain $p$-adic measure
which is explained in Subsection \ref{sec-3-1}.

\begin{notation}
Fix $\varpi\in\mathbb C_p$.
The {\it $p$-adic logarithm} $\log^\varpi$ {\it associated to 
the branch} $\varpi$ is defined to be 
the locally 
analytic group homomorphism
$\mathbb C_p^\times\to\mathbb C_p$
with the usual Taylor expansion  for the logarithm on $]\bar 1[=1+{\frak M}_{\mathbb C_p}$.
It is uniquely characterized by $\log^\varpi(p)=\varpi$
because ${\mathbb C}_p^\times\simeq 
p^{\mathbb Q}\times \mu_{\infty}\times (1+{\frak M}_{\mathbb C_p})$.
We call this $\varpi\in\mathbb C_p$ the {\it branch parameter}
of the $p$-adic logarithm. 

Let
$S=\{s_0,\dots,s_d\}$ (all $s_i$ are distinct)
be a finite subset of $\textbf{P}^{1}(\overline{\mathbb{F}}_p)$.
Define
\[
A^\varpi_{\mathrm{loc}}({\bf P}^1\setminus S)
:=\prod_{x\in {\bf P}^1({\overline{\mathbb{F}}_p}) }A^\varpi_{\mathrm{log}}(U(x)),\qquad
\Omega^\varpi_{\mathrm{loc}}({\bf P}^1\setminus S)
:=\prod_{x\in  {\bf P}^1({\overline{\mathbb{F}}_p}) }\Omega^\varpi_{\mathrm{log}}(U(x)),
\]
where
\begin{align*}
&A^\varpi_{\mathrm{log}}(U(x)):=
\begin{cases}
A^\mathrm{rig}(]x[) &  (\text{if } x\not\in S), \\
\underset{\lambda\to 1^-}{\lim}
A^\mathrm{rig}\Bigl(]x[ \ \cap U_\lambda\Bigr)\Bigl[\log^\varpi(z_i)\Bigr] &
(\text{if }x=s_i \quad (0\leqslant i \leqslant d)),
\end{cases}\\
&\Omega^\varpi_{\mathrm{log}}(U(x)):=
\begin{cases}
A^{\mathrm{rig}}(]x[)\cdot dz_x
&  (\text{if }x\not\in S), \\
A^\varpi_{\mathrm{log}}(U(x))dz_i &
(\text{if }x=s_i \quad (0\leqslant i \leqslant d)).
\end{cases}\\
\end{align*}
Here, $U_\lambda$ is the same affinoid as in \eqref{removing all closed discs},
$z_i$ is a local parameter as in  \eqref{local parameter}
and $z_x$ is a local parameter of $]x[$. 
We note that $\log^\varpi(z_i)$ is a locally analytic function defined on
$]s_i[-z_i^{-1}(0)$, 
whose differential is $\frac{dz_i}{z_i}$, and it is transcendental over 
$A^\mathrm{rig}\Bigl(]s_i[ \ \cap U_\lambda\Bigr)$, 
and
$$
\underset{\lambda\to 1^-}{\lim}
A^\mathrm{rig}\Bigl(]s_i[ \ \cap U_\lambda\Bigr)
\cong\Bigl\{f(z_i)=\sum\limits_{n=-\infty}^{\infty}a_nz_i^n  \ (a_n\in\mathbb C_p)
\text{ converging on }\lambda<|z_i|_p<1 \text{ for some } 0\leqslant \lambda<1\Bigr\}
$$
(see \cite{Be1}).
We remark that these definitions of $A^\varpi_{\mathrm{log}}(U(x))$ and $\Omega^\varpi_{\mathrm{log}}(U(x))$
are independent of any choice of local parameters 
modulo standard isomorphisms.
By taking a component-wise derivative, we obtain a $\mathbb C_p$-linear map
$d:A^\varpi_{\mathrm{loc}}({\bf P}^1\setminus S)\to
\Omega^\varpi_{\mathrm{loc}}({\bf P}^1\setminus S)$.
We may regard $A^\dag({\bf P}^1\setminus S)$
and $\Omega^{\dag,1}({\bf P}^1\setminus S)$
in Notation \ref{overconvergent functions and associated cohomologies}
to be 
subspaces of $A^\varpi_{\mathrm{loc}}({\bf P}^1\setminus S)$ 
and $\Omega^\varpi_{\mathrm{loc}}({\bf P}^1\setminus S)$, 
respectively.

By a result 
of Coleman \cite{C},  we have an 
$A^\dag({\bf P}^1\setminus S)$-subalgebra 
$$
A^\varpi_{\mathrm{Col}}({\bf P}^1\setminus S)$$
of $A^\varpi_{\mathrm{loc}}({\bf P}^1\setminus S)$,
which we call {\it the ring of Coleman functions} 
attached to a branch parameter $\varpi\in\mathbb C_p$,
and a $\mathbb C_p$-linear map
$$
\int_{(\varpi)}:A^\varpi_{\mathrm{Col}}({\bf P}^1\setminus S)\underset
{A^\dag({\bf P}^1\setminus S)}{\otimes}
\Omega^{\dag,1}({\bf P}^1\setminus S)
\to A^\varpi_{\mathrm{Col}}({\bf P}^1\setminus S)\Bigm/ \mathbb C_p\cdot 1
$$
satisfying 
$d\bigm|_{A^\varpi_{\mathrm{Col}}({\bf P}^1\setminus S)}\circ\int_{(\varpi)}=
id_{A^\varpi_{\mathrm{Col}}({\bf P}^1\setminus S)\otimes
\Omega^{\dag,1}({\bf P}^1\setminus S)}$,
which we call {\it the $p$-adic (Coleman) integration}
attached to a branch parameter $\varpi\in\mathbb C_p$.
We often drop  the subscript ${(\varpi)}$.

Since
$A^\varpi_{\mathrm{Col}}({\bf P}^1\setminus S)$ is a subspace
of $A^\varpi_{\mathrm{loc}}({\bf P}^1\setminus S)$,
each element $f$ of
$A^\varpi_{\mathrm{Col}}({\bf P}^1\setminus S)$ can be seen as a map
$f:U\to \mathbb C_p$
by  $f|_{U\cap ]x[}\in A^\varpi_{\mathrm{log}}(U_x)$ for each residue $]x[$,
where $U$ is an affinoid bigger than
${\bf P}^1({\mathbb C}_p)- ]S[$.
This is how we regard $f$ as a function.
\end{notation}

Here we recall two important properties of Coleman functions: 
%
%
%

\begin{proposition}[The Branch Independency Principle {\cite[Proposition 2.3]{Fu1}}]
\label{branch independency principle}
For $\varpi_1,\varpi_2\in\mathbb C_p$, 
define the isomorphisms
$$
\iota_{\varpi_1,\varpi_2}:A^{\varpi_1}_{\mathrm{loc}}({\bf P}^1\setminus S)
\overset{\sim}{\to}A^{\varpi_2}_{\mathrm{loc}}({\bf P}^1\setminus S) \quad \text{ and } \quad
\tau_{\varpi_1,\varpi_2}:\Omega^{\varpi_1}_{\mathrm{loc}}({\bf P}^1\setminus S)\overset{\sim}{\to}\Omega^{\varpi_2}_{\mathrm{loc}}({\bf P}^1\setminus S)
$$
which are 
obtained
by replacing each $\log^{\varpi_1}(z_{s_i})$ by $\log^{\varpi_2}(z_{s_i})$
for $1\leqslant i\leqslant d$. Then
$$\iota_{\varpi_1,\varpi_2}(A^{\varpi_1}_{\mathrm{Col}}({\bf P}^1\setminus S))=A^{\varpi_2}_{\mathrm{Col}}({\bf P}^1\setminus S),$$
$$\tau_{\varpi_1,\varpi_2}(A^{\varpi_1}_{\mathrm{Col}}({\bf P}^1\setminus S)\otimes\Omega^{\dag,1}({\bf P}^1\setminus S))
=A^{\varpi_2}_{\mathrm{Col}}({\bf P}^1\setminus S)\otimes\Omega^{\dag,1}({\bf P}^1\setminus S),$$ 
and
$$\iota_{\varpi_1,\varpi_2}\circ\int_{(\varpi_1)}=\int_{(\varpi_2)}\circ\tau_{\varpi_1,\varpi_2} \mod \mathbb C_p\cdot 1.$$
Namely the following diagram is commutative:
\[
\begin{CD}
A^{\varpi_1}_{\mathrm{Col}}({\bf P}^1\setminus S)\underset{A^\dag({\bf P}^1\setminus S)}{\otimes}\Omega^{\dag,1}({\bf P}^1\setminus S)
@>{\tau_{\varpi_1,\varpi_2}}>>
A^{\varpi_2}_{\mathrm{Col}}({\bf P}^1\setminus S)\underset{A^\dag({\bf P}^1\setminus S)}{\otimes}\Omega^{\dag,1}({\bf P}^1\setminus S) \\
@V{\int_{(\varpi_1)}}VV
@VV{\int_{(\varpi_2)}}V \\
A^{\varpi_1}_{\mathrm{Col}}({\bf P}^1\setminus S)\Bigm/ \mathbb C_p\cdot 1
@>{\iota_{\varpi_1,\varpi_2}}>>
A^{\varpi_2}_{\mathrm{Col}}({\bf P}^1\setminus S)\Bigm/ \mathbb C_p\cdot 1. \\
\end{CD}
\]
\end{proposition}
%
%

\begin{proposition}[The Coincidence Principle {\cite[Chapter IV]{C}}]\label{coincidence principle}
Let $\varpi\in\mathbb C_p$, and let
$f\in A^\varpi_{\mathrm{Col}}({\bf P}^1\setminus S)$.
Suppose that $f|_U\equiv 0$ for 
an admissible open subset $U$ 
of ${\bf P}^1(\mathbb C_p)$, 
that is,
$U$ is of the form 
$
\left\{z\in {\bf P}^1({\mathbb C}_p)\bigm| |z-\alpha_i|_p> \rs_i \ (i=1,\dots, n),
|z|_p< 1/\rs_0
\right\}
$
for some $\alpha_1, \dots, \alpha_n\in {\mathbb C}_p$ and $\rs_1,\dots,\rs_n\in{\mathbb Q}_{>0}$
and $\rs_0\in{\mathbb Q}_{\geqslant 0}$ (cf.\ \cite{B,C}).
Then $f=0$.
\end{proposition}

It follows from this proposition that a locally constant 
Coleman function is globally constant.

%

\begin{notation}\label{Coleman-notation}
Fix a branch $\varpi\in\mathbb C_p$ and let
$\omega\in A^\varpi_{\mathrm{Col}}({\bf P}^1\setminus S)\underset{A^\dag({\bf P}^1\setminus S)}{\otimes}\Omega^{\dag,1}({\bf P}^1\setminus S)$.
Then by Coleman's integration theory, there exists
a unique (modulo constant) Coleman function 
$F_\omega\in A^\varpi_{\mathrm{Col}}({\bf P}^1\setminus S)$ such that
$\int\omega\equiv F_\omega$ (modulo constant).
For $x,y\in]{\bf P}^1\setminus S[$, 
we define 
$$\int_x^y\omega:=
F_\omega(y)-F_\omega(x).$$
It is clear that $\int_x^y\omega$ does not depend on the 
choice
of $F_\omega$ (although it may depend on the choice of a branch $\varpi\in\mathbb C_p$).
If both $F_\omega(x)$ and $F_\omega(y)$ is well-defined 
when $x$ or $y$ belongs to  $]S_0[$,
we also denote $F_\omega(y)-F_\omega(x)$ by $\int_x^y\omega$.
By letting $x$ be 
fixed and $y$ vary, we may 
regard $\int_x^y\omega$ as the Coleman function
which is characterized by $dF_\omega=\omega$ and $F_\omega(x)=0$.
We note that
\begin{equation}
\iota_{\varpi_1,\varpi_2}\left(\int_{(\varpi_1),x}^y\omega\right)
=\int_{(\varpi_2),x}^y\tau_{\varpi_1,\varpi_2} (\omega).
\end{equation}
\end{notation}

We will apply Coleman's theory to the function defined below,
which is the main object in this subsection.

\begin{definition}\label{Def-TMSPL}
Let $n_1,\dots,n_r\in \mathbb{N}$, and 
$\xi_1,\dots,\xi_{r}\in\mathbb{C}_p$
such that $|\xi_j|_p\leqslant 1$ ($1\leqslant j\leqslant r$).
The {\bf $p$-adic TMSPL} 
$Li^{(p),\star}_{n_1,\dots,n_r}(\xi_1,\dots,\xi_r; z)$ is defined by the following 
$p$-adic power series:
\begin{equation}\label{series expression}
Li^{(p),\star}_{n_1,\dots,n_r}(\xi_1,\dots,\xi_r ;z):=
{\underset{0<k_1\leqslant\cdots\leqslant k_{r}}{\sum}}
\frac{\xi_1^{k_1}\cdots\xi_{r}^{k_{r}} z^{k_r}}{k_1^{n_1}\cdots k_r^{n_r}},
\end{equation}
which converges for $z\in ]\bar{0}[$. 
\end{definition}

\begin{remark}\label{remark:Coleman-Furusho} \ 
\begin{enumerate}
\item We note that when $r=1$ and $\xi=1$, $Li^{(p),\star}_{n}(1;z)$
is equal to Coleman's $p$-adic polylogarithm
$\ell_n(z)$ in \cite[p.\,195]{C}.
\item
The $p$-adic multiple polylogarithm $Li^{(p)}_{n_1,\dots,n_r}(1,\dots,1;z)$
which is a non-star version of \eqref{series expression}
with $\xi_1=\cdots=\xi_r=1$,
was investigated by the first-named author in \cite{F2}. 
It is easy to see that 
$Li^{(p),\star}_{n_1,\dots,n_r}(1,\dots,1;z)$
is written as a linear combination of such $p$-adic multiple polylogarithms.

\item
Yamashita \cite{Y} treats a non-star version of \eqref{series expression}
with the case when all $\xi_j$ are roots of unity, 
whose orders are prime to $p$.
\end{enumerate}
\end{remark}

The following can be proved by direct calculation.

\begin{lemma}\label{differential equations fo Li}
Let $n_1,\dots,n_r\in \mathbb{N}$, and 
$\xi_1,\dots,\xi_{r}\in\mathbb{C}_p$
such that $|\xi_j|_p\leqslant 1$ ($1\leqslant j\leqslant r$).
\begin{enumerate}[{\rm (i)}]
\item $\displaystyle{
\frac{d}{dz}Li^{(p),\star}_{1}(\xi_1;z)=\frac{\xi_1}{1-\xi_1z}.}$
\item $\displaystyle{
\frac{d}{dz}Li^{(p),\star}_{n_1,\dots,n_r}(\xi_1,\dots,\xi_r ;z)=
\begin{cases}
\frac{1}{z}Li^{(p),\star}_{n_1,\dots,n_{r-1},n_r-1}(\xi_1,\dots,\xi_r;z)
&\text{if}\quad n_r\neq 1 ,\\
\frac{\xi_r}{1-\xi_r z} Li^{(p),\star}_{n_1,\dots,n_{r-1}}(\xi_1,\dots,\xi_{r-2},\xi_{r-1}\xi_r;z)&\text{if}\quad n_r=1. 
\end{cases}}$
\end{enumerate}
\end{lemma}

The  definition, below, is suggested by
the differential equations, above.

\begin{theorem-definition}\label{Coleman function theorem}
Fix a branch of the $p$-adic logarithm by $\varpi\in\mathbb C_p$.
Let $n_1,\dots,n_r\in \mathbb{N}$,
$\xi_1,\dots,\xi_{r}\in\mathbb{C}_p$
such that  $|\xi_j|_p\leqslant 1$ ($1\leqslant j\leqslant r$), and let 
$$
S_r:=\{\bar 0,\overline{\infty}, \overline{(\xi_r)^{-1}},\overline{(\xi_{r-1}\xi_r)^{-1}},\dots,\overline{(\xi_1\cdots\xi_r)^{-1}}\}
\subset  {\bf{P}}^{1}(\overline{\mathbb{F}}_p).
$$
Then we have the Coleman function attached to  $\varpi\in\mathbb C_p$
$$
Li^{(p),\star,\varpi}_{n_1,\dots,n_r}(\xi_1,\dots,\xi_r; z)\in
A^\varpi_{\text{\rm Col}}
({\bf P}^1\setminus S_r
),
$$
which is constructed
by the following iterated integrals:
\begin{align}
&Li^{(p),\star,\varpi}_{1}(\xi_1;z)=-\log^\varpi (1-\xi_1z)=\int_0^z\frac{\xi_1}{1-\xi_1t}dt, 
\label{intLi1}
\\
&Li^{(p),\star,\varpi}_{n_1,\dots,n_r}(\xi_1,\dots,\xi_r; z)=
\begin{cases}
\int_0^zLi^{(p),\star,\varpi}_{n_1,\dots,n_{r-1},n_r-1}(\xi_1,\dots,\xi_r; t)\frac{dt}{t} &\text{if}\quad n_r\neq 1 ,\\
\int_0^zLi^{(p),\star,\varpi}_{n_1,\dots,n_{r-1}}(\xi_1,\dots,\xi_{r-2},\xi_{r-1}\xi_r; t)\frac{\xi_r dt}{1-\xi_r t} &\text{if}\quad n_r=1. 
\end{cases}
\label{intLi2} 
\end{align}
\end{theorem-definition}

\begin{proof}
It is achieved by the induction on the weight $w:=n_1+\cdots+ n_r$.

When $w=1$,  the integration starting from $0$ is well-defined 
because the differential form $\frac{\xi_1}{1-\xi_1t}dt$  has no pole at $t=0$.
Hence, we have \eqref{intLi1}.

When $w>1$ and $n_r>1$, 
the induction assumption implies that 
$Li^{(p),\star,\varpi}_{n_1,\dots,n_r-1}(\xi_1,\dots,\xi_r; 0)$ 
is equal to $0$ because it is an integration from $0$ to $0$.
So $Li^{(p),\star,\varpi}_{n_1,\dots,n_r-1}(\xi_1,\dots,\xi_r; t)$
has a zero of order at least $1$, 
so the integrand of the right-hand side of \eqref{intLi2} has no pole at $t=0$.
The integration \eqref{intLi2} starting from $0$ is well-defined. 
The case when $n_r=1$ can be proved in a similar 
(or an easier) way.
%
\end{proof}

\begin{remark}\label{S_r-0}
It is easy to see that
$Li^{(p),\star,\varpi}_{n_1,\dots,n_r}(\xi_1,\dots,\xi_r ;z)
=Li^{(p),\star}_{n_1,\dots,n_r}(\xi_1,\dots,\xi_r; z)$
when $|z|_p<1$ for all branches $\varpi\in\mathbb C_p$.
Thus, the Coleman function 
$Li^{(p),\star,\varpi}_{n_1,\dots,n_r}(\xi_1,\dots,\xi_r; z)
\in A^\varpi_{\text{Col}}
({\bf P}^1\setminus S_r
)$
is defined on an affinoid bigger than
${\bf P}^1({\mathbb C}_p) - ]{S_r}\setminus\{\overline{0}\} [.$
\end{remark}

\begin{proposition}\label{branch independent region}
Fix a branch  $\varpi\in\mathbb C_p$.
Let $n_1,\dots,n_r\in \mathbb{N}$, and 
$\xi_1,\dots,\xi_{r}\in\mathbb{C}_p$
such that $|\xi_j|_p\leqslant 1$ ($1\leqslant j\leqslant r$).
Then the restriction of the $p$-adic TMSPL
$Li^{(p),\star,\varpi}_{n_1,\dots,n_r}(\xi_1,\dots,\xi_r ;z)$
to
${\bf P}^1({\mathbb C}_p) - ]{S_r}\setminus\{\overline{0}\} [$
does not depend on 
$\varpi$. 
\end{proposition}

\begin{proof}
It is achieved by the induction on the weight $w:=n_1+\cdots+ n_r$.
Take $z\in {\bf P}^1({\mathbb C}_p) - ]{S_r}\setminus\{\overline{0}\} [$.

When $w=1$,  it is easy to see that
$Li^{(p),\star,\varpi}_{1}(\xi_1;z)=-\log^\varpi (1-\xi_1z)$ is 
independent of the choice of  the branch $\varpi\in\mathbb{C}_p$.

Consider the case when $w>1$ and $n_r>1$.
Due to the existence of a pole of $\frac{dt}{t}$ at $t=0$,
the integration \eqref{intLi2} might have a ``log-term'' on  its restriction to $]\bar 0[$.
However, it is easy to see that the restriction of 
$Li^{(p),\star,\varpi}_{n_1,\dots,n_r}(\xi_1,\dots,\xi_r; t)$
to $]\bar 0[$ has no log-term, 
because it is given by the series expansion \eqref{series expression}
by Remark \ref{S_r-0}.
By the induction assumption,
the restriction of
$Li^{(p),\star,\varpi}_{n_1,\dots,n_r-1}(\xi_1,\dots,\xi_r; t)$
to ${\bf P}^1({\mathbb C}_p) - ]{S_r}\setminus\{\overline{0}\} [$
is independent of any choice of the branch $\varpi\in\mathbb{C}_p$.
Therefore, $Li^{(p),\star,\varpi}_{n_1,\dots,n_r}(\xi_1,\dots,\xi_r; t)$
has no log-term and does not depend on any choice of the branch.

The above proof also works for the case when $w>1$ and $n_r=1$.
The branch independency then follows since
two poles of $\frac{\xi_r dt}{1-\xi_r t} $ are outside of the region
${\bf P}^1({\mathbb C}_p) - ]{S_r}\setminus\{\overline{0}\} [$.
\end{proof}

The $p$-adic rigid TMSPL
$\ell^{(p),\star}_{n_1,\dots,n_r}(\xi_1,\dots,\xi_{r}; z)$
in Subsection \ref{sec-5-3}
is described by the $p$-adic TMSPL
$Li^{(p),\star,\varpi}_{n_1,\dots,n_r}(\xi_1,\dots,\xi_r ;z)$ as follows:

\begin{theorem}\label{ell-Li theorem}
Fix a branch  $\varpi\in\mathbb C_p$.
Let $n_1,\dots,n_r\in \mathbb{N}$, and 
$\xi_1,\dots,\xi_{r}\in\mathbb{C}_p$
such that $|\xi_j|_p= 1$ ($1\leqslant j\leqslant r$).
The equality
\begin{align}
\ell&^{(p),\star}_{n_1,\dots,n_r}(\xi_1,\dots,\xi_r ;z)  \notag \\
&=\frac{1}{p^r}\sum_{0<\alpha_1,\dots,\alpha_r<p}\sum_{\rho_1^p=\cdots=\rho_r^p=1}
\rho_1^{-\alpha_1}\cdots\rho_r^{-\alpha_r}
Li^{(p),\star,\varpi}_{n_1,\dots,n_r}(\rho_1\xi_1,\dots,\rho_r\xi_r ;z) 
\label{ell-Li formula} 
\end{align}
holds for $z\in{\bf P}^1({\mathbb C}_p) - ]{S_r}\setminus\{\overline{0}\} [.$
\end{theorem}

\begin{proof}
By the power series expansion
\eqref{series expression for partial double}
and \eqref{series expression}, and Remark \ref{easy remark} (ii),
it is easy to see that the equality holds on $]\bar 0[$.
By Theorem \ref{rigidness III},
the left-hand side belongs to 
$A^\dag({\bf P}^1
\setminus{S_r})$
($\subset A^\varpi_{\text{Col}}({\bf P}^1 \setminus{S_r})$), 
and by Theorem-Definition \ref{Coleman function theorem},
the right-hand side  belongs to $A^\varpi_{\text{Col}}({\bf P}^1  \setminus{S_r})$.
Therefore, 
by the coincidence principle (Proposition \ref{coincidence principle}),
the equality actually holds on 
${\bf P}^1({\mathbb C}_p) - ]{S_r}\setminus\{\overline{0}\} [$,
on an affinoid bigger than the space.
\end{proof}

The following is a reformulation of
the equation in Theorem \ref{ell-Li theorem}.

\begin{theorem}\label{ell-Li reformulation prop}
Fix a branch  $\varpi\in\mathbb C_p$.
Let $n_1,\dots,n_r\in \mathbb{N}$, and 
$\xi_1,\dots,\xi_{r}\in\mathbb{C}_p$
such that $|\xi_j|_p= 1$ ($1\leqslant j\leqslant r$).
The equality
\begin{align}
&\ell^{(p),\star}_{n_1,\dots,n_r}(\xi_1,\dots,\xi_r; z) \notag \\
&=Li^{(p),\star,\varpi}_{n_1,\dots,n_r}(\xi_1,\dots, \xi_r; z) 
+\sum_{d=1}^r\left(-\frac{1}{p}\right)^d
\sum_{1\leqslant i_1<\cdots<i_d\leqslant r}\sum_{\rho_{i_1}^p=1}\cdots\sum_{\rho_{i_d}^p=1}
Li^{(p),\star,\varpi}_{n_1,\dots,n_r}\Bigl(\bigl((\prod_{l=1}^d\rho_{i_l}^{\delta_{i_lj}})\xi_j\bigr); z\Bigr)
\label{ell-Li reformulation} 
\end{align}
holds for $z\in{\bf P}^1({\mathbb C}_p) - ]{S_r}\setminus\{\overline{0}\} [$, 
where $\delta_{ij}$ is the Kronecker delta. 
\end{theorem}

\begin{proof}
It 
follows by direct calculation:
\begin{align*}
\ell^{(p),\star}_{n_1,\dots,n_r}&(\xi_1,\dots,\xi_{r}; z)
=
\underset{(k_1,p)=\cdots=(k_r,p)=1}{
{\underset{0<k_1\leqslant\cdots\leqslant k_{r}}{\sum}}}
\frac{\xi_1^{k_1}\cdots\xi_{r}^{k_{r}} z^{k_r}}{k_1^{n_1}\cdots k_r^{n_r}} 
=
{\underset{0<k_1\leqslant\cdots\leqslant k_{r}}{\sum}}
\left\{\prod_{i=1}^r
\left(1-\frac{1}{p}\sum_{\rho_i^p=1}\rho_i^{k_i}\right)
\frac{\xi_i^{k_i}}{k_i^{n_i}} \right\} z^{k_r}\\
=&Li^{(p),\star}_{n_1,\dots,n_r}(\xi_1,\dots,\xi_r; z)
+\sum_{d=1}^r\left(-\frac{1}{p}\right)^d
\sum_{1\leqslant i_1<\cdots<i_d\leqslant r}\sum_{\rho_{i_1}^p=1}\cdots\sum_{\rho_{i_d}^p=1}
Li^{(p),\star}_{n_1,\dots,n_r}\Bigl(\bigl((\prod_{l=1}^d\rho_{i_l}^{\delta_{i_lj}})\xi_j\bigr); z\Bigr).
\end{align*}
\end{proof}

\begin{example}
In the case when $r=1$ and $\xi_1=1$, \eqref{ell-Li reformulation} gives
$$
\ell_n^{(p),\star}(1;z)=Li_n^{(p),\star}(1;z)-\frac{1}{p}\sum_{\rho^p=1}Li_n^{(p),\star}(\rho;z).
$$
Combining this with
the distribution relation (cf. \cite{C})
\begin{equation}\label{distribution relation}
\frac{1}{r}\sum_{\eta\in\mu_r}Li^{(p),\star}_n(\eta; z)=
\frac{1}{r^n}Li_n^{(p),\star}(1;z^r)
\end{equation}
for $r\geqslant 1$,
we recover Coleman's formula 
$$
\ell_n^{(p),\star}(1;z)=Li_n^{(p),\star}(1;z)-\frac{1}{p^n}Li_n^{(p),\star}(1;z^p),
$$
shown in \cite[Section VI]{C}.
\end{example}

We define the $p$-adic twisted multiple $L$-star values as 
the special values of the $p$-adic TMSPLs at unity, 
under a certain condition, below.

\begin{theorem-definition}
Fix a branch  $\varpi\in\mathbb C_p$.
Let $n_1,\dots,n_r\in \mathbb{N}$,
$\rho_1,\dots,\rho_r\in\mu_p$ and
$\xi_1,\dots,\xi_r\in\mu_c$ with $(c,p)=1$.
Assume that 
\begin{equation}\label{assumption for pMLV}
\xi_1\cdots\xi_r\neq 1, \quad \xi_2\cdots\xi_r\neq 1,\quad
\dots, \quad  \xi_{r-1}\xi_r\neq 1,\quad \xi_r\neq 1.
\end{equation}
Then the special value of
\begin{equation}\label{function for pMLV}
Li^{(p),\star,\varpi}_{n_1,\dots,n_r}(\rho_1\xi_1,\dots,\rho_r\xi_r; z)
\end{equation}
at $z=1$
is well-defined.
Furthermore, it is independent of the 
choice of the branch $\varpi\in\mathbb C_p$, and
it belongs to ${\mathbb Q}_p(\mu_{cp})$.
We call this value 
{\bf the $p$-adic twisted multiple $L$-star value}, 
and denote it by
\begin{equation}\label{FYpMLV}
Li^{(p),\star}_{n_1,\dots,n_r}(\rho_1\xi_1,\dots,\rho_r\xi_r).
\end{equation}
\end{theorem-definition}

\begin{proof}
By the assumption \eqref{assumption for pMLV},
$$1\in{\bf P}^1({\mathbb C}_p) - ]{S_r}\setminus\{\overline{0}\} [.$$
Hence the special value of \eqref{function for pMLV} at $z=1$ is well-defined 
by Remark \ref{S_r-0}.

The branch independency of the special value follows from Proposition \ref{branch independent region}.

The points $0$ and $1$, and all of the differential 1-forms
$\frac{\xi_r dt}{1-\xi_rt}$, 
$\frac{(\xi_{r-1}\xi_r) dt}{1-(\xi_{r-1}\xi_r)t}$, $\dots$,
$\frac{(\xi_1\cdots\xi_r) dt}{1-(\xi_1\cdots\xi_r)t}$
are defined over ${\mathbb{Q}}_p(\mu_{cp})$.
Then by the Galois equivalency stated in \cite[Remark 2.3]{BdJ},
the special value of \eqref{function for pMLV}
for $z\in {\mathbb{Q}}_p(\mu_{cp})$ is
invariant under the action of the absolute Galois group 
${\rm Gal}(\overline{\mathbb{Q}}_p/{\mathbb{Q}}_p(\mu_{cp}))$
if we take $\varpi\in {\mathbb{Q}}_p(\mu_{cp})$.
 Since  we have shown that the special values at $z=1$ are independent of 
the choice of $\varpi$,
the value \eqref{FYpMLV} must belong to ${\mathbb Q}_p(\mu_{cp})$.
\end{proof}

\begin{remark} \label{Rem-FY}\ 
\begin{enumerate}
\item
The $p$-adic multiple zeta values were introduced as
the special values at $z=1$ of $p$-adic MPL
$Li^{(p),\varpi}_{n_1,\dots,n_r}(1,\dots,1;z)$,
the non-star version of the function \eqref{function for pMLV} with all $\rho_i=1$  and $\xi_i=1$ (thus it is not the case of \eqref{assumption for pMLV})
and their basic properties are investigated  by the first-named author in \cite{Fu1}.
\item
The $p$-adic multiple $L$-values introduced by Yamashita \cite{Y} are special values
at $z=1$ of the non-star version of the function
\eqref{function for pMLV} with all $\rho_i=1$ and
$(\xi_r,n_r)\neq (1,1)$.
Unver's \cite{U} cyclotomic $p$-adic multi-zeta values might be closely related
to Yamashita's 
values.
\end{enumerate}
\end{remark}

By Theorem \ref{L-ell theorem} and Theorem \ref{ell-Li theorem}, we have the following theorem.
A very nice point, here, is that our assumption \eqref{assumption for pMLV}
appeared as the  condition of the summation in equation \eqref{L-ell-formula}. 

\begin{theorem}\label{L-Li theorem}
For  $n_1,\dots,n_r\in \mathbb{N}$, 
and $c\in \mathbb{N}_{>1}$ with $(c,p)=1$,
\begin{align*}
L_{p,r} & (n_1,\dots,n_r;\omega^{-n_1},\dots,\omega^{-n_r};1,\dots,1;c)= \notag \\
&\frac{1}{p^r}\sum_{0<\alpha_1,\dots,\alpha_r<p}\sum_{\rho_1^p=\cdots=\rho_r^p=1}
\underset{\xi_1\cdots\xi_r\neq 1,\dots,\xi_{r-1}\xi_r\neq 1,\xi_r\neq 1}{\sum_{\xi_1^c=\cdots=\xi_r^c=1}}
\rho_1^{-\alpha_1}\cdots\rho_r^{-\alpha_r} \cdot 
Li^{(p),\star}_{n_1,\dots,n_r}(\rho_1\xi_1,\dots,\rho_r\xi_r).
\end{align*}
%
\end{theorem}

The above theorem can be 
reformulated as the following,
which is comparable to Theorem \ref{T-main-1}.

\begin{theorem}\label{L-Li theorem-2}
For  $n_1,\dots,n_r\in \mathbb{N}$
and $c\in \mathbb{N}_{>1}$ with $(c,p)=1$,
\begin{align}
L_{p,r} & (n_1,\dots,n_r;\omega^{-n_1},\dots,\omega^{-n_r};1,\dots,1;c) \notag \\
&=
\underset{\xi_1\neq 1}{\sum_{\xi_1^c=1}} \cdots
\underset{\xi_r\neq 1}{\sum_{\xi_r^c=1}} 
Li^{(p),\star}_{n_1,\dots,n_r}\left(\frac{\xi_1}{\xi_2},\frac{\xi_2}{\xi_3},\dots,\frac{\xi_{r}}{\xi_{r+1}}\right)   \notag \\
&+\sum_{d=1}^r\left(-\frac{1}{p}\right)^d
\sum_{1\leqslant i_1<\cdots<i_d\leqslant r}\sum_{\rho_{i_1}^p=1}\cdots\sum_{\rho_{i_d}^p=1}
\underset{\xi_1\neq 1}{\sum_{\xi_1^c=1}} \cdots
\underset{\xi_r\neq 1}{\sum_{\xi_r^c=1}} 
Li^{(p),\star}_{n_1,\dots,n_r}\Bigl(\bigl(\frac{\prod_{l=1}^d\rho_{i_l}^{\delta_{i_lj}}\xi_j}{\xi_{j+1}}\bigr)\Bigr),
\end{align}
where we set 
$\xi_{r+1}=1$.
\end{theorem}

\begin{proof}
This follows from 
Theorem \ref{L-ell theorem} and Theorem \ref{ell-Li reformulation prop}.
\end{proof}

Thus the positive integer values  of the $p$-adic multiple $L$-function
are described as linear combinations of 
the special values of the $p$-adic TMSPL \eqref{function for pMLV}
at roots of unity.
This might be regarded as a $p$-adic analogue of the equality:
\begin{equation}\label{zeta=Li}
\zeta_r(n_1,\dots,n_r)=Li_{n_1,\dots,n_r}(1,\dots,1)
\end{equation}
for $n_1,\dots,n_r\in {\mathbb N}$ with $n_r>1$.

As  a special case of Theorem \ref{L-Li theorem}, 
we have the following, 
when $r=1$: 

\begin{example}\label{Coleman's equation}
For $n\in \mathbb{N}$
and $c\in \mathbb{N}_{>1}$ with $(c,p)=1$,
we have
\begin{equation*}
(c^{1-n}-1)\cdot L_{p}(n;\omega^{1-n})
=\frac{1}{p}\sum_{0<\alpha<p}\sum_{\rho^p=1}
\underset{\xi\neq 1}{\sum_{\xi^c=1}}
\rho^{-\alpha} Li^{(p),\star}_{n}(\rho\xi)
\end{equation*}
by Example \ref{example for r=1}.
This formula and  \eqref{distribution relation}
recover Coleman's equation \cite[(4)]{C}
\begin{equation}\label{Coleman L-Li}
L_p(n;\omega^{1-n})=(1-\frac{1}{p^n})Li_n^{(p),\star}(1).
\end{equation}
\end{example}

When $r=2$, we have:
\begin{example}
For  $a,b\in \mathbb{N}$
and $c\in \mathbb{N}_{>1}$ with $(c,p)=1$,
\begin{align*} 
& L_{p,2}(a,b;\omega^{-a},\omega^{-b}; 1,1; c) \\
&=\frac{1}{p^2}\sum_{0<\alpha,  \beta <p} \ 
\sum_{\rho_1,\rho_2\in\mu_p} \ 
\underset{\xi_1\xi_2\neq 1,\xi_2\neq 1}{\sum_{\xi_1,\xi_2\in\mu_c}}
\rho_1^{-\alpha}\rho_2^{-\beta}Li^{(p),\star}_{a,b}(\rho_1\xi_1,\rho_2\xi_2) \\
&=
\underset{\xi_1\neq 1}{\sum_{\xi_1^c=1}} \ \underset{\xi_2\neq 1}{\sum_{\xi_2^c=1}} 
Li^{(p)}_{a,b}\left(\frac{\xi_1}{\xi_2},\xi_2 \right) 
-\frac{1}{p}\sum_{\rho^p=1} \ 
\underset{\xi_1\neq 1}{\sum_{\xi_1^c=1}} \  \underset{\xi_2\neq 1}{\sum_{\xi_2^c=1}} 
\left\{Li^{(p),\star}_{a,b}\left(\frac{\rho\xi_1}{\xi_2},\xi_2
\right) +
Li^{(p),\star}_{a,b}\left(\frac{\xi_1}{\xi_2},{\rho\xi_2}\right)\right\}
\\
&\qquad\qquad +\frac{1}{p^2}\sum_{\rho_{1}^p=1} \ \sum_{\rho_{2}^p=1} \
\underset{\xi_1\neq 1}{\sum_{\xi_1^c=1}} \ \underset{\xi_2\neq 1}{\sum_{\xi_2^c=1}} 
Li^{(p),\star}_{a,b}\left(\frac{\rho_1\xi_1}{\xi_2},{\rho_2\xi_2}\right).
\end{align*}
\end{example}

\ 

In our subsequent 
paper \cite{FKMT02}, it is shown that Theorem \ref{L-Li theorem-2}
can be 
generalized 
for indices $(n_j)$ 
that are not necessarily all positive integers.

\ 

\


\bibliographystyle{amsplain}

\ 

\begin{flushleft}
\begin{small}
H. Furusho\\
Graduate School of Mathematics \\
Nagoya University\\
Furo-cho, Chikusa-ku \\
Nagoya 464-8602, Japan\\
{furusho@math.nagoya-u.ac.jp}

\ 

Y. Komori\\
Department of Mathematics \\
Rikkyo University \\
Nishi-Ikebukuro, Toshima-ku\\
Tokyo 171-8501, Japan\\
komori@rikkyo.ac.jp

\ 

K. Matsumoto\\
Graduate School of Mathematics \\
Nagoya University\\
Furo-cho, Chikusa-ku \\
Nagoya 464-8602, Japan\\
kohjimat@math.nagoya-u.ac.jp

\ 

H. Tsumura\\
Department of Mathematics and Information Sciences\\
Tokyo Metropolitan University\\
1-1, Minami-Ohsawa, Hachioji \\
Tokyo 192-0397, Japan\\
tsumura@tmu.ac.jp

\end{small}
\end{flushleft}

\end{document}